\documentclass[reqno,11pt,letterpaper]{amsart}
\usepackage{amsmath,amssymb,amsthm,graphicx} 
\usepackage[usenames,dvipsnames]{color}
\usepackage{verbatim}

\def\?[#1]{\textbf{[#1]}\marginpar{\Large{\textbf{??}}}}

\newtheorem{prop}{Proposition}[section]
\newtheorem{theo}[prop]{Theorem}

\newtheorem{lem}[prop]{Lemma}
\newtheorem{cor}[prop]{Corollary}
\theoremstyle{definition}
\newtheorem{rem}[prop]{Remark}
\newtheorem*{rem*}{Remark}
\newtheorem{hypothesis}[prop]{Hypothesis}
\newtheorem{defi}[prop]{Definition}

\DeclareMathOperator{\Ric}{Ric}

\DeclareMathOperator{\Tr}{Tr}
\DeclareMathOperator{\Hess}{Hess}

\newcommand{\mc}{\mathcal}
\newcommand{\rr}{\mathbb{R}}
\newcommand{\nn}{\mathbb{N}}
\newcommand{\cc}{\mathbb{C}}
\newcommand{\hh}{\mathbb{H}}
\newcommand{\zz}{\mathbb{Z}}
\newcommand{\cH}{\mathcal{H}}

\newcommand{\cD}{\mathcal{D}}
\newcommand{\cM}{\mathcal{M}}
\newcommand{\cN}{\mathcal{N}}
\newcommand{\cZ}{\mathcal{Z}}

\newcommand{\la}{\lambda}
\newcommand{\eps}{\varepsilon}

\newcommand{\pl}{\partial}
\newcommand{\x}{\times}

\newcommand{\bbar}{\overline}

\newcommand{\cjd}{\rangle}
\newcommand{\cjg}{\langle}

\newcommand{\demi}{\frac{1}{2}}
\newcommand{\ndemi}{\frac{n}{2}}
\newcommand{\tra}{\textrm{Tr}}

\newcommand{\rvol}{{\mathrm{Vol}}_R}
\newcommand{\dvol}{\textrm{dvol}}

\newcommand{\ric}{\textrm{Ric}}
\newcommand{\scal}{\mathrm{Scal}}
\newcommand{\indic}{\operatorname{1\negthinspace l}}

\newcommand{\re}{\mathop{\hbox{\rm Re}}\nolimits}

\newcommand{\cun}{C^\infty}

\newcommand{\oM}{\overline{M}}
\newcommand{\II}{\mathrm{I\hspace{-0.04cm}I}}

\newcommand{\tr}{\mbox{tr}}
\newcommand{\pt}{\partial_t}
\newcommand{\hx}{\hat{x}}
\newcommand{\vol}{\rm Vol}

\newcommand{\cE}{\mathcal E}
\newcommand{\cG}{\mathcal G}

\newcommand{\cS}{\mathcal S}
\newcommand{\cT}{\mathcal T}
\newcommand{\cCP}{\mathcal C\mathcal P}
\newcommand{\rci}{\mathring{R}}

\title[The renormalized volume and uniformisation]{The renormalized volume and uniformisation of conformal structures}
\author{Colin Guillarmou}
\thanks{C.G. was partially supported by ANR 
project ACG ANR-10-BLAN-0105}
\address{Colin Guillarmou, DMA, U.M.R. 8553 CNRS, \'Ecole Normale Superieure, 45 rue d'Ulm,
75230 Paris cedex 05, France}
\email{cguillar@dma.ens.fr}

\author{Sergiu Moroianu}
\thanks{S.M. was partially supported by CNCS through project PN-II-RU-TE-2011-3-0053 and thanks the Fondation des Sciences Math\'ematiques de Paris and the \'Ecole Normale Sup\'erieure for additional supports.}
\address{Sergiu Moroianu, Institutul de Matematic\u{a} al Academiei Rom\^{a}ne\\ P.O. Box 1-764\\ RO-014700 Bucharest\\
Romania}
\email{moroianu@alum.mit.edu}

\author{Jean-Marc Schlenker}
\thanks{J.-M. S. was partially supported by the A.N.R. through projects
ETTT, ANR-09-BLAN-0116-01, and ACG, ANR-10-BLAN-0105.}
\address{Jean-Marc Schlenker, Institut de Math\'ematiques de Toulouse, UMR CNRS 5219 \\
Universit\'e Toulouse III \\
31062 Toulouse cedex 9, France}
\email{schlenker@math.univ-toulouse.fr}

\date{\today}
\begin{document}
\begin{abstract}
We study the renormalized volume of asymptotically hyperbolic Einstein (AHE in short) manifolds $(M,g)$ 
when the conformal boundary $\pl M$ has dimension $n$ even. Its definition depends on the
choice of metric $h_0$ on $\partial M$ in the conformal class at infinity determined by $g$, 
we denote it by ${\rm Vol}_R(M,g;h_0)$. We show that ${\rm Vol}_R(M,g;\cdot)$ is a functional admitting a ``Polyakov type'' formula
in the conformal class $[h_0]$ and we describe the critical points as solutions of some non-linear equation $v_n(h_0)={\rm const}$, 
satisfied in particular by Einstein metrics.
In dimension $n=2$, choosing extremizers in the conformal class amounts to 
uniformizing 
the surface, 
while in dimension $n=4$ this amounts to 
solving 
the $\sigma_2$-Yamabe problem. 
Next, we consider the variation of ${\rm Vol}_R(M,\cdot;\cdot)$ along a curve of AHE metrics $g^t$ 
with boundary metric $h_0^t$ and we use this to show that, provided conformal classes
can be (locally) parametrized by metrics $h$ solving $v_n(h)=\int_{\pl M}v_n(h){\rm dvol}_{h}$, 
the set of ends of AHE manifolds (up to diffeomorphisms isotopic to Identity) 
can be viewed as a Lagrangian submanifold in the cotangent space 
to the space  $\mc{T}(\pl M)$ of conformal structures on $\pl M$.
We obtain as a consequence a higher-dimensional version of McMullen's quasifuchsian reciprocity. We finally show 
that conformal classes admitting negatively curved Einstein metrics are local minima for the renormalized 
volume for a 
warped product 
type filling. 

\end{abstract}
\maketitle


\section{Introduction}

By Mostow rigidity, the volume of complete oriented finite volume hyperbolic $3$-manifolds is an important topological invariant, also related to the Jones polynomial of knots. 
For infinite volume hyperbolic $3$-manifolds, one should still expect some invariant derived from the volume form as well. 
Following ideas coming from the physics literature \cite{HeSk, SkSo, Kr}, Takhtajan-Teo \cite{TaTe} and Krasnov-Schlenker \cite{KrSc} defined a regularized (or renormalized) version of the volume in the case of convex co-compact hyperbolic quotients $M=\Gamma\backslash \hh^3$, and studied some of its properties. 
The renormalized volume is actually related to the uniformization theory of the boundary of the conformal compactification of 
$M$. Indeed, such hyperbolic manifolds can be compactified into smooth manifolds with boundary $\bbar{M}$ by adding a compact 
surface $N$ to $M$, and the metric on $M$ is conformal to a smooth metric $\bar{g}$ on $\bbar{M}$, inducing a conformal
class $[\bar{g}|_{TN}]$ on $N$. The renormalized volume plays the role of an action on the conformal class $[h_0]$ 
with critical points at the constant curvature metrics, in a way similar to the determinant of the Laplacian. It turns out that
this action has interesting properties when we deform the hyperbolic metric in the bulk.

In this paper, we study the higher dimensional analog of this invariant and compute its variation on the quantum conformal superspace, the higher-dimensional analog of the Teichm\"uller space.

\subsection{Dimension $n+1=3$}
Let us explain in more details the definition of the renormalized volume of hyperbolic convex co-compact $3$ manifolds.
Outside a compact set, the hyperbolic metric on $M=\Gamma\backslash \hh^3$ is isometric to a \emph{ hyperbolic end}
\begin{align}\label{productdim3} 
\left( (0,\eps_0)_x\x N , g\right),&& g=\frac{dx^2+ h_0+x^2h_2+\tfrac14  x^4h_2^2}{x^2}
\end{align} 
for some $\eps_0>0$, where $(N,h_0)$ is a compact Riemannian surface (possibly disconnected), and
$h_2$ is a symmetric $2$-tensor on $N$ satisfying the constraints 
\begin{align}\label{constraintseq}
\tra_{h_0}(h_2)=-\tfrac{1}{2} {\rm Scal}_{h_0}, &&
\delta_{h_0}(h_2)=\tfrac12 d\,{\rm Scal}_{h_0}.
\end{align} 
Here $\delta_{h_0}$ 
stands for divergence with respect to the background metric $h_0$. The remaining term 
$h_2^2$ is the square of $h_2$, identified to an endomorphism using the metric $h_0$. 
The manifold $\bbar{M}:=M\sqcup N$ becomes a smooth compactification of $M$
by declaring that the function $x$ is smooth on $\oM$ and vanishes to order $1$ on the boundary $\pl\bbar{M}=N$. 
Near the boundary, $x$ is the distance function to $N$ with respect to the metric $\bar{g}:=x^2g$. 
In particular it defines a foliation with leaves $\{x={\rm constant}\}$ near the boundary. We call 
${x}$ a \emph{geodesic boundary defining function} associated to ${h}_0$.
The important observation explained in \cite[Lemma 2.1]{Gr0}
is that the function $x$ and the metric $h_0$ above are not determined by a given $g$, but the conformal class of $h_0$ is. Moreover, for each representative $\hat{h}_0$ of the conformal class $[h_0]$, there is a unique smooth boundary defining function $\hat{x}$ on $\bbar{M}$ near $N$ such that the hyperbolic 
metric $g$ near $N$ has the form  \eqref{productdim3} with $x,h_0,h_2$ replaced by $\hat{x},\hat{h}_0,\hat{h}_2$ for some 
tensor $\hat{h}_2$ satisfying the constraints \eqref{constraintseq} with respect to the metric $\hat{h}_0$. 
In other words, conformal representatives of the conformal infinity $N$ of $M$ correspond to certain geometric foliations in the end of $M$. For $(M,g)$ fixed, the \emph{renormalized volume} ${\rm Vol}_R(M,g;\cdot)$ is a function on the conformal class  
$[h_0]$ defined by the regular value at $z=0$ of the meromorphic function $F(z):=\int_{M}x^{z}{\rm dvol}_g$ if $x$ 
is the geodesic boundary defining function associated to $h_0$:
\begin{equation}\label{volrdim3} 
{\rm Vol}_R(M,g;h_0)={\rm FP}_{z=0} \int_{M}x^{z}{\rm dvol}_g
\end{equation}
(here ${\rm FP}_{z=0}$ denotes finite part). 
The meromorphic function $F(z)$ has a pole at $z=0$, and the residue is the topological invariant $-\frac{\pi}{2}\chi(N)$ computed from \eqref{constraintseq} using the Gauss-Bonnet formula. The functional ${\rm Vol}_R(M,g;\cdot)$,
defined on the conformal class $[h_0]$, has the following remarkable properties:
\begin{enumerate}
\item \textbf{``Polyakov Formula''}: If $\omega\in C^\infty(N)$, then 
 \[{\rm Vol}_R(M,g;e^{2\omega}h_0)-{\rm Vol}_R(M,g;h_0)=-\tfrac{1}{4}\int_{\pl M}(|\nabla \omega|^2_{h_0}+
{\rm Scal}_{h_0}\omega) {\rm dvol}_{h_0}.\] 
\item \textbf{Critical points}: For fixed $g$, ${\rm Vol}_R(M,g;\cdot)$ 
is critical, among metrics in the conformal class $[h_0]$ with fixed area, exactly at constant curvature metrics. 
\item \textbf{Extrema}: If $\chi(N)<0$, the critical point is unique in a conformal class (with fixed area) and is a maximum. 
\end{enumerate}
The Polyakov-type formula is easily shown (see Proposition \ref{functionaldim2}) and the extremum is obtained by the classical 
variational approach, see for example \cite[Ch. 14.2]{Tay}. 
When $n=2$, the renormalized volume ${\rm Vol}_R$ is essentially the Liouville functional defined by 
Takhtajan and Zograf \cite{TaZo} for Schottky and quasifuchsian hyperbolic manifolds, see \cite{TaTe}.

Recall that if $N$ is a smooth compact surface with $\chi(N)<0$, the
Teichm\"uller space $\mc{T}(N)$ is the space of conformal classes of metrics on $N$ quotient 
by the group $\mc{D}_0(N)$ of diffeomorphisms of $N$ isotopic to the identity; 
equivalently it is given by the quotient $\mc{G}\backslash\mc{M}(N)$, 
where $\mc{N}$ is the space of smooth metrics on $N$ and $\mc{G}=C^\infty(N)\rtimes\mc{D}_0(N)$ is 
the semi-direct product of the conformal group with $\mc{D}_0(N)$.
By choosing the unique hyperbolic metric in each conformal class, one identifies $\mc{T}(N)$ 
with the space of hyperbolic metrics up to $\mc{D}_0(N)$, this 
can be represented as a slice of hyperbolic metrics transverse to the action of $\mc{D}_0(N)$. As mentioned above, ends of convex co-compact hyperbolic $3$ manifolds are hyperbolic ends in the sense defined by \eqref{productdim3} with $h_0,h_2$ satisfying 
the constraint equations \eqref{constraintseq}. 
Actually, more is true: if $N$ is a given surface, 
any metric $g$ on $(0,\eps)\x N$ of the form \eqref{productdim3} is hyperbolic if and only if
$h_0,h_2$ satisfy \eqref{constraintseq}. A hyperbolic end is thus determined by the pair $(h_0,h_2)$ and we 
denote by  $\mc{E}(N)$ the space of these ends.
The gauge group for $\mc{E}(N)$ is $\mc{G}$, acting by: 
\[\textrm{if }g=\frac{dx^2+h_x}{x^2},\textrm{ then }�(f,\phi).g=\psi^*g, \, \textrm{ with } \psi(x,y)=(x,\psi_x(y))\] 
where $\psi_x:N\to N$ is the diffeomorphism defined so that $\psi_0=\phi^{-1}$ and
$y\in N\mapsto 
(\hat{x},\psi_x\circ \phi(y))\in [0,\eps)\x N$ is  the flow at time $x$ 
of the gradient $\nabla^{\hat{x}^2g}\hat{x}$ (with respect to $\hat{x}^2g$) of the function $\hat{x}$ 
defined to be the unique boundary defining function of $[0,\eps)\x N$
satisfying $|d\hat{x}|_{\hat{x}^2g}=1$ and $(\hat{x}^2g)|_{x=0}=e^{2f}(\phi^{-1})^*h_0$. 
This can be seen as an action on the pairs $(h_0,h_2)$ which determine the end, and the action
on the data  $h_0$ gives rise to $\mc{T}(N)$ as the space of orbits. The action of the conformal part of 
$\mc{G}$ does not have so nice properties on $h_2$, for iinstance $h_2$ is not conformally covariant.
We then represent each conformal class $[h_0]$ by its hyperbolic metric $h_0$,
which is the maximizer of the renormalized volume in the conformal class of area $-2\pi\chi(N)$. 
This allows to identify the quotients
\[\mc{G}\backslash\mc{E}(N) \simeq
\mc{D}_0(N)\backslash\{(h_0,h_2)\in\mc{E}(N);
 h_0\textrm{ hyperbolic}\}\] 
 where $\phi\in\mc{D}_0(N)$ acts by 
  $\phi.(h_0,h_2)=((\phi^{-1})^*h_0,(\phi^{-1})^*h_2)$.
If $h_2^\circ:=h_2-\tfrac{1}{2}\tra_{h_0}(h_2)h_0$ is the trace-free part of $h_2$, we obtain that the pair 
$(h_0,h_2^\circ)$ satisfies
\begin{align}
\scal_{h_0}=-2,&& \tra_{h_0}(h_2^\circ)=0 , && \delta_{h_0}(h_2^\circ)=0.
 \end{align}
The cotangent space $T^*_{h_0}\mc{T}(N)$
to $\mc{T}(N)$ at a hyperbolic metric $h_0$ is naturally represented by symmetric $2$-tensors on $M$ 
which are trace-free and divergence-free with respect to $h_0$ (such a tensor is the real part of a holomorphic quadratic differential on $N$). Consequently, since the action of $\mc{D}_0(N)$ on $(h_0,h^\circ_2)$ coincides with the action of 
$\mc{D}_0(N)$ on the cotangent space to the space of hyperbolic metrics,
we have
\begin{enumerate}
\setcounter{enumi}{3}
\item \textbf{Cotangent vectors as ends}: There is a natural isomorphism $\mc{G}\backslash\mc{E}(N) \to T^*\mc{T}(N)$, given by the correspondence
$(h_0,h_2)\mapsto (h_0,h_2^\circ)$.
\end{enumerate}

For a hyperbolic end, we call $h_0$ and $h_2^\circ$ the Dirichlet and Neumann boundary data at $x=0$, 
by analogy with linear elliptic boundary value problems of order $2$.  Here the elliptic equation on the end $(0,\eps)\x N$ is 
Einstein's equation ${\rm Ric}_g=-2g$ and it is solvable near $x=0$ for any choice of boundary data $(h_0,h_2^\circ)$ satisfying the constraints equation \eqref{constraintseq} with $h_0$ hyperbolic. 

From a global point of view, it is interesting to understand which pairs $(h_0,h_2^\circ)$
occur as boundary data of a convex co-compact hyperbolic manifold. 
A famous theorem of Ahlfors and Bers \cite{ahlfors:riemann,Be} in the quasifuchsian setting, 
extended by Marden \cite{marden,Ma}, states that for a given convex co-compact hyperbolic manifold $M=\Gamma\backslash\hh^3$ 
with conformal boundary $N=\pl M$, there is a smooth map 
\begin{equation}\label{Phihyper}
\Phi:\mc{T}(N)\to \mc{M}_{-1}(M)\quad \textrm{ with}\quad  [x^2\Phi(h_0)|_{N}]=[h_0]
\end{equation}
where $[\cdot]$ denotes conformal class, $x$ is any smooth boundary defining function of the compactification $\bbar{M}$ 
and $\mc{M}_{-1}(M)$ is the space of hyperbolic metrics on $M$ which are convex co-compact, considered up to isotopy. 
From this viewpoint, the relevant object is the Dirichlet-to-Neumann map $h_0\mapsto h_2^\circ$, 
where $h_2^\circ$ is the Neumann data of the metric $\Phi(h_0)$.
This map can be understood as a $1$-form on $\mc{T}(N)$. 
The following facts were proved by Krasnov-Schlenker \cite{KrSc} 
(see also \cite{mcmullen} for the Lagrangian property when $M$ is quasifuchsian and \cite{GuMo} 
for an alternative proof using Chern-Simons invariants):
\begin{enumerate}
\setcounter{enumi}{4}
\item \textbf{Lagrangian submanifold}: The space $\mc{L}$ of couples of boundary data $(h_0, h_2^\circ)$ corresponding to convex co-compact metrics on $M$ is Lagrangian in $T^*\mc{T}(N)$, endowed with the natural symplectic structure.
\item \textbf{Generating function}: $\mc{L}$ is the graph of the exact $1$-form given by the differential of $h_0\mapsto {\rm Vol}_{R}(M, \Phi(h_0); h_0)$
over $\mc{T}(N)$. More precisely, if $\dot{h}_0\in T_{h_0}\mc{T}(N)$ is a variation of hyperbolic metrics at 
$h_0\in \mc{T}(N)$, then 
\begin{equation}
d {\rm Vol}_R(M,\Phi(h_0),h_0).\dot{h}_0 = -\tfrac{1}{4}\int_{N}\cjg h_2^\circ,\dot{h}_0\cjd{\rm dvol}_{h_0}.
\end{equation}
\end{enumerate}
 
Another interesting property of the function $h_0\mapsto {\rm Vol}_R(M,\Phi(h_0);h_0)$ was discovered by Takhtajan-Teo \cite{TaTe} (using 
the correspondence with the Liouville action of \cite{TaZo}) and Krasnov-Schlenker \cite{KrSc} (when $M$ is a Schottky or a quasifuchsian manifold): it generates the Weil-Petersson 
K\"ahler form $\omega_{WP}$ on $\mc{T}(N)$. This was extended by Guillarmou-Moroianu \cite{GuMo} in the general setting using Chern-Simons invariants: 
\begin{enumerate}
\setcounter{enumi}{6}
\item \textbf{K\"ahler potential}: The function ${\rm Vol}_R: h_0\mapsto {\rm Vol}_R(M,\Phi(h_0);h_0)$ is a K\"ahler potential 
for the Weil-Petersson form:
\[\pl \bar{\pl}{\rm Vol}_R =\tfrac{i}{16} \omega_{WP}.\]
 \end{enumerate}
 
An important class of convex co-compact hyperbolic $3$-manifolds are quasifuchsian manifolds, diffeomorphic to
a topological cylinder $M=\rr_t\x S$ with basis a surface $S$ of genus $g\geq 2$.
The conformal boundary has two connected components $S^+\sqcup S^-$, both diffeomorphic to $S$,
with conformal classes $[h_0^\pm]$ corresponding respectively to the limit $t\to \pm \infty$. Here we choose
the hyperbolic representative $h_0^\pm$ in the conformal classes $[h_0^\pm]$ on $S^\pm$. For each pair 
$h_0=(h_0^+,h_0-)$ there exists a quasifuchsian hyperbolic metric $g=\Phi(h_0)$ on $M$
and, as seen before,  each of the $2$ ends $e^\pm$ has a metric of the form 
$x^{-2}(dx^2+h_0^\pm+x^2h_2^\pm+\tfrac{1}{4}x^4(h_2^\pm)^2)$, with 
the trace free part $h_2^{0\pm}$ of $h^\pm_2$ interpreted as a cotangent
vector to $\mc{T}(S)$. The pair $(h_0^+,h_0-)$ then 
defines two points $h^{0-}_{2}\in T_{h^-_0}\mc{T}(S)$ and 
$h_2^{0+}\in T_{h^+_0}\mc{T}(S)$.
We can now fix $h^+_0$ and consider the linear maps 
\begin{align*}
\phi_{h^+_0}:T_{h^-_0}\mc{T}(S) \rightarrow T^*_{h^+_0}\mc{T}(S),&&
\phi_{h^-_0}:T_{h^+_0}\mc{T}(S) \rightarrow T^*_{h^-_0}\mc{T}(S)
\end{align*}
sending a first-order variation of $h^\mp_0$ to the corresponding variation of $h^{0\pm}_2$. 
\begin{enumerate}
\setcounter{enumi}{7}
\item \textbf{Quasifuchsian reciprocity}: The maps $\phi_{h_0^+}$ and $\phi_{h_0^-}$ are adjoint. 
\end{enumerate}
This was discovered by McMullen \cite{mcmullen} (see also Krasnov-Schlenker \cite{KrSc} for an alternative proof). Still in the setting of quasifuchsian manifolds, the renormalized volume behaves essentially like a squared distance for the 
Weil-Petersson metric near the Fuchsian locus (i.e., the diagonal in $\mc{T}(S)\x \mc{T}(s)$ corresponding to 
$h_0^+=h_0^-$); this follows from
\begin{enumerate}
\setcounter{enumi}{8}
\item \textbf{A local distance in Teichm\"uller}: for $h_0^+\in \mc{T}(S)$ fixed, and 
denoting $h_0=(h_0^+,h_0^-)\in \mc{T}(S)\x \mc{T}(S)$ 
the function $h_0^-\mapsto {\rm Vol}_R(M,\Phi(h_0); h_0)$ has a critical point at $h_0^-=h_0^+$, unique near $h_0^+$, which is a local minimum (the Hessian is positive definite). 
\item\textbf{A global distance in Teichm\"uller}: There are constants $C_0,C_1>0$ such that 
\[   \frac{d_{\rm WP}(h_0^+,h_0^-)}{C_0}-C_1 \leq {\rm Vol}_R(M,\Phi(h_0); h_0)\leq C_0 d_{\rm WP}(h_0^+,h_0^-)+C_1\] 
where $d_{\rm WP}(\cdot,\cdot)$ denotes distance with respect to the Weil-Petersson metric.
\end{enumerate}
The first property follows from Krasnov-Schlenker \cite{KrSc} and the second by combining the result of 
Schlenker \cite{Sc} and of Brock \cite{Br}.
 
\subsection{Dimension $n+1$ odd}
We aim to understand here to which extent the theory in dimension $2+1$ makes sense in higher odd dimensions. 
By analogy with $n=2$, we are interested in the set $\mc{T}(N)$ of conformal classes of metrics on a compact manifold 
$N$ of even dimension $n$, up to the group $\cD_0(N)$ of diffeomorphisms isotopic to identity. This space can be defined as a quotient of the space of smooth metrics 
$\mc{M}(N)$ by the action of the semi-direct product $C^\infty(N)\rtimes \cD_0(N)$. We assume that $(N,h_0)$ does not admit nonzero conformal Killing vector fields, 
so that a neighbourhood of the image of $h_0$ in the quotient is an infinite dimensional Fr\'echet manifold 
(in contrast to $n=2$, where $\dim\cT(N)=-3\chi(N)$). 
Following Fefferman-Graham \cite{FeGr}, we can view the conformal class $(N,[h_0])$ as the conformal boundary 
of a \emph{Poincar\'e-Einstein end}, that is a cylinder $(0,\eps)_x\x N$ equipped with a metric 
\begin{align}\label{expansiong}
g=\frac{dx^2+h_x}{x^2}, &&
h_x\sim_{x\to 0} \sum_{\ell=0}^\infty h_{x,\ell}(x^{n}\log x)^\ell 
\end{align} 
where $h_{x,\ell}$ are one-parameter families of tensors on $M$ depending smoothly on $x$, and satisfying the approximate Einstein equation as $x\to 0$
\[{\rm Ric}_g=-ng+\mc{O}(x^\infty).\]
The tensor $h_{x,0}$ has a Taylor expansion at $x=0$ given by
\[h_{x,0}\sim_{x\to 0} \sum_{j=0}^\infty x^{2j}h_{2j}\]
where $h_{2j}$ are formally determined by $h_0$ if $j<n/2$ and formally determined by the pair 
$(h_0,h_n)$ for $j>n/2$; for $\ell\geq 1$, the tensors $h_{x,\ell}$ have a Taylor expansion 
at $x=0$ formally determined by $h_0,h_n$. 
Like $h_2$ in \eqref{constraintseq}, $h_n$ is a  formally undetermined tensor which satisfies some constraints equations: there 
exist a function $T_n$ and a $1$-form $D_n$, natural in terms of the tensor $h_0$ (see Definition \ref{formeldet}), such that
\begin{align} 
\tra_{h_0}(h_n)= T_n, && \delta_{h_0}(h_n)=D_n.
\end{align}  
The formula for $T_n,D_n$ is complicated and not known in general, but in principle 
it can be computed reccursively. An \emph{Asymptotically Hyperbolic Einstein} (AHE) manifold is an
Einstein manifold $(M,g)$ with ${\rm Ric}_g=-ng$ which compactifies smoothly to some $\bbar{M}$ so that there exists 
a smooth boundary defining function $x$ with respect to which  $g$ has the form \eqref{expansiong}. The conformal boundary $N=\pl \bbar{M}$ inherits naturally a conformal class $[x^2g|_{TN}]$. Exactly like when $n=2$, each conformal representative $h_0\in[x^2g|_{TN}]$  determines a unique geodesic boundary defining function $x$ near $N$ so that $g$ has the form \eqref{expansiong}. The renormalized volume  ${\rm Vol}_R(M,g;h_0)$ was apparently introduced by physicists \cite{HeSk}, and appeared in \cite{Gr0} in the mathematics literature. We define it using a slightly different procedure from \cite{HeSk,Gr}, using the same approach as for $n=2$ above, that is using the formula
 \begin{equation}\label{volrdimn} 
{\rm Vol}_R(M,g;h_0):={\rm FP}_{z=0} \int_{M}x^{z}{\rm dvol}_g;
\end{equation}
the function $F(z)=\int_{M}x^z{\rm dvol}_g$ has a pole at $z=0$ with residue $\int_{N}v_n{\rm dvol}_{h_0}$, where $v_n$ is the function appearing as the coefficient of $x^n$ in the expansion of the volume form near $N$:
\begin{align}\label{dvolgexp}
{\rm dvol}_g= (v_0+v_2x^2+\dots+v_nx^n+o(x^n))dx \,{\rm dvol}_{h_0}, && v_0=1.
\end{align}
This method for renormalizing the volume was used for AHE manifolds e.g.\ in the work of Albin \cite{Al}. The quantity $v_n$, called a conformal anomaly in the physics literature, is formally determined by $h_0$ (it does not depend on the Neumann datum $h_n$), and its integral $L:=\int_{N}v_n{\rm dvol}_{h_0}$ is a conformal invariant \cite{GrZw}. For instance in dimension $n=4$, 
\[4v_4=\sigma_2({\rm Sch}_{h_0})\]
is the symmetric function of order $2$ in the eigenvalues of the Schouten tensor ${\rm Sch}_{h_0}=\tfrac{1}{2}({\rm Ric}_g-\tfrac{1}{6}{\rm Scal}_{h_0}h_0)$, see Lemma \ref{v2v4}. We first show 
\begin{theo}
Let $(M,g)$ be an odd dimensional AHE manifold with conformal boundary $N$ equipped with the conformal class $[h_0]$.
\begin{enumerate}
\setcounter{enumi}{0}
\item \textbf{Polyakov type formula}: Under conformal change $e^{2\omega_0}h_0$,  
the renormalized volume ${\rm Vol}_R(M,g;h_0)$ satisfies 
\[{\rm Vol}_R(M,g;e^{2\omega_0}h_0)={\rm Vol}_R(M,g;h_0)+ 
\int_{\pl M} \sum_{j=0}^{n/2} v_{2j}(h_0)\omega_{n-2j} \,{\rm dvol}_{h_0}
\]
where $v_{2i}$ are the volume coefficients of \eqref{dvolgexp} and 
$\omega_{2j}$ are polynomial expressions in $\omega_0$ and its derivatives of order at most $j$.
\item \textbf{Critical points}: The critical points of ${\rm Vol}_R(M,g;\cdot)$, 
among metrics of fixed volume in the conformal class $[h_0]$ are those metrics $h_0$ satisfying 
$v_n(h_0)={\rm constant}$. 
\item \textbf{Extrema}: Assume that $[h_0]$ contains an Einstein metric $h_0$ with non-zero Ricci curvature. 
Then $h_0$ is a local extremum for ${\rm Vol}_R(M,g;\cdot)$ in its conformal class with fixed volume: it is a maximum 
if ${\rm Ric}_{h_0}<0$ or 
$n/2$ is odd,  it is a minimum if $n/2$ is even.
Moreover if $(N,[h_0])$ is not the sphere, then for each conformal classes $[h]$
close to $[h_0]$, there is a a metric $h\in[h]$ solving $v_n(h)={\rm constant}$ and
${\rm Vol}_R(M,g;h)$ is a local extremum in $[h]$ with fixed volume. 
\end{enumerate}
\end{theo}
These properties are proved in Section \ref{sectionvol}. The property (2) also follows directly 
from the discussion after \cite[Th. 3.1]{Gr0} and is certainly known from specialists and theoretical physicists.

Following the theory in dimension $n=2$, after choosing representatives in the conformal class 
satisfying the condition $v_n={\rm constant}$, it is natural to expect a correspondence between Poincar\'e-Einstein ends and 
cotangent vectors to the space $\mc{T}(N)$ of conformal structures (i.e. conformal classes modulo $\mc{D}_0(N)$). 
A Poincar\'e-Einstein end
is determined by the pair $(h_0,h_n)$. 
When $\mc{T}(N)$ (or an open subset) has a Fr\'echet manifold structure, 
we can use a symplectic reduction of the cotangent space $T^*\mc{M}(N)$ of the space of metrics $\mc{M}(N)$ by the semi-direct product
$C^\infty(N)\rtimes \cD_0(N)$, and we can identify $T_{[h_0]}^*\mc{T}(N)$ to the space of 
trace-free and divergence-free tensors on $N$ (with respect to $h_0$).
Unlike for $n=2$, even after choosing a metric $h_0$ with $v_n(h_0)={\rm constant}$, 
the formally undetermined tensor $h_n$ is neither divergence-free, nor is its trace constant. However, we show the following 
\begin{theo}
There exists a symmetric tensor $F_n$, formally determined by $h_0$, 
such that $G_n:=-\tfrac{1}{4}(h_n+F_n)$ satisfies 
\begin{align} 
{\rm Tr}_{h_0}(G_n)=\tfrac12 v_n, && \delta_{h_0}(G_n)=0.
\end{align} 
\begin{enumerate}
\setcounter{enumi}{3}
\item \textbf{Cotangent vectors as ends}: Assume that there exists an open set $\mc{U}\subset\mc{T}(N)$ and
a smooth Fr\'echet submanifold $\mc{S}_0\subset \mc{M}(N)$ of metrics $h_0$ solution to 
$v_n(h_0)=\int_{N}v_n(h_0){\rm dvol}_{h_0}$ so that the projection $\pi: \mc{M}(N)\to \mc{T}(N)$ is a 
homeomorphism from $\mc{S}$ to $\mc{U}$. 
Then there is a bijection between the space of Poincar\'e-Einstein ends with $h_0\in \mc{S}$ and
the space $T_{\mc{U}}^*\mc{T}(N)$ given by
 $(h_0,h_n)\mapsto (h_0,G_n^\circ)$, where $G_n^\circ$ is the trace-free 
 part of $G_n$. 
\end{enumerate}
\end{theo}

The  existence of a slice $\mc{S}_0$ is proved for instance in Corollary \ref{sectionvncst} in a neighbourhood of a conformal class containing an Einstein metric which is not the sphere. We have learnt from Robin Graham that there is a result
related to the first part of the Theorem about $G_n$ in the physics literature \cite{DeSkSo}, although the renormalization 
for the volume seems a bit different from ours. 

As for $n=2$, we can define the Cauchy data for the Einstein equation to be $(h_0,G_n^\circ)$ where $h_0$ solves  $v_n(h_0)=\int_{N}v_n(h_0){\rm dvol}_{h_0}$. We may ask if those Cauchy data which are ends of AHE manifolds span a Lagrangian submanifold of $T^*\mc{T}(N)$. 
\begin{theo}\label{slagrange}
Assume that there is a smooth submanifold $\mc{S}_0\subset \mc{M}(N)$ of metrics $h_0$ solving
$v_n(h_0)=\int_{N}v_n(h_0){\rm dvol}_{h_0}$ so that the projection $\pi: \mc{M}(N)\to \mc{T}(N)$ is a homeomorphism from $\mc{S}_0$ to $\mc{U}$. Let $\bbar{M}$ be a smooth manifold with boundary $N$ and assume that there is a smooth map 
$\Phi:\mc{S}_0\to \mc{M}(M)$ such that ${\rm Ric}_{\Phi(h_0)}=-n\Phi(h_0)$ and $[x^2\Phi(h_0)|_N]=[h_0]$ for some 
boundary defining function $x$. 
\begin{enumerate}
\setcounter{enumi}{4}
\item \textbf{Lagrangian submanifold}: The set $\mc{L}$ of Cauchy data $(h_0, G_n^\circ)$ of 
the AHE metrics $\Phi(h_0)$ with  $h_0\in \mc{S}_0$ is a Lagrangian submanifold in $T^*\mc{T}(N)$ with respect to the canonical symplectic structure.
\item \textbf{Generating function}: $\mc{L}$ is the graph of the exact $1$-form given by the differential of $h_0\mapsto {\rm Vol}_{R}(M, \Phi(h_0); h_0)$
over $\mc{S}_0$. More precisely, if $\dot{h}_0\in T_{h_0}\mc{S}_0$ is a variation of metrics satisfying $v_n=\int_Nv_n$ and 
$h_0\in \mc{S}_0$, then 
\begin{equation}
d {\rm Vol}_R(M,\Phi(h_0),h_0).\dot{h}_0 =\int_{N}\cjg G_n^\circ,\dot{h}_0\cjd{\rm dvol}_{h_0}.
\end{equation}
\end{enumerate}
\end{theo}
Here, what we mean by Lagrangian is an isotropic submanifold such that the projection on the base is a diffeomorphism.
In Section \ref{qfr}, we show that the assumptions of Theorem \ref{slagrange} are satisfied for instance in a neighbourhood of what we call 
a Fuchsian-Einstein manifold, in a way similar to the quasifuchsian metrics near a Fuchsian metric. 
A Fuchsian-Einstein  metric is a product $M=\rr_t\x N$ with a metric $g^0:=dt^2+\cosh^2(t)\gamma$ where $\gamma$ 
is a metric on $N$ such that 
${\rm Ric}_{\gamma}=-(n-1)\gamma$ and the sectional curvatures of $\gamma$ are non-positive.
By Corollary \ref{sectionvncst}, near an Einstein metric $\gamma$ on a compact manifold
$N$ with negative Ricci curvature, there is a smooth slice $\mc{S}_0\subset \mc{M}(N)$ of metrics solution to 
$v_n=\int_{N}v_n$ and so that the projection $\pi: \mc{M}(N)\to \mc{T}(N)$ is a homeomorphism from $\mc{S}_0$ 
to a neighbourhood $\mc{U}$ of $[h_0]$. Using a result by Lee \cite{lee:fredholm}, and 
possibly after taking an open subset of $\mc{S}_0$ instead of $\mc{S}_0$,  for each pair 
$(h_0^+,h_0^-)\in \mc{S}_0\x \mc{S}_0$ there exists an AHE metric $g=\Phi(h_0^+,h_0^-)$ satisfying 
\begin{align*}
{\rm Ric}_g=-ng  \textrm{ on }M,&& \Phi(\gamma,\gamma)=g^0, &&
 [x^2g|_{t=\pm \infty}]=[h_0^\pm]\, \textrm{ for } \,x:=e^{-|t|}.
 \end{align*}
For each of the two ends ($t\to \pm \infty$) we have a tensor $G_n^{\circ\pm}$ which can be viewed as an element 
in $T^*\mc{T}(N)$. Like for $n=2$, we have
\begin{theo}\label{quasifuschsianrec}
Fix $h^+_0\in\mc{S}_0$, respectively $h_0^-\in \mc{S}_0$, and consider the linear maps
\begin{align*}
\phi_{h^+_0}: T_{h_0^-}\mc{S}_0 \rightarrow T^*_{h^+_0}\mc{T}(N), &&
\phi_{h^-_0}:T_{h^+_0}\mc{S}_0 \rightarrow T^*_{h^-_0}\mc{T}(N)
\end{align*}
defined as follows  
\[ \begin{gathered}
\phi_{h_0^+}: \dot{h}^-_0\mapsto (dG_n^{\circ +})_h(\dot{h}_0^-,0) -\tfrac{n}{2} \cjg G_n^{\circ -}(h),\dot{h}_0^-\cjd h^+_0\\
\phi_{h_0^-}: \dot{h}^+_0\mapsto (dG_n^{\circ -})_h(0,\dot{h}_0^+) -\tfrac{n}{2} \cjg G_n^{\circ +}(h),\dot{h}_0^+\cjd h^-_0
\end{gathered}\]
where $G_n^{\circ\pm}$ and its variation are obtained using the AHE metrics $g=\Phi(h_0^+,h_0^-)$. Then
\begin{enumerate}
\setcounter{enumi}{6}
\item \textbf{Quasifuchsian reciprocity}: The linear maps $\phi_{h^+_0}$ and $\phi_{h_0^-}$ are adjoint.
\end{enumerate}
\end{theo}
A more general statement holds, which states that  $(dG_n^\circ)_{h_0}-\ndemi \cjg G_n^\circ(h_0),\cdot\cjd h_0$  
is self-adjoint if $G_n^\circ$ is the cotangent data coming from an Einstein filling in the bulk $M$ and 
 $dG_n^\circ$ is its linearisation (see Corollary \ref{cr:kleinian}).
 
Finally, we study the second variation of $h_0=(h_0^+,h_0^-)\mapsto 
{\rm Vol}_R(M,\Phi(h_0),h_0)$ at the Fuchsian-Einstein metric, i.e., when 
$\Phi(h_0)=g^0$.
\begin{theo}
In the setting of Theorem \ref{quasifuschsianrec},
consider the function ${\rm Vol}_R: \mc{S}\x\mc{S}\to \rr$ defined by ${\rm Vol}_R(h_0):= {\rm Vol}_R(M,\Phi(h_0),h_0)$ for $h_0=(h_0^+,h_0^-)\in
\mc{S}\x\mc{S}$, and set $n=4$. Assume that $L_{\gamma}-2> 0$ where $L_\gamma:=\Delta_{\gamma}-2\rci_\gamma\geq 0$ is the linearized Einstein operator at $\gamma$ acting on divergence-free, trace-free tensors  (see Section \ref{secDNmap} for precise definition).
\begin{enumerate}
\setcounter{enumi}{7}
\item \textbf{Hessian at the Fuchsian-Einstein locus}: The point $h_0=(\gamma,\gamma)$ is a critical point
for $\rvol$, i.e.,  
$d{\rm Vol}_R(h_0)=0$ on $T_\gamma \mc{S}\x T_\gamma\mc{S}$, the Hessian at $(\gamma,\gamma)$ is positive in the sense that
there exists $c_0>0$ such that for all $\dot{h}_0=(\dot{h}_0^+,\dot{h}_0^-)
\in T_{\gamma}\mc{S}\x T_{\gamma}\mc{S}$ with $\delta_{\gamma}(\dot{h}^\pm_0)=0$
\[ {\rm Hess}_{h_0}({\rm Vol}_R)(\dot{h}_0,\dot{h}_0)\geq c_0||\dot{h}_0||^2_{H^2(N)}\]
where $H^2(N)$ is the $L^2$-based Sobolev space of order $2$.
\end{enumerate}
\end{theo}

The lower bound $L_{\gamma}-2\geq 0$ is for instance satisfied if $\gamma$ has constant sectional curvature $-1$ and
$\ker d_D=\ker d^*_D=0$ where $d_D$ is the exterior derivative on $T^*N$-valued $1$-forms and $d_D^*$ its adjoint. 
In Proposition  \ref{prophessien}, we compute the Hessian explicitly: the quadratic form acting on divergence-free tensors tangent to $\mc{S}\x \mc{S}$  is given by a self-adjoint linear elliptic pseudo-differential operator $\mc{H}$, ${\rm Hess}_{(\gamma,\gamma)}({\rm Vol}_R)(\dot{h}_0,\dot{h}_0)=
\cjg \mc{H}\dot{h}_0,\dot{h}_0\cjd_{L^2}$, and $\mc{H}$ is in fact a function of $L_\gamma$ 
(the condition $\dot{h}^\pm_0\in T_{\gamma}\mc{S}$ and $\delta_{\gamma}(\dot{h}_0^\pm)=0$ actually implies that 
$\tra_\gamma(\dot{h}_0^\pm)=0$). If $L_\gamma-2$ has non-positive eigenvalues, the same result remains true
along deformations orthogonal to (the finite dimensional) range of $\indic_{\rr_-}(L_\gamma-2)$.

The equivalent of the K\"ahler potential property valid in dimension $n=2$ does not seem to extend to general dimensions without additional geometric assumptions.

\subsection{Dimension $n+1$ even} 
When $n+1$ is even, the renormalized volume defined by \eqref{volrdimn} has been more extensively studied. It is also an interesting quantity, more tractable than in the case $n+1$ odd, but has quite different properties. For instance it is related to the Chern-Gauss-Bonnet formula and does not depend on the choice of conformal representative $h_0$ (i.e., it is independent of the geodesic boundary-defining function $x$). Anderson \cite{An} gave a formula when $n+1=4$ for ${\rm Vol}_R(M,g)$ in terms of the $L^2$ norm of the Weyl tensor and the Euler characteristic $\chi(M)$
 if $g$ is AHE. This was extended by Chang-Qing-Yang \cite{ChQiYa} in higher dimensions (see also Albin \cite{Al} for Chern-Gauss-Bonnet formula), while Epstein \cite[Appendix A]{PaPe} proved that for convex co-compact hyperbolic manifolds it equals a constant times $\chi(M)$.
When $n=4$, Chang-Qing-Yang \cite{ChQiYa0} also proved a rigidity theorem if the renormalized volume is pinched enough near that of hyperbolic space $\hh^4$. As for variations, Anderson \cite{An} and Albin \cite{Al} proved that the derivative of the renormalized volume
for AHE metrics is given by the formally undetermined tensor $-\tfrac{1}{4}h_n$, see Theorem \ref{alban}. A byproduct of our computation in Section \ref{secDNmap} is a formula 
for the Hessian of the renormalized volume when $n+1$ is even, at a Fuchsian-Einstein metric, see expresion \eqref{formulaforrn} in Proposition \ref{prophessien}.

\subsection*{Aknowledgements} We thank Thomas Alazard, Olivier Biquard, Alice Chang, Erwann Delay, Yuxin Ge, Robin Graham, Matt Gursky, Dima Jakobson, Andreas Juhl, Andrei Moroianu and Yoshihiko Matsumoto for helpful 
discussions related to this project. Thanks also to Semyon Dyatlov for his help with matlab.

\section{Moduli space of conformal structures and Poincar\'e-Einstein manifolds}
\label{sc:basics}

\subsection{Spaces of metrics and conformal structures} \label{ssc:conformal}

We use the notions of tame Fr\'echet manifoldd and Fr\'echet Lie groups as in Hamilton \cite{Ha}. 
Let $N$ be a compact smooth manifold of dimension $n$, and $\cM(N)$ the set of Riemannian metrics on $N$. This set is an open convex subset in the Fr\'echet space $C^\infty(N,S^2N)$ of symmetric smooth $2$-tensors on $N$. It has a tautological non-complete Riemannian metric  given on $T_h\cM(N)=C^\infty(N,S^2N)$ by the $L^2$ product with respect to $h\in\mc{M}(N)$:
\begin{align*}
 \cjg k_1,k_2\cjd :=\int_{N} \cjg k_1,k_2\cjd_h {\rm dvol}_h,&& k_1,k_2\in T_h\mc{M}(N)
\end{align*}
where $\cjg k_1,k_2\cjd_h=\tra(h^{-1}k_1h^{-1}k_2)$ is the scalar product on $S^2N$ induced naturally 
by $h$ (here $K=h^{-1}k$ means the symmetric endomorphism defined by $h(K\cdot,\cdot)=k$).
 Let $\cD(N)$ be the group of smooth diffeomorphisms of $N$ and $\cD_0(N)$ the connected component of the identity. The groups $\cD_0(N)$ and $C^\infty(N)$ are Fr\'echet Lie groups, the latter being in fact a Fr\'echet vector space. Consider the map
\begin{align*} \Phi:  C^\infty(N)\times \cD_0(N) \times \cM(N)\to \cM(N),&& (f,\phi,h)\mapsto e^{2f} (\phi^{-1})^*h.
\end{align*}
This map defines an action of the semi-direct product $\cG:=C^\infty(N)\rtimes \cD_0(N)$ on $\cM(N)$,
and this action is smooth and proper if $N$ is not the sphere $S^n$ (see Ebin \cite{Eb}, Fischer-Moncrief \cite{FiMo}). 
The isotropy group at a metric $h$ for the action $\Phi$ is the group of conformal diffeomorphism of $(N,h)$ isotopic to the Identity; by 
Obata \cite{Ob2} it is compact if $N$ is not the sphere. 

\subsection*{Definitions}  The object studied in this paper is the space of conformal structures 
(called \emph{quantum conformal superspace} in physics), denoted by 
\begin{equation}\label{defofTN}
\mc{T}(N):=\cG\backslash \cM(N).
\end{equation}
This space is the Teichm\"uller space when $n=2$ and $N$ has negative Euler characteristic. In higher dimension, it is infinite dimensional and has a complicated structure near general metrics. 
In \cite{FiMo}, Fischer-Moncrief describe the structure of $\mc{T}(N)$: they show for instance that it 
is a smooth Inverse Limit Hilbert orbifold if the degree of symmetry of $N$ is $0$ (the isotropy group is then finite).
Moreover, if the action is proper and the isotropy group at a metric $h$ is trivial, then a neighbourhood of 
$[h]$ in $\mc{T}(N)$ is a Fr\'echet manifold.  By a result of Frenkel \cite{Fr}, the isotropy group is trivial if $h\in\cM(N)$ is a metric of negative Ricci curvature and non-positive sectional curvatures.
An equivalent way to define $\mc{T}(N)$ is to 
consider $\mc{D}_0(N)\backslash\mc{C}(N)$, where 
\begin{equation}\label{defofCN}
\mc{C}(N):=C^\infty(N)\backslash\mc{M}(N)
\end{equation} 
is the space of conformal classes of metrics on $N$. 

\subsection*{Slices} Since we will use this later, let us describe the notion of \emph{slice} introduced by Ebin \cite{Eb} in these settings. 
We will say that $\mc{S}\subset \mc{M}(N)$ is a slice at $h_0\in \mc{M}(N)$ for the conformal action of $C^\infty(N)$ if it is a tame Fr\'echet submanifold such that there is a neighbourhood $U$ of $0$ in $C^\infty(N)$ and a neighbourhood $V\subset \mc{M}(N)$ of $h_0$  such that 
\begin{align}\label{slice1} 
\Psi: U \x \mc{S} \to V, && (f,h)\mapsto e^{2f}h
\end{align}
is a diffeomorphism of Fr\'echet manifolds. Since the action of $C^\infty(N)$ on $\mc{M}(N)$ is free and proper, it is easy to see that $\Psi$ extends to $C^\infty(N)\x \mc{S}\to \mc{M}(N)$ and is injective. In other words, $\mc{S}$ defines a tame Fr\'echet structure on $\mc{C}(N)$ near the conformal class $[h_0]$.
Similarly, if $\mc{S}\subset \mc{M}(N)$ is a Fr\'echet submanifold containing $h_0$, on which a neighbourhood $U\subset \mc{D}_0(N)$ of ${\rm Id}$ acts smoothly, then a Fr\'echet submanifold $\mc{S}_0$ of $\mc{S}$ is a slice at $h_0$ for the action of  $\mc{D}_0(N)$ if there exists a neighbourhood $V\subset \mc{S}$ of $h_0$ such that
\begin{align}\label{slice2} 
\Phi: U \x \mc{S}_0 \to V,&& (\phi,h)\mapsto (\phi^{-1})^*h
\end{align}
is a diffeomorphism of Fr\'echet manifolds. Extending $\Phi$ to $\mc{D}_0(N)\x \mc{S}_0\to \mc{M}(N)$,
and assuming that the action of $\mc{D}_0(N)$ on $\Phi(\mc{D}_0(N)\x \mc{S}_0)$ is free and proper, 
the extension of $\Phi$ is injective in a small neighbourhood of $h_0$ in $\mc{S}_0$. 
If $\mc{S}$ was a slice for the conformal action, 
then $\mc{S}_0$ is a slice at $h_0$ for the action of 
$\mc{G}$ on $\mc{M}(N)$, thus giving a tame Fr\'echet structure on $\mc{T}(N)$ near the class of $h_0$ in $\mc{T}(N)$.

\subsection*{Cotangent bundles}
The tangent bundle $T\mc{M}(N)$ over $\mc{M}(N)$ is the trivial Fr\'echet bundle $\mc{M}(N)\x C^\infty(N,S^2N)$.
The topological dual bundle to $T\mc{M}(N)$ is not a Fr\'echet manifold since 
it should contain distributional tensors. In this work, we are interested in 
$C^\infty$ objects and Fr\'echet manifolds, we thus define the \emph{smooth cotangent space} $T_{h}^*\mc{M}(N)$ 
to be the vector space of continuous linear forms on $T_{h}\mc{M}(N)$ which are 
represented by smooth tensors through the $L^2$ pairing: 
\[ k^*\in T_{h}^*\mc{M}(N)\textrm{ if }\exists k\in T_h\mc{M}(N),\,\forall v\in T_{h}\mc{M}(N),\,\, 
k^*(v)=\int_N\cjg k,v\cjd_{h}{\rm dvol}_h.\]
This identifies the smooth cotangent bundle $T^*\mc{M}(N)$ with $T\mc{M}(N)=\mc{M}(N)\x C^\infty(N,S^2N)$, making it a Fr\'echet bundle.
There exists a symplectic form $\Omega$ on $T^*\mc{M}(N)$,  derived from the Liouville canonical $1$-form:
\begin{equation}\label{sympform}
\Omega_{(h,k)}((\dot{h}_1,\dot{k}_1),(\dot{h}_2,\dot{k}_2))=\int_{N}\cjg \dot{k}_1,\dot{h}_2\cjd_h-\cjg\dot{k}_2,\dot{h}_1\cjd_h + \ndemi\cjg \dot{h}_2\tra_{h}(\dot{h}_1) -\dot{h}_1\tra_h(\dot{h}_2),k\cjd_h  {\rm dvol}_{h}.
\end{equation} 
The group $\cG$  acts on $T^*\mc{M}(N)$, with a symplectic action induced from the base and using the 
Riemannian metric on $\cM(N)$:
\begin{equation}\label{actionsurT*}
(f,\phi):(h,k)\mapsto \left(e^{2f} (\phi^{-1})^*h, e^{(2-n)f}(\phi^{-1})^*k\right). 
\end{equation}

We then define (locally) the cotangent bundle to $\mc{T}(N)$.
We will always assume that there is a slice $\mc{S}_0$ at $h_0$ representing a neighbourhood $U\subset \mc{T}(N)$ of the class $[h_0]$, as we just explained.
The tangent space $T_{[h]}\mc{T}(N)$ at a point $[h]\in \mc{T}(N)$ near $[h_0]$ is then identified 
with $T_h\mc{S}_0$ where $h$ is the representative of $[h]$ in $\mc{S}_0$, and $T\mc{T}(N)$
is then locally represented near $[h_0]$ as a Fr\'echet subbundle of $T_{\mc{S}_0}\mc{M}(N)$.
We define the \emph{smooth cotangent space} $T_{[h]}^*\mc{T}(N)$ 
to be the vector space of continuous linear forms on $T_{h}\mc{S}_0\simeq T_{[h]}\mc{T}(N)$ which are 
represented by smooth tensors through the $L^2$ pairing and vanish on the tangent space of the orbit 
$\mc{G}_h$ of $h$ by the group $\mc{G}$: 
\[ k^*\in T_{[h]}^*\mc{T}(N)\textrm{ if }\exists k\in T_h\mc{M}(N),\,\forall v\in T_{h}\mc{S}_0,\,\, 
k^*(v+T_h\mc{G}_h)=\int_N\cjg k,v\cjd_{h}{\rm dvol}_h.\]
Since $T_h\mc{G}_h=\{L_Xh+fh; X\in C^\infty(N,TN),f\in C^\infty(N)\}$ (where $L_Xh$ is the Lie derivative), $k$ must satisfy
\begin{align*}
\int_N\cjg k,L_Xh+fh\cjd_{h}{\rm dvol}_h=0,&& \forall X\in C^\infty(N,TN),f\in C^\infty(N),
\end{align*}
which is equivalent to asking $\delta_h(k)=0$ and $\tra_h(k)=0$. The smooth cotangent bundle $T^*\mc{T}(N)$
over a neighbourhood $U\subset \mc{T}(N)$ of  $[h_0]$ represented by a slice $\mc{S}_0$ is then 
\begin{equation}\label{T^*TN} 
T_{U}^*\mc{T}(N)=\{(h,k)\in \mc{S}_0\x C^\infty(N,S^2N);\delta_h(k)=0, \tra_h(k)=0\}.
\end{equation}
\begin{lem}
Assume that $h$ has no conformal Killing vector fields for all $h\in \mc{S}_0$. 
The space $T_U^*\mc{T}(N)$ is a Fr\'echet subbundle of $T_{\mc{S}_0}\mc{M}(N)$, therefore a Fr\'echet 
bundle over $\mc{S}_0$.
\end{lem}

\begin{proof} We are going to exhibit a trivialisation of the fiber bundle defined by \eqref{T^*TN}. Define
\begin{align*}
\Phi_h:C^\infty(N,S^2N)\to C^\infty(N,TN\oplus \rr),&& \Phi_h(k)=(\delta_h k+\tfrac{1}{n}d\Tr_h(k),\Tr_h(k)).
\end{align*}
Evidently, $\ker \Phi_h=T_h^*\mc{T}(N)$.
The formal adjoint of the differential operator $\Phi_h$ is
\[\Phi_h^*(\sigma,f)=\delta^*_h \sigma+(\tfrac{1}{n} d^*\sigma+f)h.\]
Since $h$ is a metric without conformal Killing vector fields,  $\Phi_h^*$ is injective. 
The projector on the kernel of $\Phi_h$ is $P_h:=1-\Phi_h^*(\Phi_h\Phi_h^*)^{-1}\Phi_h$. We claim that $P_{h_0}:T_{h}^*\mc{T}(N)\to
T_{h_0}^*\mc{T}(N)$ is a tame isomorphism. Let us check that it is indeed a tame family 
of 0-th order pseudo-differential operators.  In matrix form, the operator
$\Phi_h\Phi_h^*$ is
\[\Phi_h\Phi_h^*=\begin{bmatrix}
\delta_h\delta_h^*-\tfrac{1}{n}dd^* & 0\\ 0 & n
\end{bmatrix}\]
where $n=\Tr_h(h)$ is the dimension of $N$. 
This operator acts on mixed Sobolev spaces as follows:
$\Phi_h\Phi_h^*:H^s(N,TN)\times H^{s}(N)\to H^{s-2}(N,TN)\times H^s(N)$ for every $s\in\rr$.
The self-adjoint operator $A_h:=\delta_h\delta^*_h-\tfrac{1}{n}dd^*$ is elliptic and invertible and thus has a tame family of 
 pseudo-differential inverses of order $-2$ (see \cite[Section II.3.3]{Ha}). 
Then the inverse of $\Phi_h\Phi_h^*$ is also invertible and tame.
In particular, we see that $P_h$ is smooth tame family of pseudodifferential operators of order $0$.
We have that $P_{h}P_{h_0}:T_{h}^*\mc{T}(N)\to T_{h}^*\mc{T}(N)$ and  
$P_{h_0}P_{h}:T_{h_0}^*\mc{T}(N)\to T_{h_0}^*\mc{T}(N)$ are invertible for $h$ close to $h_0$ in some Sobolev 
since they are the Identity when $h=h_0$ and by Calderon-Vaillancourt theorem. This gives the desired trivialisation.
\end{proof}

To obtain a local description of $T^*\mc{T}(N)$ which is independent of the choice of slice, 
it is necesary and sufficient that for another choice of metric 
$\hat{h}=(f,\phi).h$ in the orbit $\mc{G}_h$, the new representative for $k$ becomes $e^{(2-n)f}(\phi^{-1})^*k$, which is indeed a divergence-free/trace-free tensor with respect to $\hat{h}$. 
In a small neighbourhood of $[h_0]\in \mc{T}(N)$, we can therefore identify $T^*\mc{T}(N)$ with 
the quotient  
\begin{equation}
\label{symquo}
 \mc{G}\backslash\{(h,k)\in \mc{M}(N)\x C^\infty(N,S^2N); \tra_h(k)=0, \delta_h(k)=0\}
 \end{equation}
where the group action of $\cG$  is \eqref{actionsurT*}.
The action of $\cG$ is Hamiltonian, and $T^*\mc{T}(N)$ is the symplectic reduction of $T^*\mc{M}(N)$, 
where the moment map is given at $(f,v)\in \cun(N)\times C^\infty(N,TN)=\mathrm{lie}(\cG)$ in terms of the $L^2$ inner product with respect to $h$ by 
\begin{align*}
\mu_{(f,v)}(k)=\langle -2\tr_h(k),f\rangle+\langle\delta_h(k),v\rangle,&& k\in T^*_h \cM(N).
\end{align*}
Therefore the zero set of the moment map is exactly the space appearing in \eqref{symquo} before quotienting. 
Finally, the symplectic form $\Omega$ descends to $T^*\mc{T}(N)$.

\subsection{Asymptotically hyperbolic Einstein manifolds}\label{AHEM}
The reader can find more details about the theory of  this section in the books \cite{DjGuHe, Ju,FeGr}.
\begin{defi}
Let $\oM^{n+1}$ be a compact smooth manifold with boundary, and $M\subset \oM$ its interior. 
A metric $g$ on $M$ is called \emph{asymptotically hyperbolic Einstein} (or AHE) if $\Ric_g=-ng$ and if there exists a smooth 
boundary defining function $x:\oM\to [0,\infty)$ such that, in a collar neighbourhood of $\pl\bbar{M}$ induced by $x$, $g$ is of the form
\begin{equation}\label{ahe}
g=x^{-2} (dx^2+h_x)
\end{equation}
where $h_x$ is a continuous family of smooth metrics on $N:=\partial M$, depending smoothly on the variable
$x$ when $n$ is odd, and on the variables $x,x^n\log x$ when $n$ is even. The conformal class $[h_0]$ of $h_0$ on $\pl \bbar{M}$ (which is independent of the choice of $x$) is called the \emph{conformal infinity of $(M,g)$}. 
\end{defi}
By a collar neighbourhood induced by $x$ we mean a diffeomorphism $\Phi: [0,\eps)_t\x \pl\bbar{M}\to \bbar{M}$ onto its image, such that
$\Phi^*(x)=t$, $\Phi(0,\cdot)={\rm Id}_{\pl\bbar{M}}$ and the meaning of \eqref{ahe} is $\Phi^*g=(dt^2+h_t)/t^2$ on $(0,\eps)_t\x \pl\bbar{M}$.

In particular, AHE metrics are smooth on $M$ and of class $C^{n-1}$ on $\oM$. In even dimension, the definition with the regularity statement is justified by the result of Chrusciel-Delay-Lee-Skinner \cite{ChDeLeSk}, which states that an Einstein metric on a conformally compact $C^2$ manifold
with smooth conformal infinity admits an expansion at the boundary in integral 
powers of $x$ and $x^n\log x$. We notice that the sectional curvatures of a AHE metric are $-1+\mc{O}(x)$ and that the 
metric $g$ is complete. 

In this paper we will be essentially interested in the more complicated case where $n$ is even (so that the dimension of $M$ is odd) but at the moment
we do not fix the parity of $n$. 

We say that a function $f$ is \emph{polyhomogeneous conormal} (with integral index set) 
on $M$ if it is smooth in $M$ and for all $J\in \nn$, $f$
has an expansion at $\pl\bbar{M}$ of the form: 
\[f =\sum_{j=0}^J\sum_{\ell=0}^{\ell_j}x^{j}\log(x)^{\ell}f_{j,\ell}+o(x^{J})\]
where $f_{j,\ell}\in C^\infty(\pl\bbar{M})$ and $x$ is a smooth boundary defining function. 
The same definition applies to tensors on $M$. There are natural topologies of Fr\'echet space for polyhomogeneous 
conormal functions or tensors; we refer to \cite[Chap 4]{Me1} and \cite{Me} for 
details and properties of these conormal polyhomogeneous spaces. 

\subsection{Poincar\'e-Einstein ends}

There is a weaker notion of metric that will prove useful, that of Poincar\'e-Einstein metrics, introduced by Fefferman-Graham \cite{FeGr}.
Let $(M,g)$ be  an $(n+1)$-dimensional asymptotically hyperbolic Einstein manifold. 
Since by \cite{ChDeLeSk}, the metric $g$ in a collar $(0,\eps)_x\x\pl M$ 
induced by $x$ near $\pl M$ has an expansion of the form 
\begin{align}\label{asyex}
g=\frac{dx^2+h_x}{x^2}, &&
h_x\sim_{x\to 0} \sum_{\ell=0}^\infty h_{x,\ell}(x^{n}\log x)^\ell 
\end{align} 
where $h_{x,\ell}$ are one-parameter families of tensors on $M$ depending smoothly on $x$, 
we want to define the asymptotic version of AHE manifolds: 

\begin{defi}
An \emph{Poincar\'e-Einstein end} is a half-cylinder $\cZ=[0,\eps)\times N$ 
equipped with a smooth metric $g$ on $(0,\eps)\times N$ with an expansion of the form \eqref{asyex} near $x=0$, 
such that $\Ric_g+ng=\mc{O}(x^{\infty})$. If $(\cZ,g)$ is Einstein, we call it an \emph{exact Poincar\'e-Einstein end}. 
\end{defi}

In \cite{FeGr}, Fefferman and Graham analyze the properties of Poincar\'e-Einstein ends.
To explain their results we need the notion of formally determined tensors.

\subsection{Formally determined tensors}

\begin{defi}\label{formeldet}
Let $N$ be a compact manifold, and $m,\ell\in\nn_0$. A map $F:\mc{M}(N)\to C^{\infty}(M,(T^*M)^{\ell})$
from metrics on $N$ to covariant $\ell$-tensors is said to be natural of order $m$ 
(and the tensor $F(h_0)$ is said to be formally 
determined  by $h_0$ of order $m\in \nn$) if there exists a 
tensor-valued polynomial $P$ in the variables 
$h_0$, $h_0^{-1}$, $\sqrt{\det(h_0)}$, $\pl^{\alpha}h_0$ with $|\alpha|\leq m$, so that in any
 local coordinates $y$
\[F(h_0)=P(h_0,h_0^{-1},\sqrt{\det(h_0)},\pl_y^{\alpha}h_0).\]
\end{defi}

\begin{rem}\label{isolocal}
A formally determined tensor $F(h_0)$ is preserved by local isometries: if $\phi: U\to U'$ is a diffeomorphism
where $U,U'$ are open sets of Riemannian manifolds $N,N'$ and $h_0,h_0'$ are metrics on $U,U'$ then 
if $h_0=\phi^*h_0'$ on $U$, we get $F(h_0)=\phi^*F(h_0')$ on $U$. As a consequence, a formally determined tensor is $0$
if it vanishes for all metrics on the sphere $S^{n}$.
\end{rem}

\begin{lem}\label{deriveP}
Let $h_0^t$ be a smooth one-parameter family of metrics on $N$ with $h_0^t=h_0+t\dot{h}_0+\mc{O}(t^2)$ at $t=0$, and let $P(h_0^t),Q(h_0^t)$ be  
tensors formally determined by $h^t_0$ of respective order $p,q$. There exists 
a formally determined tensor $R(h_0)$ in $h_0$ of order $r=p+q$ such that
\[\cjg \pl_tP(h_0^t)|_{t=0},Q(h_0)\cjd_{L^2(N,h_0)} =\cjg \dot{h}_0,R(h_0)\cjd_{L^2(N,h_0)}.\] 
\end{lem}
\begin{proof} 
By using a partition of unity we can assume that $h_0^t$ has support in a coordinate domain. Then $\pl_tP(h_0^t)|_{t=0}$
is a polynomial in the variables $\pl_y^{\beta}\dot{h}_0$, $h_0,h_0^{-1}$, $\sqrt{\det(h_0)}$,   $\pl_y^{\alpha}h_0$, linear in $\dot{h}_0$.
Integrating by parts with respect to the coordinates $y_j$ it is clear that there exists a polynomial $R$ such that
\begin{equation}\label{fordt}
\cjg \pl_tP(h_0^t)|_{t=0},Q(h_0)\cjd_{L^2(N,h_0)} =\cjg \dot{h}_0,R(h_0(y))\cjd_{L^2(N,h_0)}.
\end{equation}
The polynomial $R$ is the same for different coordinate systems. To see that it defines a formally determined tensor, 
we need to prove that the $2$-tensor $R(h_0(y))$ is independent of the coordinate system $y$. This follows from the identity 
\eqref{fordt} since $\dot{h}_0$ is arbitrary, and all the terms except $R(h_0(y))$ are known to be intrinsically defined.
\end{proof}

\begin{prop}[\textbf{Fefferman-Graham} \cite{FeGr}] \label{FeGr}
Let $(\cZ,g)$ be a Poincar\'e-Einstein end. Using the expansion \eqref{asyex}, define
$h_{j}=\tfrac{1}{j!}\pl_x^{j}h_{x,0}|_{x=0}$ and $k:=h_{x,1}|_{x=0}$. 
Then the following hold:
\begin{enumerate}
\item When $n$ is odd, $h_{x,\ell}=0$ when $\ell\geq 1$.
\item The tensors $h_{2j+1}$ are $0$ for $2j+1<n$.
\item The tensors $h_{2j}$ for $j<n/2$ and $k$ are formally determined by $h_0$, of order $2j$.
\item The tensors $h_{2j}$ for $j>n/2$ are formally determined by $h_0$ and $h_n$.
\item The trace ${\rm Tr}_{h_0}(h_n)$ depends only on $h_0$ and defines a formally determined 
function $T_n=T_n(h_0)$ of order $n$, which is zero for $n$ odd.
\item The divergence $\delta_{h_0}(h_n)$ depends only on $h_0$ and not on $h_n$, 
and defines a formally determined tensor $D_n=D_n(h_0)$ of order $n+1$
which is zero for $n$ odd.
\item The tensor $k$, called \emph{obstruction tensor}, is trace- and divergence-free with respect to $h_0$.
\item All coefficients in the Taylor expansion at $x=0$ of $h_{x,\ell}$ for $\ell\geq 1$ are formally determined by 
$h_0$ and $h_n$.
\end{enumerate}
\end{prop}
A consequence of this is the expansion for $h_x$
 \begin{equation}\label{hxexpanded}
h_x=h_0+h_2x^2+\dots+k x^n\log(x)+h_nx^{n}+o(x^n).
\end{equation}

This proposition follows (not directly though) from the decomposition of the Ricci tensor of $g$ in terms of $h_x$ 
in the collar neighbourhood $\mc{Z}$. 
Since we shall use it later, we recall some standard computations of Ricci curvatures on a generalized cylinder, see e.g.\ \cite{bgm}.
On $M:=\rr\times N$ consider a metric $g=dt^2+g_t$ and let 
\begin{align*}
\II:=-\tfrac12 \pt g_t=g_t(W\cdot,\cdot),&&W:=g_t^{-1} \II
\end{align*}
be the second fundamental form, respectively the Weingarten operator. Set $\nu=\pt$ the unit normal vector field to the level hypersurfaces
$\{t=\rm{constant}\}$. Then, for $U,V$ tangent vectors to $N$, the Ricci tensor of $g$ is described by
\begin{equation}\label{ricec}\begin{split}
\Ric_g(\nu,\nu)={}&\tr(W^2)-\tfrac12\tr(g_t^{-1}\pt^2 g_t),\\
\Ric_g(\nu,V)={}&V(\tr(W))+\langle \delta_{g_t}W,V\rangle\\
\Ric_g(U,V)={}&\Ric_{g_t}(U,V)+2\langle W(U),W(V)\rangle-\tr(W)\langle W(U),V\rangle-\tfrac12\pt^2 g_t(U,V).
\end{split}\end{equation}
Using these equations, the Einstein equation $\Ric_g=-ng$ for $g=x^{-2}(dx^2+h_x)$ can be restated 
using the variable $t=e^x$ in terms of the $1$-parameter family of endomorphisms $A_x$ defined by
\begin{align}
A_x:=h_x^{-1}\pl_x h_x={2}{x}^{-1}(1-W) \label{eq:Ax}
\end{align}
as follows:
\begin{align} 
\pl_x\tra(A_x)+\tfrac{1}{2}|A_x|^2={}&{x}^{-1}\tra(A_x),\label{eqEinsteintrace}\\
\delta_{h_x}(\pl_x h_x)={}&-d\tra(A_x),\nonumber\\
x\pl_xA_x+(1-n+\tfrac{1}{2}x\tra(A_x))A_x={}&2xh_x^{-1}{\rm Ric}(h_x)+\tra(A_x){\rm Id}.\nonumber
\end{align}
The same equations are valid modulo $x^\infty$ on asymptotic Poincar\'e-Einstein ends. 

The coefficients in the asymptotic expansion of $h_x$ in \eqref{hxexpanded} near $\{x=0\}$ can be recursively computed from $h_0$ until the $n$-term, and the dependence is local: one has the following formulas
\begin{enumerate}
\item In dimension $n=2$, the obstruction tensor $k$ is $0$, and the coefficient $h_2$ can be any symmetric tensor 
satisfying (see \cite[Th 7.4]{FeGr})
\begin{align} \label{dim2}
\tra_{h_0}(h_2)=-\tfrac12 {\rm Scal}_{h_0}, && \delta_{h_0}(h_2)=\tfrac12 d\,{\rm Scal}_{h_0}.
\end{align}

\item In dimension $n>2$, the tensors $h_2$ is minus the Schouten tensor of $h_0$ and in dimension $n>4$, $h_4$ is expressed in terms of 
Schouten and Bach tensors of $h_0$ (see \cite[Eq (3.18)]{FeGr}): 
\begin{align}\label{formulah2}
-h_2={\rm Sch}_{h_0}:=\tfrac{1}{n-2}\left(\Ric_{h_0}-\tfrac{1}{2(n-1)}{\rm Scal}_{h_0}h_0\right) &&
h_4=\tfrac{1}{4}\left( h_2^2-\tfrac{1}{n-4}B_{h_0}\right)
\end{align}
where $B_{h_0}$ is the Bach tensor of $h_0$ if $n>4$   
and $h_2^2(\cdot,\cdot):=h_0(H_2^2\cdot,\cdot)$ if $H_2$ is the endomorphism of $TN$ defined by $h_2=h_0(H_2\cdot,\cdot)$.\\

\item In dimension $n> 4$, when $h_0$ is locally conformally flat, one has  $k=0$ and 
\begin{align}\label{formulalcf}
-h_2={\rm Sch}_{h_0}, && h_4=\tfrac{1}{4} h_2^2, && h_{2j}=0  \,\,  \textrm{ for }\,2<j<\tfrac{n}{2}.
\end{align}
See \cite[Th 7.4]{FeGr} or \cite{SkSo} for a proof. When $h_n=0$, then  the metric $g=x^{-2}(dx^2+h_x)$ 
has constant sectional curvature $-1$ in a small neighbourhood of $x=0$ if $h_x=h_0+x^2h_2+x^4h_4$ 
with $h_2,h_4$ of \eqref{formulalcf}. When $n=4$, one still has $h_2=-{\rm Sch}_{h_0}$
but $h_4$ is not necessarily $h_2^2$.\\

\item When $h_0$ is an Einstein metric with  ${\rm Ric}_{h_0}=\la (n-1)h_0$, it is easily checked that $k=0$ and 
\begin{align}\label{einsteinexp}
h_2=-\tfrac{\la}{2}h_0, && h_4:= \tfrac{\la^2}{16}h_0, && h_{2j}=0 \,\, \textrm{ for } 2<j<\tfrac{n}{2} . 
\end{align}
When $h_n=0$, the metric  $g=(dx^2+h_x)/x^2$ with $h_x:=(1- \frac{\la x^2}{4})^2h_0$
is an exact Poincar\'e-Einstein end in $x<x_0$ for some small $x_0>0$, see Section \ref{ssc:fuchsian}.
\end{enumerate}

\subsection{The conformal class at infinity}

By \cite{GrLe,ChDeLeSk}, the whole conformal class $[h_0]$ of the metric $h_0$ induced by $g$ on the boundary at infinity (with respect to a given boundary defining function $x$) can be parametrized by a family of ``geodesic'' boundary defining functions: 
\begin{lem}\label{geodesicbdf}
Let $(M,g)$ be an odd dimensional AHE manifold, of the form \eqref{ahe} near $\pl \bbar{M}$ for some $x$. 
Let $h_0$ be the induced metric at infinity. For any $\hat{h}_0\in [h_0]$, there is a neighborhood $V$ of $\pl\bbar{M}$ and a unique boundary defining function $\hat{x}$
such that $\hat{x}^2g|_{T\pl M}=\hat{h}_0$ and $|d\hat{x}|_{\hat{x}^2g}=1$ in $V$. The function $\hat{x}$ has a polyhomogeneous expansion with respect to $x$ and the metric $g$ is of the form $(d\hat{x}^2+\hat{h}_{\hat{x}})/\hat{x}^2$ in a collar near $\pl M$, where 
$\hat{h}_{\hat{x}}$ is a one-parameter family  of tensors on $\pl\bbar{M}$ which is smooth in $\hat{x},\hat{x}^n\log(\hat{x})$.
\end{lem}
\begin{proof} The existence and polyhomogeneity of $\hat{x}$ is shown in \cite[Lemma 6.1]{ChDeLeSk}.
The form of the metric in the collar neighborhood induced by $\hat{x}$ follows for instance from Theorem A in \cite{ChDeLeSk}. 
Since it will be used later, we recall that the proof amounts to  
seting $\hx=e^\omega x$ for some unknown function $\omega$ defined on $\bbar{M}$ near $N=\pl\bbar{M}$ which solves
near the boundary the equation $|d\hx|^2_{{\hx}^2h}=1$ with $\hat{h}_0=e^{2\omega_0}h_0$. This leads to the following  Hamilton-Jacobi equation in the collar neighbourhood $[0,\eps)\x N$ of the boundary:
\begin{align} 
\label{ecom}
\pl_x\omega+\tfrac{x}{2}\left((\pl_x\omega)^2+|d_N\omega|^2_{h_x}\right)=0,&& \omega|_{N}=\omega_0.
\end{align}
where $d_N$ is the de Rham differential on $N$.
%
\end{proof}
Geometrically, the function $\hat{x}$ corresponding to $\hat{h}_0$ yields a particular foliation by hypersurfaces $\{x={\rm const}\}$ diffeomorphic to $N$ 
near infinity, induced by the choice of conformal representative at infinity.

\subsection{Cauchy data for Einstein equation, non-linear Dirichlet-to-Neumann map} 

By Proposition \ref{FeGr}, a Poincar\'e-Einstein end is uniquely determined modulo $\mc{O}(x^\infty)$.
There is in fact a stronger statement proved by Biquard \cite{Bi}, based on unique continuation for elliptic equations: 

\begin{prop}[\textbf{Biquard}]\label{bi} 
An exact Poincar\'e-Einstein end $([0,\eps)_x\x N,g=\frac{dx^2+h_x}{x^2})$ is uniquely determined 
by the data $(h_0,h_n)$ where $h_x=\sum_{j=0}^{n/2}x^{2j}h_{2j}+kx^{n}\log x+o(x^n)$.
\end{prop}
On a manifold with boundary $\bbar{M}$, the unique continuation of \cite{Bi} also holds true: 
if two AHE metrics on  $M$ agree to infinite order at $\pl \bbar{M}$, 
then, near the boundary, one is the pull back of the other by a diffeomorphism of $\bbar{M}$ which is the identity on $\pl\bbar{M}$.

We will then call $(h_0,h_n)$ the \emph{Cauchy data} for the Einstein equation, 
\begin{align}\label{Dirichlet/neumann}
h_0 \textrm{ is the \emph{Dirichlet datum}} ,&& h_n \textrm{ is the \emph{Neumann datum}}.
\end{align} 
We emphasize that here the pair $(h_0,h_n)$ is associated to the 
geodesic boundary function of Lemma \ref{geodesicbdf} determined by $h_0$.

It is of interest to study those pairs $(h_0,h_n)$ for which there does exist an AHE manifold $(M,g)$ which 
can be written in a collar neighbourhood $[0,\eps)_x\x \pl M$ under the form $g=\frac{dx^2+h_x}{x^2}$   
with $h_x=\sum_{j=0}^{n/2}x^{2j}h_{2j}+kx^{n}\log x+o(x^n)$.

We can define a Dirichlet-to-Neumann map under the assumption that a local existence result for the following Dirichlet 
problem on $M$ holds: 
let $g^0$ be an AHE metric on $M$ and $h_0=(x^2g^0)_{|TN}$ be a representative of the conformal infinity of $g$ associated to 
a geodesic boundary defining function $x$, then  
there exists a smooth submanifold $\mc{S}\subset \mc{M}(N)$ containing $h_0$ (with $N=\pl \bbar{M}$), transverse to 
the action of $C^\infty(N)$ on $\mc{M}(N)$, 
such that for any $h\in \mc{S}$, there is an AHE metric $g$ near $g^0$ such that
\begin{align} \label{localpb}
\Ric_g=-ng,&&(x^2g)_{|\pl \bbar{M}}=h 
\end{align}
and $g$ depends smoothly on $h$. 
The topology here can be chosen to be some  $C^{k,\alpha}(M)$ norms for some $k\in \nn$ and $\alpha>0$. 
Such an existence result has been proved by Graham-Lee \cite{GrLe} when $(M,g^0)=(\hh^{n+1},g_{\hh^{n+1}})$ where $\hh^{n+1}$ is viewed as the unit ball in $\rr^{n+1}$, and has been extended by Lee \cite{lee:fredholm} to the case where $g^0$ is AHE with negative sectional curvatures. 
We can then define a (local) \emph{Dirichlet-to-Neumann map}\footnote{In even dimension $n$, we will see later that it is more natural to modify $h_n$ with a certain formally determined 
tensor in the definition of $\mc{N}$.} near $h_0$
\begin{align}
\mc{N}:  C^\infty(M,S^2_+T^*\pl\bbar{M})\to C^\infty(M,S^2T^*\pl\bbar{M}),&&  h\mapsto h_n.
\end{align}
where $h_n$ is the Neumann datum of the metric $g$ satisfying \eqref{localpb}.
Graham \cite{GrOb} computes its linearization at the hyperbolic metric 
in the case $n$ odd and when $(M,g^0)=(\hh^{n+1},g_{\hh^{n+1}})$. For $n$ odd, this was also studied by Wang 
\cite{Fawa} in a general setting: she proved that this linearized operator is a pseudo-differential 
operator on the boundary and she computed its principal symbol.

\section{The renormalized volume in a fixed conformal class}\label{sectionvol}

\subsection{The renormalized volume} \label{ssc:renormalized}

We follow the method introduced by Henningson-Skenderis \cite{HeSk}, Graham \cite{Gr0}.
The volume form near the boundary is 
\[\dvol_g=v(x)d{\rm vol}_{h_0}\frac{dx}{x^{n+1}}=\det(h_0^{-1}h_x)^{\frac12}\dvol_{h_0}\frac{dx}{x^{n+1}}.\]
Since $\tra(k)=0$,
the function $v(x)$ has an asymptotic expansion of the form
\begin{equation}\label{v2i}
v(x)=1+v_2x^2+\dots+v_nx^n+o(x^n).
\end{equation}

\begin{defi}
The renormalized volume of $(M,g)$ with respect to a conformal representative $h_0$ of $[h_0]$ is the Hadamard regularized integral
\begin{equation}\label{defrvol}
\rvol(M,g;h_0)= {\rm FP}_{\eps\to 0}\int_{x>\eps}\dvol_g.
\end{equation}
where, near $\pl M$, $x$ is the geodesic boundary defining function such that $x^2g|_{T\pl M}=h_0$. When $g$ 
is fixed and we consider $\rvol(M,g;h_0)$ as a function of $h_0$, we shall write it $\rvol(M;h_0)$.
\end{defi}

An equivalent definition for $\rvol$ was given by Albin \cite{Al} 
using Riesz regularization
\begin{align}\label{riesz}
\rvol(M,h_0)={\rm FP}_{z=0}\int_{M}x^z\dvol_g, && z\in\cc
\end{align}
where $x$ is any positive function equal to the geodesic boundary-defining function associated to $h_0$ near $\pl M$.
With this definition
we can easily compute the variation of $\rvol$ inside the conformal class $[h_0]$. 
\begin{prop}
Let $(M,g)$ be an odd dimensional Einstein conformally compact manifold with conformal 
infinity $[h_0]$. The renormalized volume $\rvol(M,\cdot)$ of $M$, as a functional 
on $\mc{M}_{[h_0]}:=\{h_0\in [h_0]; \int_{\pl M}\dvol_{h_0}=1\}$, admits 
a critical point at $h_0$ if and only if $v_n(h_0)$ is constant.
\end{prop}
\begin{proof}
We set $h_0^s:=h_0e^{2s\omega_0}$ for $s\geq 0$, then from Lemma \ref{geodesicbdf} there exists a unique function $\omega^s$ such that
the geodesic boundary defining function $x_s$ associated to $h_0^s$ is given by 
\begin{align}\label{xs}
x^s=e^{\omega^s}x,&& \omega^s=s\omega_0+\mc{O}(xs^2).
\end{align}
Indeed, for all $s$ we have $|d\log(x^s)|^2_g=1$, thus $\omega^s$ must satisfy
\begin{align*} 
2\pl_x\omega^s=-x((\pl_x\omega^s)^2+|d_y\omega^s|_{h_x}^2),&& \omega^s|_{x=0}=s\omega_0.
\end{align*}
This is a non-characteristic Hamilton-Jacobi equation which has a unique solution depending smoothly in $s$ on the initial data with
$\omega^0=0$.  Then $\omega^s=\mc{O}(s)$ and thus $\pl_x\omega^s=\mc{O}(xs^2)$, which implies 
that \eqref{xs} holds.
Taking the derivative of  \eqref{riesz} at $s=0$, we obtain using the expansion \eqref{v2i}
\begin{equation}\label{plsvol}
\pl_s\rvol(M,h_0^s)|_{s=0}={\rm FP}_{z=0}\int_{M}z\omega_0x^z v(x)\dvol_{h_0}\frac{dx}{x^{n+1}}=
\int_{\pl M}\omega_0 v_n\dvol_{h_0}.\end{equation}
We now make a variation within constant volume metrics in $[h_0]$, thus $\int_{\pl M}\omega_0 \dvol_{h_0}=0$.
We thus conclude that 
\begin{equation}\label{vnconst}
v_n={\rm constant}
\end{equation}
is the equation describing a critical point of the renormalized volume functional
in the conformal class with constant total volume. 
\end{proof}

\begin{rem} \label{rk:GZ}
From Graham-Zworski \cite{GrZw}, the following identity holds
\begin{equation}\label{grz}
\int_{\pl M}v_n\dvol_{h_0}=C_n\int_{\pl M}Q_n\dvol_{h_0}, \end{equation}
where $C_n$ is an explicit constant and  $Q_n$ is Branson's Q-curvature. This integral depends 
only on the conformal class $[h_0]$ and not on $h_0$. For locally conformally flat metrics, 
this is a constant times the Euler characteristic, as proved  by Graham-Juhl \cite{GrJu}.
\end{rem}

\begin{rem} 
According to Graham-Hirachi \cite{GrHi}, the infinitesimal variation of the integral of $v_n$ 
along a $1$-parameter family of Poincar\'e-Einstein metrics $g_s$ inducing $h_0^s$ on $N$ with $\dot{h}_0:=\pl_s(h^s_0)_{|s=0}$ is determined by the obstruction tensor $k$ of $h_0$:
\begin{equation}\label{grhi}
\pl_s\left(\int_{\pl M}v_n\,{\rm dvol}_{h_0}\right)_{|s=0}=\tfrac{1}{4}\int_{\pl M}\cjg k,\dot{h}_0\cjd \dvol_{h_0}.
\end{equation}
\end{rem}

In fact, we can give a formula for the renormalized volume $\rvol(M,e^{2\omega_0}h_0)$ in terms of $\omega_0$.

\begin{lem}\label{volumeformula}
Let $h_0\in [h_0]$ be fixed, $\omega_0\in C^\infty(\pl M)$, and let \[\omega=\sum_{j=0}^{\ndemi}\omega_{2j}x^{2j}+\mc{O}(x^{n+1})\]
be the solution of the Hamilton-Jacobi equation $|dx/x+d\omega|^2_{g}=1$ near $\pl M$ 
with boundary condition $\omega|_{\pl M}=\omega_0$. The renormalized volume 
$\mc{V}_n(\omega_0):=\rvol(M,e^{2\omega_0}h_0)$  as a function of $\omega_0$ is given by 
\[\mc{V}_n(\omega_0)=\mc{V}_n(0)+ \int_{\pl M} \sum_{i=0}^{n/2} v_{2i}(h_0)\omega_{n-2i} \,{\rm dvol}_{h_0}
\]
where $v_{2i}(h_0)\in C^\infty(\pl M)$ are the terms in the expansion of the volume element \eqref{v2i} at $\pl M$.
\end{lem}
\begin{proof} 
From the expansion $e^{z\omega}=1+z\omega+\mc{O}(z^2)$ near $z=0$, we get
\[\begin{split}
\rvol(M,e^{2\omega_0}h_0)={} &{\rm FP}_{z=0}\int_{M}x^{z-n}e^{z\omega}v(x)\frac{dx}{x}{\rm dvol}_{h_0}\\
 ={}& \rvol(M,h_0)+{\rm FP}_{z=0}\Big(z\int_{M}x^{z-n}\omega(x) v(x)\frac{dx}{x}{\rm dvol}_{h_0}\Big)\\
 ={}& \rvol(M,h_0) + {\rm Res}_{z=0}\int_{M}x^{z-n}\omega(x) v(x)\frac{dx}{x}{\rm dvol}_{h_0}\\
={}&  \rvol(M,h_0) + \int_{\pl M} \sum_{i=0}^{n/2} v_{2i}\omega_{n-2i} \,{\rm dvol}_{h_0}
  \end{split}\] 
where in the last equality we have exhibited the residue as the coefficient of $x^{n}$ in $\omega(x) v(x)$
\end{proof}
We mention a similar statement after Theorem 3.1 in \cite{Gr0}.

Let us now give some properties of the $\omega_{2i}$ in the expansion of $\omega(x)$ at $x=0$:
\begin{lem}\label{expomega}
The function $\omega$ solving the equation $|dx/x+d\omega|_{g}=1$ near $x=0$ and $\omega|_{x=0}=\omega_0$
satisfies $\omega(x)=\sum_{i=0}^{n/2}x^{2i}\omega_{2i}+o(x^n)$ for some  $\omega_{2i}\in C^\infty(M)$ with
\begin{align*}
\omega_2={}&-\tfrac14 |\nabla\omega_0|^2_{h_0}\\
\omega_4={}&\tfrac18\left( -\frac14 |\nabla \omega|^4 +h_2( \nabla\omega_0,\nabla\omega_0) 
-2h_0(\nabla\omega_0,\nabla\omega_2)\right).
\end{align*}
where $\nabla$ denotes the gradient with respect to $h_0$.
If we replace $\omega_0$ by $s\omega_0$ for $s>0$ small, for all $i>0$ one has as $s\to 0$
\begin{equation}\label{omega2i}
\omega_{2i}= -\frac{s^2}{4i}h^{(2i-2)}(d\omega_0,d\omega_0)+\mc{O}(s^3).
\end{equation}
where $h_x^{-1}=\sum_{i=0}^{n/2}x^{2i}h^{(2i)}+O(x^n\log x)$ is the metric induced by $h_x$ on the cotangent bundle $T^*\pl M$.
\end{lem}
\begin{proof} The computation for $\omega_2$ and $\omega_4$ is simply obtained by expanding in powers of $x$ 
the equation $2\pl_x\omega=-x((\pl_x\omega)^2+|d_y\omega|_{h(x)}^2)$ and identifying the terms:
\[\sum_{i=0}^{n/2}4ix^{2i-1}\omega_{2i}=-x\Big(\sum_{i=0}^{n/2}2ix^{2i-1}\omega_{2i}\Big)^2-
\sum_{i,j,k=0}^{n/2}x^{2(i+j+k)+1}h^{(2i)}(d\omega_{2j},d\omega_{2k})+o(x^{n-1})\]
where $h_x^{-1}=\sum_{i=0}^{n/2}x^{2i}h^{(2i)}+O(x^n\log x)$ 
if $h_x^{-1}$ is the metric on the cotangent bundle. In particular, we have $h^{(2)}(d\omega_{2k},d\omega_{2j})=-h_2(\nabla\omega_{2k},\nabla\omega_{2j})$. Now for \eqref{omega2i}, we observe that $\omega_{2i}=\mc{O}(s^2)$ for each $i\not =0$, and 
so by looking at the terms modulo $s^3$ in the equation above, only the terms with $j=k=0$ appear and we get 
\[\sum_{i=0}4ix^{2i-1}\omega_{2i}= -s^2\sum_{i=0}^{n/2}x^{2i+1}h^{(2i)}(d\omega_0,d\omega_0)+\mc{O}(s^3)\] 
which implies the desired identity.
\end{proof}

From this, we can give an expression for the Hessian of $\omega_0\mapsto \rvol(M,e^{2\omega_0}h_0)$ at a critical point $h_0$, as a quadratic form of $\omega_0$. 
\begin{cor}\label{hessian}
Let $(M,g)$ be an $(n+1)$-dimensional Poincar\'e-Einstein manifold with conformal infinity $(\pl M,[h_0])$. Then for $\omega_0\in C^\infty(M)$ we have 
\[\Hess_{h_0}(\mc{V}_n)(\omega_0):=\pl_s^2 {\rm Vol}_R(M,e^{2s\omega_0}h_0)_{|s=0}=-\sum_{j=1}^{n/2}\int_{\pl M}\frac{v_{n-2j}(h_0)}{2j}h^{(2j-2)}(d\omega_0,d\omega_0){\rm dvol}_{h_0}\]
where $h_x^{-1}=\sum_{i=0}^{n/2}x^{2j}h^{(2j)}+O(x^n\log x)$ is the metric induced by $h_x$ on the cotangent bundle $T^*\pl M$, 
and $v_{2j}(h_0)\in C^\infty(\pl M)$ are the coefficients in the expansion \eqref{v2i} of the volume element at $\pl M$.
\end{cor}
Remark that the Hessian of $\mc{V}_n$  depends only on the conformal infinity $(\pl M,[h_0])$ of $M$. 
Since the positive/negative definiteness of the Hessian of $\mc{V}_n=\rvol$ is entirely characterized by the tensor $-\sum_{j=1}^{n/2}\int_{\pl M}\frac{v_{n-2j}(h_0)}{2j}h^{(2j-2)}$ we shall call this tensor the Hessian of $\mc{V}_n$ at $h_0$ and denote it 
\begin{equation}\label{hessH}
{\rm Hess}_{h_0}(\mc{V}_n)=-\sum_{j=1}^{n/2}\tfrac{1}{2j}v_{n-2j}(h_0)h^{(2j-2)}.
\end{equation}

\begin{rem}\label{symmetric}
We remark that the tensors $h^{(2j-2)}$ are symmetric tensors on $T^*\pl M$ and thus 
${\rm Hess}_{h_0}(\mc{V}_n)$ is also symmetric.
Robin Graham informed us that such a formula for the Hessian can also be obtained from his variation formula for $v_n$ 
computed in Theorem 1.5 of \cite{Gr}.
 \end{rem}
 
\subsection{Computations of $v_2,v_4,v_6$}

To express the renormalized volume functional in dimension $2,4,6$, we need to compute the volume coefficients $v_2,v_4,v_6$.
This will serve also later for the variation formula for the renormalized volume of AHE metrics. The formulas are 
already known \cite{Gr} (see also \cite[Th 6.10.2]{Ju} for a proof) but to be self-contained we give a couple of details of how 
the computations go. 
We recall first that for a symmetric endomorphism $A$ on an $n$-dimensional vector space equipped 
with a scalar product, the elementary symmetric 
function of order $k$ of $A$ is defined by
\begin{equation}\label{sigmak}
 \sigma_k(A)=\sum_{i_1<\dots<i_k}\la_{i_1}\dots \la_{i_k}
 \end{equation}
where $(\la_1,\dots, \la_n)$ are the eigenvalues of $A$ repeated with multiplicities.

\begin{lem}\label{v2v4}
Let $((0,\eps)_x\x N, g=\frac{dx^2+h_x}{x^2})$ be an asymptotic Poincar\'e-Einstein end, and $H_{2j}$, $K$ the endomorphisms of $TN$ defined by
\[ h_x(\cdot,\cdot)=h_0\Big( \Big(\sum_{j=0}^{n}H_{2j}x^{2j}+Kx^n\log(x)\Big)\cdot,\cdot\Big)+o(x^n).\] 
If $v_{2j}$ are the volume coefficients in \eqref{v2i}, one has
\begin{align*}
v_2={}&\tfrac{1}{2}\sigma_1(H_2)=\tfrac12 {\rm Tr}(H_2), \\
v_4={}&\tfrac{1}{4}\sigma_2(H_2)=\tfrac{1}{8}( {\rm Tr} (H_2)^2-{\rm Tr}(H_2^2))  \\
v_6={}&\tfrac{1}{8}\sigma_3(H_2)+\tfrac{1}{24(n-4)}\cjg B_{h_0},h_2\cjd,
\end{align*}
where $h_0(H_2\cdot,\cdot)=h_2=-{\rm Sch}_{h_0}$. In addition, we have
\begin{align}\label{traceH_4}
 4\tra(H_4)-\tra(H_2^2)=0, && 6\tra(H_6)-4\tra(H_2H_4)+\tra(H_2^3)=0.
 \end{align}
\end{lem}

\begin{proof} 
From \eqref{eq:Ax} 
we obtain modulo $\mc{O}(x^6)$
\[\begin{split}
A_x 
={}& 2xH_2+x^3(4H_4-2H_2^2)+ x^{5}(6H_6-6H_2H_4+2H_2^3) +x^{n-1}K(n\log(x)+1)
\end{split}\]
Taking the trace and using that the obstruction tensor is trace-free (ie. $\tra(K)=0$), we get modulo $\mc{O}(x^6)$
\[\begin{split}
\tra(A_x)={}& 2x\tra(H_2)+x^3(4\tra(H_4)-2\tra(H^2_2))+6x^{5}(\tra(H_6)-\tra(H_2H_4)+\tfrac{1}{3}\tra(H_2^3))\\
\tfrac{1}{2}x|A_x|^2={}& \tfrac{1}{2}x\tra(A_x^2)=2x^3\tra(H_2^2)+4x^5(2\tra(H_2H_4)-\tra(H_2^3))+\mc{O}(x^6).
\end{split}\]
Now from \eqref{eqEinsteintrace}, we obtain \eqref{traceH_4}.
We can expand the volume form (using the expansion of determinant in traces) modulo $\mc{O}(x^7)$ and use \eqref{traceH_4}
\[\begin{split}
\det(h_0^{-1}h_x)
={} & 1+x^2\tra(H_2)+x^4\Big(-\tfrac{1}{4}\tra(H_2^2)+\tfrac{1}{2}(\tra(H_2))^2\Big)\\
 & +x^{6}\Big(\tfrac{1}{6}\tra(H_2^3)-\tfrac{1}{3}\tra(H_2H_4)+\tfrac{1}{6}(\tra(H_2))^3+\tfrac{1}{4}\tra(H_2)\tra(H_2^2)\Big)
\end{split}\] 
thus taking the square root and using the expression of $H_4$ given by \eqref{formulah2}, we obtain the desired formula for $v_2,v_4,v_6$.
\end{proof}
 
\begin{rem}
If $h_0$  is a locally conformally flat metric on $N$, the expression of $v_{2j}(h_0)$ 
has been computed by Graham-Juhl \cite{GrJu}: they obtain 
\begin{align}\label{lcf} 
v_{2j}(h_0)=2^{-j}\sigma_{j}(H_2),&& h_2(\cdot,\cdot)=h_0(H_2\cdot,\cdot)=-{\rm Sch}_{h_0}(\cdot,\cdot).
\end{align}  
\end{rem}

\subsection{The renormalized volume in dimension $n=2$}
 
Combining Lemma \ref{volumeformula} with Lemma \ref{expomega} and Lemma \ref{v2v4}, we obtain:

\begin{prop}\label{functionaldim2}
The renormalized volume functional $\mc{V}_2(\omega_0)={\rm Vol}_{R}(M,e^{2\omega_0}h_0)$ 
on the conformal class $[h_0]$ in dimension $2$ is given by the expression
\begin{align*}
\mc{V}_2(\omega_0)=\mc{V}_2(0)-\tfrac{1}{4}\int_{\pl M}(|\nabla \omega_0|^2_{h_0}+
{\rm Scal}_{h_0}\omega_0) {\rm dvol}_{h_0}.
\end{align*}
Its Hessian at $h_0$ is  
$\Hess_{h_0}(\mc{V}_2)=-\demi h_0^{-1}$.
\end{prop}
The critical points of the functional $\mc{V}_2$ restricted to the set
\[\{\omega_0\in C^\infty(\pl M); \int_{\pl M}e^{2\omega_0}{\rm dvol}_{h_0}=1\}\]
are the solutions of the equation ${\rm Scal}_{e^{2\omega_0}h_0}=4\pi \chi(\pl M)$.
We notice that this is the usual functional for uniformizing surfaces, that is, of finding the 
constant curvature metrics in the conformal class as critical points. When $\chi(\pl M)<0$, there is existence and uniqueness 
of critical points by strict convexity of the functional (see e.g \cite{Tay}). The renormalized volume is maximized at the 
hyperbolic metric in the conformal class.

It is instructive to recall here the Polyakov formula for the regularized determinant of the Laplacian (see e.g. \cite[Eq (1.13)]{OsPhSa})
\[ 3\pi \log({\det}'\Delta_{e^{2\omega_0}h_0})-3\pi \log({\det}'\Delta_{h_0})=-\tfrac{1}{4}\int_{\pl M}(|\nabla \omega_0|^2_{h_0}+
{\rm Scal}_{h_0}\omega_0) {\rm dvol}_{h_0}. \]
As a consequence, we deduce 
\begin{lem}
Let $(N,[h_0])$ be a closed compact Riemann surface, 
and let $M$ be a Poincar\'e-Einstein manifold  with conformal infinity $(N,[h_0])$. Then the functional
\begin{align*} 
F_M: [h_0]\to \rr ,&& h\mapsto {\det}'(\Delta_h)\exp\Big({-\frac{{\rm Vol}_R(M,h)}{3\pi}}\Big)
\end{align*}
is constant. 
\end{lem}
The constant $F_M([h_0])$, which depends on $M$ and $[h_0]$, is computed 
by Zograf \cite{Zo} for the case where $M$ is a Schottky 
$3$-manifold: $M$ is a handlebody, its interior is equipped with a complete hyperbolic metric 
and the space of conformal classes $[h_0]$ on the conformal infinity $\pl M$ is identified to 
the Teichm\"uller space $\mc{T}_{\pl M}$ of $\pl M$. The function $F_M: \mc{T}_{\pl M}\to \rr^+$ can be expressed in terms of 
a period matrix determinant on $\pl M$ and the modulus of a holomorphic function on $\mc{T}_{\pl M}$.

\subsection{The renormalized volume in dimension $n=4$}
 
Combining  Lemma \ref{volumeformula}, Lemma \ref{expomega} and Lemma \ref{v2v4}, we obtain an explicit formula for the functional 
\begin{align*}
\mc{V}_4:\cun(\pl M)\to \rr,&& \mc{V}_4(\omega_0):= \rvol(M,h_0e^{2\omega_0}).
\end{align*}
\begin{prop}\label{functionaldim4}
The renormalized volume functional $\mc{V}_4$ on the conformal class $[h_0]$
in dimension $4$ is given by the expression
\begin{align*}
\mc{V}_4(\omega_0)=\mc{V}_4(0)+\int_{\pl M}&[\tfrac14 \sigma_2(H_2)\omega_0+\tfrac18 (h_2-\Tr_{h_0}(h_2)h_0)(\nabla \omega_0,\nabla\omega_0)\\
&+\tfrac{1}{16}\Delta_{h_0}\omega_0 .|\nabla\omega_0|_{h_0}^{2}
-\tfrac{1}{32}|\nabla \omega_0|_{h_0}^4] {\rm dvol}_{h_0}
\end{align*}
where $h_2=h_0(H_2\cdot,\cdot)$. Its Hessian at $h_0$ is given by 
\begin{align*}
\Hess_{h_0}(\mc{V}_4)(\omega_0)={}&\tfrac14 \int_{\pl M}(h_2(\nabla \omega_0,\nabla\omega_0)-\Tr_{h_0}(h_2)|\nabla\omega_0|_{h_0}^2)
{\rm dvol}_{h_0}\\
={}&-\tfrac{1}{8}\int_{\pl M}(\Ric_{h_0}-\tfrac12\scal_{h_0}h_0)(\nabla \omega_0,\nabla\omega_0) {\rm dvol}_{h_0}.
\end{align*}
\end{prop}

The critical points of the functional $\mc{V}_4$ restricted to the set
\[\left\{\omega_0\in C^\infty(\pl M); \int_{\pl M}e^{4\omega_0}{\rm dvol}_{h_0}=1\right\}\]
are, as we have seen, the solutions of the equation
\[\sigma_2({\rm Sch}_{e^{2\omega_0}h_0})=4\int_{\pl M} v_4 \,{\rm dvol}_{h_0}=C_4 \Big(4\pi^2\chi(\pl M)-\tfrac12 \int_{\pl M}|W|_{h_0}^2\Big)\]
with $C_4$ is an universal constant, $\chi(M)$ the Euler characteristic, $W$ the Weyl tensor of $h_0$, and ${\rm Sch}_{h_0}$
the Schouten tensor.

\section{Metrics with $v_n$ constant}

Equations of the type $v_k={\rm constant}$ appeared first in the work of Chang-Fang \cite{ChFa}, who proved that for 
$k<n/2$, these equations are variational.
We will exhibit some cases where the equation $v_n={\rm constant}$ has solutions. We shall consider either $n\leq 4$ or 
perturbations of computable cases, typically conformal classes containing Einstein manifolds or locally conformally flat manifolds.

First let us give an expression for the linearisation of $v_n$ in the conformal class. 
\begin{lem}\label{dvn}
Let $h_0$ be a smooth metric, then for any $\omega_0\in C^\infty(M)$
\[\pl_s(e^{ns\omega_0} v_n(e^{2s\omega_0}h_0))|_{s=0}= d^*_{h_0}(H_{h_0}(d\omega_0))\]
where $H_{h_0}\in C^\infty(N,{\rm End}(T^*N))$ is 
defined by $h_0^{-1}(H_{h_0}\cdot,\cdot)={\rm Hess}_{h_0}(\mc{V}_n)(\cdot,\cdot)$, 
using the notation \eqref{hessH}. 
\end{lem} 
\begin{proof}
Let $(M,g)$ is a Poincar\'e-Einstein manifold with conformal infinity $[h_0]$, then 
we have seen from \eqref{plsvol} that $\pl_s({\rm Vol}_R(M,e^{2s\omega_0}h_0))=
\int_{N}v_n(e^{2s\omega_0}h_0)\omega_0{\rm dvol}_{e^{2s\omega_0} h_0}$  thus
\[\begin{split}
\pl_s^2({\rm Vol}_R(M,e^{2s\omega_0}h_0))|_{s=0}= \int_{N}\pl_s(v_n(e^{2s\omega_0}h_0))|_{s=0}
\omega_0{\rm dvol}_{e^{2\omega_0} h_0}
 +n\int_{N}v_n(h_0)\omega_0^2{\rm dvol}_{h_0}.
\end{split}\]
We therefore have
\begin{equation}
\int_{N}\pl_s(v_n(e^{2s\omega_0}h_0))|_{s=0}\omega_0{\rm dvol}_{h_0}=
\int_{N}{\rm Hess}_{h_0}(\mc{V}_n)(d\omega_0,d\omega_0)-nv_n(h_0)\omega_0^2{\rm dvol}_{h_0}.
\end{equation}
Using the symmetry of the tensor ${\rm Hess}_{h_0}(\mc{V}_n)$ as mentionned in Remark \ref{symmetric}, 
this quadratic form can be polarized and this provides the desired expression for the linearisation of $v_n$. 
\end{proof}

This Lemma suggests that $v_n(e^{\omega_0}h_0)$ depends only on derivatives of order $2$ of $\omega_0$. In fact Graham \cite[Th. 1.4]{Gr} proved a stronger statement, 
namely that $v_n(h_0)$ depends only on two derivatives of $h_0$.

Using the Nash-Moser implicit function theorem we can deal with perturbations of model cases for which we know that $v_n$ is constant.

\begin{prop} \label{pr:deformation}
Let $N$ be an $n$-dimensional  compact manifold with a conformal class $[h_0]$ admitting a representative $h_0$ with $v_n(h_0)=\int_{N}v_n(h_0){\rm dvol}_{h_0}$. 
Assume that ${\rm Hess}(\mc{V}_n)$ is a positive (resp.\ negative) definite tensor at $h_0$ and that the quadratic form
\begin{equation}\label{quadratic}
f \mapsto \int_N \left({\rm Hess}_{h_0}(\mc{V}_n)(df,df)-n v_n(h_0) f^2 
\right){\rm dvol}_{h_0} 
\end{equation} 
is non-degenerate on $C^\infty(N)$.
Then there is a neighbourhood $U_{h_0}\subset \mc{M}(N)$ of $h_0$ such that  
\[ \mc{S}:= \{h\in U_{h_0}; v_n(h)=\int_{N}v_n(h){\rm dvol}_h\}\]
is a slice at $h_0$ for the conformal action of $C^\infty(N)$ as defined in \eqref{slice1}.
\end{prop}
\begin{proof} We shall use the Nash-Moser implicit function theorem. 
We first take a slice $S_{h_0}$ at $h_0$ for the conformal action, in order to view a neighbourhood 
$U_{[h_0]}\subset \mc{C}(N)$ of $[h_0]$ as a Fr\'echet submanifold of $\mc{M}(N)$ and 
a neighbourhood $U_{h_0}$ in $\mc{M}(N)$ as a product space $S_{h_0}\x C^\infty(N)$:
for instance, take the open subset of Fr\'echet space 
\[B_{h_0}=\{r\in C^\infty(N,S^2N); \tra_{h_0}(r)=0, \sup_{m\in N} |r(m)|_{h_0}<1\},\] 
the map
\[\Psi : B_{h_0}\x C^\infty(N) \to \mc{M}(N),\quad \Psi(r,\omega_0)=e^{2\omega_0}(h_0+r)\]
is a tame Fr\'echet diffeomorphism onto its image and 
$S_{h_0}:=\Psi(B_{h_0}\x\{0\})$ is a slice.
Let $\Phi$ be the smooth map of Fr\'echet manifolds
\begin{align*}
\Phi: B_{h_0}\x C^\infty(N) \to C^\infty(N) , &&  \Phi(r,\omega_0):=v_n(\Psi(r,\omega_0))
-\int_{N}v_n(\Psi(r,\omega_0)){\rm dvol}_{\Psi(r,\omega_0)}.
 \end{align*}
where we recall from Remark  \ref{rk:GZ} that $\int_{N}v_n(h){\rm dvol}_h$ is a conformal invariant.
The map $\Phi$ is a non-linear differential operator and thus is tame in the sense of \cite{Ha}. 
Notice that $\Phi(0,0)=0$.  We compute its differential with respect to the coordinate $\omega_0$:
\[ D\Phi_{(r,\omega_0)}(0,f)= \pl_s (v_n(e^{2sf}\Psi(r,\omega_0)))|_{s=0}.\]
Using Lemma \ref{dvn}� and writing $h=\Psi(r,\omega_0)$, we therefore have
\[D\Phi_{(r,\omega_0)}(0,f) =d_{h}^*(H_{h}df)-nv_n(h)f\]
where $d_{h}^*$ is the adjoint of $d$ with respect to $h$.
If $H_{h}$ (or equivalently ${\rm Hess}_{h}(\mc{V}_n)$) is positive definite or negative definite, then 
$f\mapsto D\Phi_{(r,\omega_0)}(0,f)$ is an elliptic self-adjoint differential operator of order $2$ acting on $C^\infty(N)$. If in addition the quadratic form \eqref{quadratic} is non-degenerate, then by continuity of $h\mapsto H_{h}$ and 
$h\mapsto v_n(h)$ in $C^\infty(N,S^2N)$ and the theory of elliptic differential operators, we deduce that 
$f\mapsto D\Phi_{(r,\omega_0)}(0,f)$ is an isomorphism on $C^\infty(N)$ for 
$(r,\omega_0)$ in a small neighbourhood of $(0,0)$ in $B_{h_0}\x C^\infty(N)$. 
Moreover the inverse is a pseudo-differential operator of order $-2$, depending continuously on 
$(r,\omega_0)$, which is automatically tame (see for instance \cite[Chap II.3]{Ha}). 
Therefore we can apply the Nash-Moser theorem and we obtain that there 
exists a smooth tame map 
\begin{equation}\label{omega_0k}
r\mapsto \omega_0(r)
\end{equation} 
of Fr\'echet spaces 
such that $\Phi(r,\omega_0(r))=0$, if $r$ is in a small open subset of $B_{h_0}$. 
The slice $\mc{S}$ is simply the image of $r\mapsto \Psi(r,\omega_0(r))$ for 
$k$ near $0$.
\end{proof}
 
\begin{rem}
Notice that we could instead apply the implicit function theorem in some $C^j(N)$ space with $j$ large 
enough by using \cite[Th. 1.4]{Gr}� which says that $\omega_0\mapsto v_n(e^{2\omega_0}h_0)$ maps 
$C^j(N)$ to $C^{j-2}(N)$, and then use uniqueness of the solution near the model cases to show that the solution 
$e^{2\omega_0}h_0$ is indeed $C^\infty(N)$ if $h_0$ is smooth. The proof amounts essentially to the same argument 
as Proposition \ref{pr:deformation} except that only the isomorphism of $D\Phi(0,h_0)$ is needed.
\end{rem}

We now apply the existence result of Proposition \ref{pr:deformation} to a couple of cases.
 
\subsection{Einstein manifolds}

We now consider the behavior of the renormalized volume in Poincar\'e-Einstein manifolds with a conformal
infinity containing an Einstein metric. A prime example is given by the ``Fuchsian'' Poincar\'e-Einstein
manifolds defined in the previous section.

\begin{lem}\label{vncst}
Let $N$ be an $n$-dimensional manifold with a conformal class $[h_0]$ that  contains an Einstein metric. 
Then the Einstein representative $h_0\in [h_0]$ with $\ric_{h_0}=\la(n-1)h_0$,  
satisfies  
\[v_n(h_0)=\frac{n!}{(n/2)!^2}(-\tfrac{\la}{4})^{\ndemi}.\] 
The Hessian of the renormalized volume
$\mc{V}_n$ at $h_0$, viewed as a symmetric tensor on $T^*N$, is given by 
\begin{equation}\label{hessianeinstein}
{\rm Hess}(\mc{V}_n)=-\frac{1}{4} \big(-\frac{\la}{4}\big)^{\ndemi-1}\frac{n!}{(n/2)!^2}h_0^{-1}.
\end{equation}
The Einstein metric $h_0$ is a local maximum for ${\rm Vol}_R$ 
in $\{h_0\in [h_0]; \int_{N}{\rm dvol}_{h_0}=1\}$  if either  $\la<0$ or if $\la>0$ and $\ndemi$ is odd. If $\la>0$ and $\ndemi$ is even,
it is a local minimum.
\end{lem}
\begin{proof}
In all these cases, one has from the expression \eqref{einsteinexp}
\begin{equation}\label{einsteincase}
\begin{gathered}
h_2=-\frac{\la}{2}h_0, \quad h_4=\frac{\la^2}{16}h_0, \quad h_{2j}=0 \textrm{ for }j>2, \quad v_{2j}=C^n_{j}(-1)^j(\frac{\la}{4})^j\\
h_x^{-1}=h_0^{-1}\sum_{j=0}^\infty (j+1)(\frac{\la}{4})^j x^{2j}, \quad h^{(2j)}=(j+1)(\frac{\la}{4})^jh_0^{-1}.
\end{gathered}
\end{equation}
In particular the Einstein metric $h_0$ satisfies $v_n(h_0)=C^{n}_{n/2}(-\demi)^\ndemi \la^{\ndemi}$, which is constant. 
Now Corollary \ref{hessian} gives the expression for the Hessian 
of ${\rm Vol}_R(M,e^{2\omega_0}h_0)$:
\[ {\rm Hess}_{h_0}(\mc{V}_n)(\omega_0)=-\tfrac12 (\frac{\la}{4})^{\ndemi-1}\sum_{k=0}^{\ndemi-1} C^n_{k}(-1)^k
\int_{\pl M}|\nabla \omega_0|_{h_0}^2{\rm dvol}_{h_0}.\] 
Using the binomial formula we get $2\sum_{k=0}^{\ndemi-1} C^n_{k}(-1)^k=-C^n_{n/2}(-1)^\ndemi$, which achieves the computation.

Let us check this is a local maximum in the $\la<0$ case, the other cases are similar. One has 
\[ \mc{V}_n(\omega_0)-\mc{V}_n(0)=\int_{0}^1(1-s) \pl_s^2({\rm Vol}_R(M,e^{2s\omega_0}h_0))ds=
\int_{0}^1(1-s){\rm Hess}_{e^{2s\omega_0}h_0}(\mc{V}_n)(\omega_0)ds.\]
Now from the formula giving the hessian in Corollary \ref{hessian} and the negativity of \eqref{hessianeinstein}, 
we have by continuity that there exists $\eps>0$ small, $k\gg n$ large and $c_0>0$ such that for all
 $||\omega_0||_{C^k(N)}\leq \eps$ and all $s\in[0,1]$
 \[{\rm Hess}_{e^{2s\omega_0}h_0}(\mc{V}_n)(\omega_0)\leq -c_0||d\omega_0||_{L^2}^2.\]
This implies that $\mc{V}_n(\omega_0)\leq \mc{V}_n(0)$ with equality if and only if $\omega_0$ is constant, but
since we restrict to $\int_{N}e^{n\omega_0}{\rm dvol}_{h_0}=\int_N {\rm dvol}_{h_0}=1$, the equality happens only if
$\omega_0=0$.
\end{proof}

Robin Graham informed us that in a forthcoming joint work with A.Chang and H.Fang, they also consider
the extremals of $v_n$ for conformal classes containing Einstein metrics.

Using this computation and applying Proposition \ref{pr:deformation}, 
we obtain that in a neighbourhood of a conformal class admitting an Einstein metric, the equation
$v_n={\rm const}$ can be solved except for the case of the canonical sphere.
\begin{cor}\label{sectionvncst}
Let $[h_0]$ be a conformal class on $N$ admitting a metric $h_0$ with ${\rm Ric}_{h_0}=\la(n-1)h_0\not=0$,  which is not conformal to 
the canonical sphere. 
Then, there is a neighbourhood $U_{h_0}\subset \mc{M}(N)$ of $h_0$ such that  
$\mc{S}:= \{h\in U_{h_0}; v_n(h)=\int_{N}v_n(h){\rm dvol}_h\}$
is a slice at $h_0$ for the conformal action of $C^\infty(N)$.\end{cor} 
\begin{proof}
The quadratic form \eqref{quadratic} is a non-zero constant times $\cjg (\Delta_{h_0}-n\la)\omega_0,\omega_0\cjd_{L^2}$ 
and using the Lichnerowicz-Obata theorem \cite{Ob}, then $\Delta_{h_0}-n\la$ has trivial kernel except for the case 
of the sphere. The result follows from Proposition \ref{pr:deformation}.
\end{proof}

\subsection{Locally conformally flat metrics}

In this case, we can take the Poincar\'e-Einstein metric to be of the form \eqref{formulalcf}, which can be rewritten 
\begin{align*} g=\frac{dx^2+h_x}{x^2},&& h_x(\cdot,\cdot)=h_0((1+\tfrac{1}{2}x^2H_2)^2\cdot ,\cdot)
\end{align*}
with $H_2$ some endomorphism of $TN$ (representing $-{\rm Sch}_{h_0}$). 
The metric $h_x^{-1}$ dual to $h_x$ has expansion near $x=0$ given by 
\begin{align*}
h_x^{-1}=h_0^{-1}(\sum_{j=0}^{\ndemi}H^{2j}\cdot,\cdot)+\mc{O}(x^{n+2}),&& H^{2j}=2^{-j}(j+1)(-H_2^*)^{j}
\end{align*}
where $H_2^*$ here denotes the endomorphism of $T^*N$ dual of $H_2$. Recall by \eqref{lcf} that 
\[v_{n}(h_0)=2^{-\ndemi}\sigma_{\ndemi}(H_2).\]

\begin{lem}
The Hessian of $\mc{V}_n$ at a locally conformally flat metric $h_0$ is given by
\[{\rm Hess}_{h_0}(\mc{V}_n)=2^{-\ndemi}h_0^{-1}\Big( \sum_{j=0}^{\ndemi-1}\sigma_{j}(H_2^*)(-H_2^*)^{\ndemi-j-1}\cdot ,\cdot\Big).
\] 
where $H_2^*$ is the dual endomorphism to $H_2$ defined by $h_2(\cdot,\cdot)=h_0(H_2\cdot,\cdot)$ and $\sigma_j(H_2^*)$ is the elementary symmetric function of order $j$  of $H_2^*$, as defined in \eqref{sigmak}. If $e_1,\dots,e_n$ is an orthogonal basis of 
eigenvectors of $H_2^*$, then  
\begin{equation}\label{hesslcf}
{\rm Hess}_{h_0}(\mc{V}_n)|_{\rr e_j}=2^{-\ndemi}\sigma_{\ndemi-1}(H_2^*|_{(\rr e_j)^\perp})h^{-1}_0.
\end{equation}
\end{lem}
\begin{proof} The first formula for the Hessian is a direct application of \eqref{hessH} and \eqref{lcf}, it remains to prove 
\eqref{hesslcf}. Let $\lambda_\ell$ be the eigenvalue of $H_2^*$ corresponding to $e_\ell$. Then, denoting 
by $F(t)_{[j]}$  the coefficient of $t^j$ in a power series $F(t)$,  we compute
\begin{align*}
\sum_{j=0}^{\ndemi-1}(-\la_\ell)^{\ndemi-j-1}\sigma_j(H_2^*)={}&\sum_{j=0}^{\ndemi-1}(-\la_\ell)^{\ndemi-j-1}\det(1+tH_2^*)_{[j]}\\
={}&  \sum_{j=0}^{\ndemi-1}[(-t\la_\ell)^{\ndemi-j-1}\det(1+tH_2^*)]_{[\ndemi-1]}\\
={}& [(1+t\la_\ell)^{-1} \det(1+tH_2^*)_{[\ndemi-1]]}\\
={}& \sum_{\substack{i_1<\dots<i_{\ndemi-1}\\ i_\bullet\not= \ell}}\la_{i_1}\dots\la_{i_{\ndemi-1}}
\end{align*}
which is the claimed formula. 
\end{proof}
We remark that $\sum_{j=0}^{\ndemi-1}\sigma_{j}(H_2^*)(-H_2^*)^{\ndemi-j-1}$ is the so called 
$(\ndemi-1)$-Newton transform $T_{\ndemi-1}(H_2^*)$ associated with $H_2^*$. 
The fact that $\pl_t\sigma_{\ndemi}(A(t))=T_{\ndemi-1}(A(t)).\pl_tA(t)$ for a family of symmetric matrices is well-known, see \cite{Re}. 
When the eigenvalues of $H_2^*$ are in the connected component containing $(\rr_+)^{n}$ inside the positive cone
\[\Gamma_{\ndemi}^+:=\{\la=(\la_1,\dots,\la_n)\in\rr^n; \sigma_{j}(\la)>0, \forall j=1,\dots, \tfrac{n}{2}\}\]
then $T_{\ndemi-1}(H_2^*)$ is positive definite, while when they are in $-\Gamma_{\ndemi}^+$, 
it is negative definite, see e.g. \cite{CaNiSp}. In the first case, it is proved in \cite{GuViWa} that if 
$\sigma_{\ndemi}(H_2)>0$ in the locally conformally flat case, 
then the manifold has to be of constant positive sectional curvature. On the other hand, when the eigenvalues of  $H_2$
are in $-\Gamma_{\ndemi}^+$, there seem to be no existence result for the equation $\sigma_{\ndemi}(H_2)={\rm const}$ (although there are interesting partial results in Gursky-Viaclovsky \cite{GuVi}).

\subsection{Dimension $4$} 

By Lemma \ref{v2v4},�the equation $v_4(e^{2\omega_0}h_0)={\rm const}$ is the $\sigma_2$-Yamabe equation, as introduced in the work of Viaclovsky \cite{Vi}.
It has solutions in dimension $n=4$ under certain ellipticity 
condition: when the eigenvalues of the Schouten tensor (viewed as an endomorphism via $h_0$) 
are in the connected component containing $(\rr_+)^{4}$ inside 
\[\Gamma_{2}^+:=\{\la=(\la_1,\dots,\la_4)\in\rr^4; \sigma_{2}(\la)>0, \sigma_1(\la)>0\},\]
then Chang-Gursky-Yang \cite{ChGuYa,ChGuYa2} proved that there is a solution $\omega_0$ 
of $v_4(e^{2\omega_0}h_0)={\rm const}$. Another proof appears in the work of Gursky-Viaclovsky \cite[Cor. 1.2]{GuVi2} and 
in Sheng-Trudinger-Wang \cite{ShTrWa}.
 
We now give a uniqueness result using maximum principle.
\begin{lem}\label{unique}
Let $(N,h_0)$ be a compact manifold. Assume that $\int_Nv_4(h_0){\rm dvol}_{h_0}>0$ and that 
${\rm Sch}_{h_0}-{\rm Tr}_{h_0}({\rm Sch}_{h_0})$ is positive definite. 
Then the equation $v_4(e^{2\omega_0}h_0)=\int_{N}
v_4(e^{2\omega_0}h_0){\rm dvol}_{e^{2\omega_0}h_0}$ has at most one solution $\omega_0\in C^\infty(N)$.
These conditions are satisfied in a neighbourhood of an Einstein metric $h_0$ with negative Ricci curvature.
\end{lem} 
\begin{proof} Assume there are two solutions.
Changing $h_0$ by a conformal factor we can assume that $0$ is a solution and let $\omega_0$ be the other solution, 
we then have $v_4(h_0)=v_4(e^{2\omega_0}h_0)$ as $\int_{N}v_4$ is a conformal invariant.
At the minimum $p\in N$ of $\omega_0$, one has $\nabla \omega_0(p)=0$. Since  
\[{\rm Sch}_{e^{2\omega_0}h_0}={\rm Sch}_{h_0}-2\nabla^2\omega_0+2d\omega_0\otimes d\omega_0-|d\omega_0|_{h_0}^2h_0\]
where $\nabla^2 \omega_0$ is the Hessian with respect to $h_0$, we deduce by using the expresion of $v_4$ in Lemma \ref{v2v4} that 
\[ v_4(h_0) =v_4(e^{2\omega_0}h_0)= e^{-4\omega_0}\Big(v_4(h_0)+
\sigma_2(B_{\omega_0})+\tfrac{1}{2}\cjg {\rm Sch}_{h_0}-\tra_{h_0}({\rm Sch}_{h_0}), \nabla^2\omega_0 \cjd _{h_0}\Big).\]
where $\sigma_2(B_{\omega_0})$ is the symmetric function of order $2$ in the eigenvalues of the symmetric endomorphism $B_{\omega_0}$ defined by $\nabla^2\omega_0=h_0(B\cdot,\cdot)$.
At $p$, the eigenvalues of $B_{\omega_0}$ are non negative, thus $\sigma_2(B_{\omega_0})\geq 0$ there.
Moreover, if $v_4(h_0)\not=0$, since $\int_{N}(e^{-4\omega_0}-1){\rm dvol}_{h_0}=0$ and thus 
$1-e^{-4\omega_0(p)}<0$ if $\omega_0\not=0$.
We then obtain, if $v_4(h_0)>0$, 
\[\cjg {\rm Sch}_{h_0}-\tra_{h_0}({\rm Sch}_{h_0}), \nabla^2\omega_0 \cjd _{h_0}(p)<0\]
thus if ${\rm Sch}_{h_0}-\tra_{h_0}({\rm Sch}_{h_0})$ is positive definite, we obtain a contradiction with the maximum principle.
\end{proof}

\section{General variations of the renormalized volume}

We shall now compute the variation of $\rvol$ for a family of Einstein metrics. 

\subsection{The Schl\"afli formula}

We recall the Schl\"afli formula proved by Rivin-Schlenker 
\cite{RiSc} for Einstein manifolds with boundary and non zero Einstein constant. 
For completeness, we give  short proof of this formula arising from the variation formula for scalar curvature, this is similar to 
\cite[Lemma 2.1]{An}.
\begin{lem}[\textbf{Rivin-Schlenker}]\label{Rivin}
Let $M$ be an $n+1$-dimensional manifold with boundary  and 
$g^t$ a family of Einstein metrics on $M$ with $H^t$ the mean curvature at $\pl M$, and $\II^t$ the second fundamental 
form at $\pl M$, computed with respect to the inward-pointing unit normal vector field to $\pl M$. 
Let $\Ric_{g^t}=n\lambda_tg^t$ and assume that $\lambda_0\neq 0$. Then
\begin{equation}\label{Schlafli}
\pl_t{\rm Vol}(M,g^t)|_{t=0}=-\frac{(n+1)\dot{\la}}{2\la_0}{\rm Vol}(M,g)-
\frac{1}{n\la_0}\int_{\pl M}(\dot{H}+\tfrac12 \cjg \dot{g}, \II\cjd_{g} ){\rm dvol}_{\,\pl M}
\end{equation}
where dot denotes the time derivative at $t=0$ and ${\rm dvol}_{\pl M}$ is the volume form induced by the restriction of
$g^0$ on $\pl M$, and $g=g^0$. 
\end{lem} 
\begin{proof} We use the variation formula of the scalar curvature of a $1$-parameter family of Riemannian metrics \cite[Theorem 1.174]{besse}:
\begin{equation}\label{vscal}
\pl_t \scal_{g^t}|_{t=0} = \Delta_g \Tr_g(\dot{g})+d^* \delta^g \dot{g} -\langle\Ric_g,\dot{g}\rangle.
\end{equation}
Since $g^t$ is Einstein, we have $\Ric_{g^t}=\frac{\scal_{g^t}}{n+1}g^t$ and hence
\[\langle\Ric_g,\dot{g}\rangle {\rm dvol}_{g} =\frac{\scal_g}{n+1} \Tr_g(\dot{g}){\rm dvol}_{g}
= \frac{2\scal_g}{n+1}\pl_t {\rm dvol}_{g^t}|_{t=0}.\]
Let $\nu$ be the inward-pointing unit vector field on $\pl M$. Integrating \eqref{vscal} times ${\rm dvol}_g$ 
over $M$ and using Stokes we get
\begin{align*}
2n\lambda_0 \pl_t{\rm Vol}(M,g^t)|_{t=0}={}& \int_M (\Delta_g \Tr_g(\dot{g})+d^* \delta^g \dot{g}-n(n+1)\dot{\lambda}){\rm dvol}_g\\
={}&-n(n+1)\dot{\lambda}{\rm Vol}(M,g)+\int_{\pl M} (\nu(\Tr_g(\dot{g})) + \delta^g (\dot{g})(\nu)) {\rm dvol}_{\pl M}.
\end{align*}
To compute the right-hand side we reduce to the case where the metric is of the form $g^t=dx^2+h_x^t$ near the boundary 
where $x$, the distance function to the boundary for $g^t$, is independent of $t$, and $h^t_x$ are metrics 
on $\pl M$ depending smoothly on $x,t$. One way to do that is to pull-back $g^t$ by a diffeomorphism 
$\psi^t$ which is the identity on $\pl M$ and constructed as follows: let 
\begin{align*}
\phi^t: \pl M\x [0,\eps)\to M, &&\phi^t(p,s):=\exp^{g^t}_p(s\nu^t)
\end{align*}
be the normal geodesic flow where $\nu^t$ is the inward-pointing unit normal to $\pl M$ with respect to $g^t$, and
then set $\psi^t$ to be any diffeomorphism of $M$ extending $\phi^0\circ (\phi^t)^{-1}$ defined near $\pl M$.
We replace $g^t$ by $(\psi^t)^*g^t$ and remark that all the terms in \eqref{Schlafli} are invariant by this operation.  

We have $\II^t=-\demi \pl_xh^t_x|_{x=0}$ and $H^t=\tra_{h^t_0}(\II^t)=-\demi\tra ((h^t_0)^{-1}\pl_xh^t_x)|_{x=0}$. 
Since $\pl_x\tra(A^{-1}\pl_tA)=\pl_t\tra(A^{-1}\pl_x A)$ if $A=A(x,t)$ is a family of invertible matrices, we deduce
\[\begin{split}
\nu(\tra_g(\dot g))={}&\pl_x \tra_{h_x}(\dot{h}_x)|_{x=0}=\pl_x\tra(h_x^{-1}\pl_th^t_x)|_{t=0,x=0}\\
 ={}&
\pl_t\tra((h_0^t)^{-1}\pl_xh^t_x)|_{t=0,x=0}=-2\dot{H}.\end{split}\]
Using that $\dot{g}=\dot{h}_x$, it is easy to see that 
\[ \delta_g(\dot{g})(\nu)=-\cjg \dot{g},\II\cjd,\]
which concludes the proof. 
\end{proof}

\subsection{Variation of the renormalized volume in arbitrary dimensions}

Let $g^t$ be a family of asymptotically hyperbolic Einstein metrics on $M$, and choose a family of boundary defining functions $x^t$.
We can pull back $g^t$ by a diffeomorphism 
$\psi^t$ so that $x^t=(\psi^t)^*x$ is a fixed function on $M$ and consider $(\psi^t)^*g^t$ instead of $g^t$. Clearly  
$\rvol(M,g^t;x^t)=\rvol(M,(\psi^t)^*g^t;x)$ therefore we can assume that $x^t$ is independent of $t$ and to simplify notation we will write 
$\rvol^t(M)$ for $\rvol(M,g^t;x)$.

We write $g^t=(dx^2+h^t_x)/x^2$ and use the dot notation for $\pl_t|_{t=0}$ and $h_x:=h^0_x$.

\textbf{Regularity Assumption:} we assume in this section that $g^t$ is a $C^1$ function of $t$, near $t=0$, with values in the space of 
smooth metrics on $M$, 
and that $h^t_x$ is a $C^1$ function of $t$ with values in the space of conormal 
polyhomogeneous tensors equipped with the natural topology (i.e.,
the asymptotic expansions of $h_x^t$ at $x=0$ are $C^1$ in $t$).

In dimension $n+1$ even, Albin \cite[Th. 1.3]{Al} and Anderson \cite[Th. 0.2]{An} proved

\begin{theo}[\textbf{Albin, Anderson}]\label{alban}
Let $g^t$ be a family of AHE metrics on $M$ with $n$ odd 
and let $h_0^t=h_0+t\dot{h}_0+o(t)$ be a $C^1$  
family of representatives of the conformal infinity $(\pl M,[h_0^t])$ of $(M,g^t)$. Let $h_n$ be the 
Neumann datum of $g^0$ in the sense of \eqref{Dirichlet/neumann}. Then 
\[
  \pl_t{\rm Vol}^t_R(M)|_{t=0}= -\tfrac{1}{4}\int_{\pl M}\cjg h_n,\dot{h}_0\cjd{\rm dvol}_{h_0} 
  \]
\end{theo}

Here, we study the more complicated case when $n$ even and obtain:

\begin{theo} \label{variationvol}
Let $g^t$ be a family of AHE metrics on $M$ with $n$ even, satisfying the regularity assumption described above when 
written under the form $g^t=(dx^2+h^t_x)/x^2$ for some fixed $x$ near $\pl M$.
We write $h_0^t=h_0+t\dot{h}_0+o(t)$ and let $h_n$ be the 
Neumann datum of $g^0$.  There exists a symmetric covariant $2$-tensor $F_n$
formally determined by $h_0$, of order $n$, such that
\begin{equation}
  \label{eq:schlafli}
  \pl_t{\rm Vol}^t_R(M)|_{t=0}= \int_{\pl M}\cjg G_n,\dot{h}_0\cjd{\rm dvol}_{h_0} 
  \end{equation}
where $G_n:=-\frac{1}{4}(h_n+F_n)$ satisfies $\delta_{h_0}(G_n)=0$
and ${\rm Tr}_{h_0}(G_n)=\demi v_n$.
\end{theo}

\begin{proof} 
We will use the Schl\"afli formula  for the compact manifold with boundary $\{x\leq \eps\}$. The second  
fundamental form, mean curvature and their variation on the hypersurface $\{x=\eps\}$ are given by  the value at $x=\eps$ of   
\begin{align*} 
\II={}&-\tfrac12 x\pl_x(h_x/x^2)=-x^{-2}(\tfrac12 x\pl_x h_x-h_x), \\
H={}&\tra_{h_x}(\II)=-\tfrac12 \tra_{h_x}(x\pl_xh_x)-n,\\
\dot H={}& \tfrac12 \cjg \dot{h}_x,x\pl_xh_x\cjd_{h_x} -\tfrac12 \tra_{h_x}(x\pl_x\dot{h}_x).
\end{align*}
Let us denote $\dot{{\rm Vol}}_R=\pl_t{\rm Vol}^t_R(M)|_{t=0}$ .
We are interested in computing  
\begin{equation}\label{dotvol} 
-n\dot{{\rm Vol}}_R=\tfrac12 {\rm FP}_{\eps\to 0} \int_{x=\eps}(\tra_{h_x}((x\pl_x-1)\dot{h}_x)-\tfrac12 \cjg \dot{h}_x,x\pl_xh_x\cjd_{h_x})\frac{v_x}{x^{n}}\dvol_{h_0} 
\end{equation}
where $v_x\dvol_{h_0}=\dvol_{h_x}$. We write modulo $o(x^{n})$ 
\begin{align*}
v_x={}&\sum_{2q\leq n}x^{2q}v_{2q},& v_0={}&1\\
h_x={}&h_0\Big(\sum_{2j\leq n}x^{2j}H_{2j}+x^{n}\log(x)K\Big) , & H_0={}&1\\ 
h_x^{-1}={}&\Big(\sum_{2j\leq n}x^{2j}H^{2j}-x^{n}\log(x)K\Big)h_0^{-1},& H^0={}&1
\end{align*}
\begin{align*}
\dot{h}_x={}&\dot{h}_0\Big(\sum_{2j\leq n}x^{2j}H_{2j}+x^{n}\log(x)K\Big)+h_0\Big(\sum_{2j\leq n}x^{2j}\dot{H}_{2j}+x^{n}\log(x)\dot{K}\Big)\\
x\pl_x h_x={}&h_0\Big(\sum_{2j\leq n}2jx^{2j}H_{2j}+nx^{n}\log(x)K+x^nK\Big)\\ 
(x\pl_x-1) \dot{h}_x={}&\dot{h}_0\Big(\sum_{2j\leq n}(2j-1)x^{2j}H_{2j}+(n-1)x^{n}\log(x)K+x^nK\Big)\\
 & +h_0\Big(\sum_{2j\leq n}(2j-1)x^{2j}\dot{H}_{2j}+(n-1)x^{n}\log(x)\dot{K}+x^n\dot{K}\Big).
\end{align*}
Taking the term of degree $x^{n}$ and using ${\tra}(K)=\tra(\dot{K})=0$, we get 
\begin{equation}\label{calcul1}\begin{split}
[v_x\tra_{h_x}((x\pl_x-1)\dot{h}_x)]_{n}=\cjg\dot{h}_0,k\cjd+\sum_{i+j+q=\ndemi} &(2j-1)[\tra(H^{2i}\dot{H}_{2j})\\ &+ 
\tra(h_0^{-1}\dot{h}_0H_{2j}H^{2i})]v_{2q}
\end{split}
\end{equation}
\begin{equation}\label{calcul2}
\begin{split}
[v_x \cjg \dot{h}_x,x\pl_xh_x\cjd]_n={}&\cjg\dot{h}_0,k\cjd+\sum_{i+j+m+\ell+q=\ndemi}2\ell (\tra(H^{2i}\dot{H}_{2j}H^{2m}H_{2\ell})
\\
& + \tra(h_0^{-1}\dot{h}_0H_{2j}H^{2m}H_{2\ell}H^{2i}))v_{2q}\\
={}& \cjg\dot{h}_0,k\cjd+\sum_{i+j+m+\ell+q=\ndemi}2\ell (\tra(H^{2i}\dot{H}_{2j}H^{2m}H_{2\ell})v_{2q}\\
 & +\sum_{i+\ell+q=\ndemi}
 2\ell\tra(h_0^{-1}\dot{h}_0H_{2\ell}H^{2i}))v_{2q}\end{split}\end{equation}
where in the last line we used $\sum_{j+m=u}H_{2j}H^{2m}=0$ for all $u>0$ .
Let us single out the terms in $-n\dot{\rm Vol}_R$ which do not depend formally on $h_0$. 
Since the $H_{2j},H^{2j},v_{2j}$ are formally determined by $h_0$ of order $2j$ when $j<n/2$, 
by Lemma \ref{deriveP} we know that there exist $R_{n}$ formally determined  by $h_0$ of order 
$n$ such that 
\[-n\dot{\rm Vol}_R=\tfrac12\Big((n-1)(\tra(\dot{H}_n)+\cjg \dot{h}_0,h_n\cjd)-\tra(h_{0}^{-1}\dot{h}_0H^n)\Big)-\tfrac{n}{4}\cjg \dot{h}_0,h_n\cjd+
\cjg \dot{h}_0,R_n\cjd.\]
But since $H^n+H_n$ depends formally on $h_0$, this reduces to considering terms containing 
$H_n,\dot{H}_n$ and we get that there exists $R_n'$ formally determined by $h_0$ of order $n$ such that 
\[ \begin{gathered}
-n\dot{\rm Vol}_R= \tfrac{n-1}{2}\pl_t \tra_{h^t_0}(h^t_n)|_{t=0}+\tfrac{n}{4}\cjg\dot{h}_0,h_n\cjd + \cjg \dot{h}_0,R'_n\cjd
\end{gathered}.\]
Now we know that $\tra_{h^t_0}(h^t_n)$ is formally determined with respect to $h^t_0$ of order $n$ for each $t$, therefore we have 
established \eqref{eq:schlafli} with $G_n=-\frac{1}{4}(h_n+F_n)$ for some $F_n$ formally determined by $h_0$ of order $n$.

Let us now show that $\tra_{h_0}(G_n)=\demi v_n$ and $\delta_{h_0}(G_n)=0$. Let $h_0^t=e^{2t\omega_0}h_0$ for some function $\omega_0\in C^\infty(\pl M)$.
We have $\dot{h}_0=2\omega_0h_0$, and combining \eqref{eq:schlafli} with Lemma \eqref{volumeformula} and \eqref{expomega} we get
\[ \dot{\rm Vol}_R= \int_{\pl M} v_n\omega_0 {\rm dvol}_{h_0}=2\int \tra_{h_0}(G_n)\omega_0 {\rm dvol}_{h_0}\]
for all $\omega_0$, and so $2\tra_{h_0}(G_n)= v_n$. It remains to compute the divergence of $G_n$.
Let $\phi^t=\exp(tV)$ be a one-parameter family of diffeomorphisms of $M$ generated by a vector field $V$ such that $dx(V)=0$ near $\pl M$. 
Then $\rvol(M,(\phi^t)^*g;x)$ is independent of $t$ because $\phi^t$ preserves the regions $\{x>\eps\}$ for any small $\eps>0$.
Therefore from \eqref{eq:schlafli} applied to $\dot{h}_0=L_Vh_0=2\delta_{h_0}^*V$ we get 
\[ 0=\dot{\rvol}=\cjg \dot{h}_0,G_n\cjd= 2\cjg V,\delta_{h_0}(G_n)\cjd.\]
Since $V|_{\pl M}$ can be chosen arbitrarily, we conclude that $\delta_{h_0}(G_n)=0$. 
\end{proof}

Although $F_n$ has been defined as a function of $h_0$ when $h_0$ is the conformal infinity of an Einstein metric, 
the fact that it is formally determined implies that we can consider $F_n(h_0)$ for any metric $h_0$. 
\begin{cor}\label{divGn}
Let $(N,h_0)$ be a Riemannian manifold.  There exists a tensor $F_n=F_n(h_0)$ formally determined by $h_0$, of order $n$, such that
\begin{align}\label{TrFn}
{\rm Tr}_{h_0}(F_n)=-T_n-2 v_n,&& \delta_{h_0}(F_n)=-D_n
\end{align}
where  $D_n,T_n$ are the formally determined tensors of Propoisition \ref{FeGr} and $v_n$ is the formally determined function 
defined by the volume expansion in \eqref{v2i}. If $(h_0,h_n)$ is a Poincar\'e-Einstein end, then 
$\delta_{h_0}(h_n+F_n)=0$ and $\Tr(h_n+F_n)=-2v_n$.
\end{cor}
\begin{proof} Since $F_n(h_0)$ is formally determined by $h_0$, we see by Remark \ref{isolocal} that it suffices to prove the 
result on metrics on the sphere $S^n$. For the round metric $h_{S^n}$ on $S^n$, or any other metrics which is the conformal infinity of an AHE metric on the unit ball $B^{n+1}$, 
the conclusion \eqref{TrFn} follows directly from Theorem \ref{variationvol}, more precisely from the last part of its proof. 
If now $h_0$ is any metric on $S^n$, we define the metrics $h^t_0:=th_0+(1-t)h_{S^n}$ for $t\in[0,1]$. By Graham-Lee \cite{GrLe},
for small $t\in[0,\eps]$, the metric $h_0^t$ is the conformal infinity of some AHE metric $g^t$ on $B^{n+1}$ and we have seen that 
this implies \eqref{TrFn} for $h_0^t$ with $t\in[0,\eps]$. But $F_n(h_0^t)$, $T_n(h_0^t)$ and $v_n(h_0^t)$ are 
real analytic in $t$, therefore by unique continuation we deduce that \eqref{TrFn} holds for $h_0=h_0^1$. 
\end{proof}

\subsection{Case $n=2$} \label{ssc:n=2}

We do not give full details of the computation, since this case has been analyzed in \cite{KrSc,GuMo}. 
With the notation of the proof of Theorem \ref{variationvol}) we have
\begin{align} \label{eq:v2}
k=0, && v_2=\tfrac12\tra_{h_0}(h_2)=\tfrac12\tra(H_2),&& H^2=-H_2,
\end{align}
and from \eqref{dotvol}, \eqref{calcul1}, \eqref{calcul2}, we obtain 
\[ \dot{\rvol}=-\tfrac{1}{4}\left( \int_{\pl M} 2\pl_t(\tra_{h_0}(h_2))|_{t=0}{\rm dvol}_{h_0} +\left\cjg \dot{h}_0, h_2-v_2h_0\right\cjd\right).\]
By \cite[Prop 7.2]{FeGr},  $\tra_{h_0}(h_2)=-\demi {\rm Scal}_{h_0}$, and thus, using the Gauss-Bonnet formula, we easily get 
$\int_{\pl M}\dot{\rm Scal}_{h_0}{\rm dvol}_{h_0}=-\demi\cjg \dot{h}_0,{\rm Scal}_{h_0}h_0\cjd$. We conclude
\begin{align}\label{G2F2}
\dot{\rvol}=-\tfrac{1}{4}\left\cjg \dot{h}_0, h_2+\tfrac12 {\rm Scal}_{h_0}h_0\right\cjd, && -4G_2=h_2+\tfrac12{\rm Scal}_{h_0}h_0, && F_2=\tfrac12{\rm Scal}_{h_0}h_0.
\end{align}

\subsection{Case $n=4$}

First, we have the relations (with the notation of the proof of Theorem \ref{variationvol})
\begin{align*}
H^2=-H_2, && H^4=-H_4-H_2H^2=-H_4+H_2^2.
\end{align*}
From \eqref{dotvol}, \eqref{calcul1} and \eqref{calcul2} we obtain 
\begin{equation}\label{dim4dotvol}
\begin{split}
-8\dot{\rvol}={}& \left\cjg \dot{h}_0,\tfrac12 k+2h_4-h^2_2+v_2h_2-v_4h_0\right\cjd\\
 & +\int_{\pl M}(v_2\tra(\dot{H}_2)-2\tra(H_2\dot{H}_2)+3\tra(\dot{H}_4)){\rm dvol}_{h_0}.
\end{split}\end{equation}
where $h^2_2:=h_0^{-1}H_2^2$ is the tensor obtained by composing the endomorphism $H_2$ with itself.
Now, recall Lemma \ref{v2v4} obtained from the constraint equation on the trace of the shape operator, which gives 
\begin{equation}\label{dotv4}
\begin{split}
v_2\tra(\dot{H}_2)-2\tra(H_2\dot{H}_2)+3\tra(\dot{H}_4)={}&\pl_t\Big(\tfrac{1}{4}\tra(H_2)^2-\tra(H_2^2)+3\tra(H_4)\Big)|_{t=0}\\
={}& 2\dot{v}_4.
 \end{split}
 \end{equation}
But we also have from \eqref{grhi}
\[\int_{\pl M}\dot{v}_4\,{\rm dvol}_{h_0}+\tfrac12\cjg\dot{h}_0,v_4h_0\cjd=\tfrac{1}{4}\cjg \dot{h}_0,k\cjd\]
and by combining with \eqref{dim4dotvol} and \eqref{dotv4}, we obtain 
\begin{align*}\dot{\rvol}=\cjg \dot{h}_0,G_4\cjd, && -4G_4:=h_4-\tfrac12 h_2^2+\tfrac12 v_2h_2-v_4h_0+\tfrac12 k
\end{align*}
and by Lemma \ref{v2v4} this can be rewritten as 
\[ -4G_4=h_4-\tfrac12 h_2^2+\tfrac{1}{4} \tra_{h_0}(h_2)h_2
-\tfrac{1}{4}\sigma_2(h_2)h_0+\tfrac12 k,\]
where $h_2=- {\rm Sch}_{h_0}=- \demi({\rm Ric}_{h_0}-\frac{1}{6}{\rm Scal}_{h_0}h_0).$

\subsection{Einstein metric in the conformal infinity} 

If ${\rm Ric}_{h_0}=\la (n-1)h_0$ for some $\la\in \rr$, 
one can prove that the tensor $F_n$ is a constant times $h_0$:
\begin{lem}\label{FnEinstein}
Let $h_0$ be Einstein, ${\rm Ric}_{h_0}=\la (n-1)h_0$. Then $F_n=-2\frac{(n-1)!}{(n/2)!^2}(-\tfrac{\la}{4})^{\ndemi}h_0$ and $G_n=-\tfrac{1}{4}(h_n-2\frac{(n-1)!}{(n/2)!^2}(-\tfrac{\la}{4})^{\ndemi}h_0)$. In particular, $\pl_s{\rm Vol}_R(M,g^s;h^s_0)|_{s=0}=0$ if $g^s$ is a family of AHE metrics 
with $(g^s,h_0^s)|_{s=0}=(g,h_0)$ and ${\rm Vol}(N,h_0^s)=1$, and if the trace-free part of the tensor $h_n$ in the expansion of $g$ is $0$.  
\end{lem}
\begin{proof} First, notice that $T_n=0$ in that case since $T_n=\tra_{h_0}(h_n)$ depends only on $h_{2j}$ for $j<n/2$ 
and the metric $g:=x^{-2}(dx^2+(1-\la x^2/4)^2h_0)$ is an exact Einstein metric near $x=0$ 
which has $h_n=0$ (see Section \ref{ssc:fuchsian} below).
Therefore it suffices to prove that $F_n$ is proportional to $h_0$ and the multiplicative 
constant is deduced directly from \eqref{TrFn} and the formula  $v_n=C^n_{n/2}(-\tfrac{\la}{4})^\ndemi$ of Lemma \ref{vncst}.
Let $A_x=h_x^{-1}\pl_xh_x=\frac{-\la x}{(1-\la x^2/4)}{\rm Id}$ if $h_x=(1-\la x^2/4)^2h_0$. If $g^t=(dx^2+h^t_x)/x^2$ 
is a family of Poincar\'e-Einstein metrics near $x=0$, with $g^0=g$, then differentiating the first constraint 
equation  in \eqref{eqEinsteintrace} at $t=0$ gives 
\begin{align*} 
\pl_x F(x)-\frac{\la x}{1-\la x^2/4}F(x)=0,&& F(x)=x^{-1}\tra(\dot{A}_x)\in C^\infty([0,\eps))
\end{align*}
and $\dot{A}_x=\pl_{t}A_x|_{t=0}$. In particular 
\begin{equation}\label{traAx}
\tra(\dot{A}_x)=a_0x\exp\left(\int_0^x\tfrac{\la t}{1-\la t^2/4}dt\right)
\end{equation} 
is determined by a constant  $a_0\in\rr$.  

Using the notations in the proof of Theorem \ref{variationvol}, we claim that there exists $c_j,d_j\in\rr$ such that 
for all $j\leq n/2$,
\begin{align}\label{multiples}
a_0=2\tra(\dot{H}_2), && \tra(\dot{H}_{2j})=c_j\tra(\dot{H}_2), && \tra(\dot{H}^{2j})=d_j\tra(\dot{H}_2).
\end{align}
Since $\sum_{j,k=0}^{n/2} x^{2(j+k)}H_{2j}H^{2k}={\rm Id}+\mc{O}(x^{n+1})$ and $H_{2j}|_{t=0},H^{2k}|_{t=0}$ are multiples of ${\rm Id}$,
then $\tra(\dot{H}^{2j})=-\tra(\dot{H}_{2j})+\sum_{k=0}^{j-1}b_k\tra(\dot{H}_{2k})$ for some constants $b_k\in\rr$. 
But modulo $o(x^{n})$, we have
\[x^{-1}\tra(\dot{A}_x)=\sum_{k=1}^{n/2}\sum_{j=0}^{n/2} 2k(\alpha_k\tra(\dot{H}^{2j})+\beta_j\tra(\dot{H}_{2k}))x^{2(j+k-1)}
\]
for some $\alpha_k,\beta_j\in\rr$ such that $\beta_0=1$, thus an easy induction and \eqref{traAx} prove \eqref{multiples}.

Inserting \eqref{multiples} in \eqref{calcul1} and \eqref{calcul2}, and using that 
$v_{2q}$ are constant if $h_0$ is Einstein for $q\leq n/2$ by \eqref{einsteincase},
we deduce directly that there exists $C\in\rr$ such that
\[\dot{{\rm Vol}}_R=-\tfrac{1}{4}\cjg h_n,\dot{h}_0\cjd+ C\int_{\pl M}\tra(\dot{H}_2){\rm dvol}_{h_0}.\]
Since $\tra(\dot{H}_2)=\pl_t(\tra_{h^t_0}(h^t_2))|_{t=0}$ and $\tra_{h^t_0}(h^t_2)=C'{\rm Scal}_{h^t_0}$, we can use 
the variation formula \eqref{vscal} for the scalar curvature, integration by parts and 
the fact that ${\rm Ric}_{h^t_0}=\la (n-1)h^t_0$ when $t=0$ to conclude that 
$\int_{\pl M}\tra(\dot{H}_2){\rm dvol}_{h_0}=C''\cjg h_0,\dot{h}_0\cjd$ for some $C''\in\rr$. 
If $(M,g)$ is an AHE manifold with conformal infinity containing an Einstein representative $h_0$, then the traceless part of 
$G_n$ is the traceless part of the formally undetermined term $h_n$ (for the choice of $x$ associated to the metric $h_0$). 
This achieves the proof.
\end{proof}


\section{Cotangent space of conformal structures and quasifuchsian reciprocity in higher dimension}\label{qfr}

We can now explain how the results of the previous section for hyperbolic manifolds in three 
dimensions can be used to identify Poincar\'e-Einstein ends  modulo gauge
with cotangent bundles to the space of conformal structures. This 
allows to extend McMullen's quasifuchsian reciprocity, or more generally Kleinian reciprocity \cite{mcmullen}, in dimension
$n+1$. We will work in both even and odd dimensions, but for $n$ even we shall need more hypotheses.

\subsection{Assumptions and the slice $v_n={\rm const}$}

To get a satisfactory picture where the analogs of the 3-dimensional phenomena can be stated and proved, 
two technical hypothesis will be necessary.
We show below that those hypothesis are satisfied in non-trivial situations.

Like in Section \ref{ssc:conformal}, we denote by $\mc{M}(N)$ the space of smooth metrics on $N$ and  
by $\mc{M}(M)$ the space of 
polyhomogeneous metrics on $M$ in the sense of Section \ref{AHEM} together with its natural Fr\'echet structure.
 
We will consider in this section the situation where the following hypotheses hold. Let $h_0\in \mc{M}(N)$ be a fixed metric.
\begin{hypothesis}\label{hypo0}
The metric $h_0$ has no conformal Killing fields and the quotient space $\mc{T}(N)=\mc{G}\backslash \mc{M}(N)$ 
has a Fr\'echet manifold structure near $[h_0]\in\mc{T}(N)$.
\end{hypothesis}

\begin{hypothesis} \label{hypo2}
There is a slice $\mc{S}_0$ at $h_0$ for the action of 
$\mc{G}=\mc{D}_0(N)\ltimes C^\infty(N)$ on $\mc{M}(N)$ as defined in \eqref{slice2}, and $\mc{S}_0$ 
is included in the subset of metrics $\{ h\in \mc{M}(N); v_n(h)=\int_Nv_n(h){\rm dvol}_h\}$. 
\end{hypothesis}

\begin{hypothesis} \label{hypo1}
Let $\mc{S}_0$ be a slice at $h_0$ for the action of $\mc{G}=\mc{D}_0(N)\ltimes C^\infty(N)$ on $\mc{M}(N)$. Then
there is a $C^1$ map of Fr\'echet manifolds
$\Xi: \mc{S}_0\to \mc{M}(M)$ such that $\Xi(h)$ is asymptotically hyperbolic Einstein 
with conformal boundary $(N,[h])$.
\end{hypothesis}

Using the existence results for Einstein equation obtained by Biquard or Lee \cite{biquard:metriques,lee:fredholm} and 
the result of Corollary \ref{sectionvncst}, we obtain
\begin{prop} \label{rk:rigidity}
Let $h_0\in \mc{M}(N)$ be an Einstein metric with negative sectional curvatures and let $g_0\in \mc{M}(M)$ 
be an AHE metric with non-positive sectional curvatures on a manifold $M$ with conformal boundary $(N,[h_0])$. 
Then Hypothesis \ref{hypo0} and \ref{hypo1} are satisfied. If $n$ is even, Hypothesis  \ref{hypo2} is also satisfied. Moreover 
$\mc{S}_0$ can be chosen so that $T_{h_0}\mc{S}_0=\{r_0\in C^\infty(N,S^2N); {\rm Tr}_{h_0}(r_0)=0, \delta_{h_0}(r_0)=0\}$.
\end{prop}
\begin{proof} Hypothesis \ref{hypo0} comes from the fact that $\mc{G}$ acts properly since $N$ is not the 
sphere and there is no conformal Killing field for 
$h_0$ since the Ricci curvature is negative (by Yano \cite{Ya}), ie.\ the isotropy group at $h_0$ is finite and in fact it is trivial by Frenkel \cite{Fr} since we assumed the sectional curvatures to be non-positive.   
If $\mc{S}_0$ is any given slice at $h_0$ for the action of $\mc{G}$ and if $g_0\in\mc{M}(M)$ is
an AHE metric with non-positive sectional curvatures on $M$ and with conformal boundary $[h_0]$,
then Hypothesis \ref{hypo1} holds, after intersecting $\mc{S}_0$ with a small enough neighbourhood of 
$h_0$; this is proved by Biquard \cite{biquard:metriques} and Lee \cite[Theorem A]{lee:fredholm}.
In fact, technically speaking, \cite{lee:fredholm} does not prove it with the topology we need,
(i.e. that for which the whole expansion of the metric at the boundary depends in a $C^1$ fashion on $h_0$), but the arguments used by Biquard in the K\"ahler-Einstein setting \cite{BiGAFA} give the right property, in fact it is even simpler in our case.  
If $n$ is even, we know by Corollary \ref{sectionvncst} that there is a slice $\mc{S}\subset\mc{M}(N)$ at $h_0$
for the conformal action with $\mc{S}=\{h\in U_{h_0}; v_n(h)=\int_Nv_n(h){\rm dvol}_h\}$ for some neighbourhood 
$U_{h_0}\subset \mc{M}(N)$ of $h_0$. There is an action by pull-back
\begin{align*} \Theta:\mc{D}_0(N)\x \mc{S}\to \{h\in \mc{M}(N); v_n(h)=\int_Nv_n(h){\rm dvol}_h\},&& \Theta(\phi,h)=\phi^*h.\end{align*}
The set on the right is a Fr\'echet submanifold when intersected with a small neighbourhood of $h_0$ in $\mc{M}(N)$. Let us first define a slice $\mc{S}_0\subset \mc{S}$ at $h_0$ for the action $\Theta$ in the sense of \eqref{slice2}.  
To that aim, we return to the proof of Proposition \ref{pr:deformation} and use the notations there. We define the smooth tame map
\begin{align}\label{PI}
 \Pi: B_{h_0}\to \mc{S}, && r\mapsto e^{2\omega_0(r)}(h_0+r)
 \end{align}
where $\omega_0(r)$ is obtained from \eqref{omega_0k} by solving $\Phi(r,\omega_0(r))=0$. 
This is a Fr\'echet chart for $\mc{S}$.  The derivative is the  tame family of isomorphisms defined 
on $\{\dot{r}\in C^\infty(N,S^2N);\tra_{h_0}(\dot{r})=0\}$
\begin{equation}\label{DPir}
D\Pi_{r}(\dot{r})=e^{2\omega_0(r)}\dot{r}+2(D\omega_0)_r(\dot{r})\Pi(r)
\end{equation}
where, from the proof of Proposition \ref{pr:deformation} using the Nash-Moser implicit function theorem,
we have that $r\mapsto (D\omega_0)_r$ is a tame map 
into pseudo-differential operator on $N$ of order $0$.
We take the  open neighbourhood $B'_{h_0}:=\{r\in B_{h_0}; \delta_{h_0}(r)=0\}$ of the Fr\'echet space of 
trace-free/divergence-free tensors with respect to $h_0$.
We will call $\mc{S}_0$ the image by $\Pi$ of a neighbourhood of $h_0$ contained in $B'_{h_0}$; this is a Fr\'echet submanifold of $\mc{S}$ and we are  now going to show that it is a slice for the action of $\mc{D}_0(N)$. In that aim, 
we apply the Nash-Moser inverse function theorem to the restriction
$\Theta_0:\mc{D}_0(N)\x \mc{S}_0\to \Theta(\mc{D}_0(N)\x \mc{S}_0)$ of $\Theta$ to $\mc{D}_0(N)\x \mc{S}_0$. The derivative at $(\phi,h)$ is 
\[ (D\Theta_0)_{(\phi,h)}(X,\dot{h})=\phi^*(L_Xh+\dot{h})\in T_{\phi^*h}\mc{S}\]
where $X\in {\rm lie}(\mc{D}_0(N))$ is a smooth vector field and $\dot{h}\in T_{h}\mc{S}_0$. Here 
$(\phi,h)$ are in a small neighbourhood of $({\rm Id},h_0)$ so that $\phi^*h\in \mc{S}$. Now 
$T_{\phi^*h}\mc{S}={\rm Im}(D\Pi_{\Pi^{-1}(\phi^*h)})$.  Then 
we want to find a smooth tame map $((\phi,h),\dot{r}) \mapsto  (X,\dot{h})\in {\rm lie}(\mc{D}_0(N))\x T_{h}\mc{S}_0$ 
so that 
\[ \phi^*(L_Xh+\dot{h})=D\Pi_{r}(\dot{r})\]
where $\tra_{h_0}(\dot{r})=0$ and $r=\Pi^{-1}(\phi^*h)\in B_{h_0}$. Using the chart $\Pi$, we translate this 
into the problem of solving for $(X,\dot{r}_0)$ in
\begin{equation}\label{firstequation} 
 D\Pi^{-1}_r(L_Xh)+\dot{r}_0 =D\Pi^{-1}_{r}\Big((\phi^{-1})^*D\Pi_{r}(\dot{r})\Big)
 \end{equation}
with $h=\Pi(r)$ and $\tra_{h_0}(\dot{r}_0)=0,\delta_{h_0}(\dot{r}_0)=0$. Applying $\delta_{h_0}$, this leads to  
\begin{equation}\label{reduction}
\delta_{h_0}D\Pi^{-1}_r(L_Xh)=\delta_{h_0}D\Pi^{-1}_{r}\Big((\phi^{-1})^*D\Pi_{r}(\dot{r})\Big)
\end{equation}
First, observe that the map $F_h: X\mapsto \delta_{h_0}D\Pi^{-1}_r(L_Xh)$ is a pseudo-differential operator
on $N$ of order $2$ acting on vector fields, depending smoothly in a tame way on $h$. 
We now state the following Lemma, the proof of which is defered  below the proof of this Proposition.
\begin{lem}\label{implytrace0}
Let $h_0^s\subset \mc{S}$ be a one-parameter smooth family of metrics on $N$, i.e. with  $v_n(h_0^s)=\int_{N}v_n(h_0^s){\rm dvol}_{h_0^s}$,  such that $h^0_0=h_0$. Let
$\dot{h}_0:=\pl_sh_0^s|_{s=0}$ and assume that $\delta_{h_0}(\dot{h}_0)=0$, then ${\rm Tr}_{h_0}(\dot{h}_0)=0$. Moreover we have 
$D\Pi_{0}={\rm Id}$. 
\end{lem} 
We denote by $\Psi^{m}(N)$ the class of pseudo-differential operator of order $m$ on $N$ (acting on vector fields).
The operator $F_{h_0}$ being equal to the elliptic differential operator $F_{h_0}(X)=\delta_{h_0}L_Xh_0$  of order 
$2$, we deduce by smoothness of $F_h$ with respect to $h$ that $F_h\in \Psi^{2}(N)$ is  elliptic when 
$||h-h_0||_{H^L}$ is small enough (for some $L$). 
The operator  $F_{h_0}$ is elliptic self-adjoint and invertible from $H^2$ to $L^2$ since there is no Killing field 
on $(N,h_0)$ by \cite{Bo}, therefore $F_h$ is also invertible from $H^2$ to $L^2$ with 
inverse an operator $F_h^{-1}\in \Psi^{-2}(N)$ and $(h,X)\to F^{-1}_h(X)$ is a tame map.
This allows to solve for $X$ in \eqref{reduction}. Note that $X$ is uniquely determined, according to the argument we used. 
Then $\dot{r}_0$ is obtained by \eqref{firstequation}, it has $\tra_{h_0}(\dot{r}_0)=0$ by the property of $D\Pi^{-1}$
and it satisfies $\delta_{h_0}(\dot{r}_0)=0$ by construction of $X$ solving \eqref{reduction}.
We can therefore apply the Nash-Moser inverse function theorem to deduce that $\mc{S}_0$ is a slice for the $\mc{D}_0(N)$ action on $\mc{S}$. 
\end{proof}

\noindent\emph{Proof of Lemma \ref{implytrace0}}. Here we take a family of Poincar\'e Einstein metrics  $g^s=(dx^2+h_x^s)/x^2$ near the conformal infinity $x=0$. 
We use the notation in the proof of Theorem \ref{variationvol} and remove the superscript $s$ when $s$ is set to be $0$.
We are going to show that $\dot{v}_n=c_n \dot{v}_2$ for some $c_n\not=0$.
To prove that, for the moment we do not assume that $v_n(h_0^s)$ is constant and we simply assume that $g^s$  is Poincar\'e-Einstein for $s\not=0$ with $g^0=g$.  
Using $\pl_xv^s_x=\demi v_x^s\tra_{h_x^s}(\pl_xh_x^s)$, differentiating this identity with respect to $s$ at $s=0$,
one has modulo $o(x^n)$
\[ \sum_{j,k\leq \ndemi}\dot{v}_{2k}j\gamma_{j}x^{2k+2j-1}+
\sum_{i,j,k\leq n/2}x^{2i+2j+2k-1}v_{2k}j(\alpha_j\tra (\dot{H}^{2i})+\beta_i\tra(\dot{H}_{2j}))= \sum_{j\leq \ndemi} 2j\dot{v}_{2j}x^{2j-1};\]
notice that we have used that $H^{2j}=\alpha_{2j}{\rm Id}$, $H_{2j}=\beta_{2j}{\rm Id}$ for some $\alpha_{2j},\beta_{2j}\in\rr$, and 
$\gamma_i$ are some constants. Then by a straightforward induction and using \eqref{multiples}, we deduce that $\dot{v}_{2j}=c_{2j}\dot{v}_2$ for some $c_{2j}\in\rr$ if $1\leq j\leq \ndemi$.  To compute $c_{n}$, we notice that (the obstruction tensor $k=0$ for an Einstein metric)
\[ \pl_s\Big(\int_{N}v^s_{n}{\rm dvol}_{h^s_0}\Big)|_{s=0}=\tfrac{1}{4}\cjg k,\dot{h}_0\cjd=0= \int_N \dot{v}_n{\rm dvol}_{h_0}+\int_N\frac{v_n}{2}\tra_{h_0}(\dot{h}_0){\rm dvol}_{h_0}\]
 and thus $c_n\int_{N}\dot{v}_2=-\tfrac{1}{2}v_n\int_{N}\tra_{h_0}(\dot{h}_0)$; but since $v^s_2=-\tfrac{1}{4(n-1)}{\rm Scal}_{h_0^s}$, we can use \eqref{vscal} to deduce that $\int_N\dot{v_2}=-\tfrac{1}{4}\int_{N}\tra_{h_0}(\dot{h}_0)$, and 
since $\tra_{h_0}(\dot{h}_0)$ can be chosen so that its integral is not $0$, we obtain that $c_n=2v_n$. 
Now we come back to our setting where $g^s$ is AHE with $\delta_{h_0}(\dot{h}_0)=0$. Since 
$\dot{v}_n=2v_n\dot{v}_2$ and $v_n\not=0$ (by Lemma \ref{vncst}), we deduce from \eqref{vscal}
\[\dot{v}_n=0 \iff (\Delta-\la(n-1)) \tra_{h_0}(\dot{h}_0)=0 \iff \tra_{h_0}(\dot{h}_0)=0.\]
if ${\rm Ric}_{h_0}=\la(n-1)h_0$. 
This concludes the first part of the proof since $\dot{v}_n=\tfrac{1}{4}\cjg k,\dot{h}_0\cjd=0$ if $v_n(h_0^s)=\int_{N}v_n(h_0^s){\rm dvol}_{h_0^s}$. 

Let us finally show that $D\Pi_0={\rm Id}$ where $\Pi$ is defined in \eqref{PI}. Let $\dot{r}_0$ be divergence-free and trace free with respect to $h_0$, then by the discussion above, we have 
$(Dv_n)_{h_0}(D\Pi_0(\dot{r}_0))=0=(Dv_2)_{h_0}(D\Pi_0(\dot{r}_0))$ and by
 \eqref{DPir}, we have  also have $D\Pi_0(\dot{r}_0)=\dot{r}_0+2(D\omega_0)_0(\dot{r}_0)h_0$. By \eqref{vscal}
we deduce that $(\Delta_{h_0}-\la n)(D\omega_0)_0(\dot{r}_0)=0$ 
and thus $D\Pi_0(\dot{r}_0)=\dot{r}_0$. If now 
$X$ is a vector field so that $\tra_{h_0}(L_Xh_0)=0$, we set $\phi_t=e^{tX}$ and write 
$\phi_t^*h_0=\Pi(r_t)$ for some $r_t$ with $\tra_{h_0}(r_t)=0$. Then, differentiation gives 
$L_Xh_0=D\Pi_0(\dot{r})$ and since $\Pi(r_t)=e^{2\omega_0(r_t)}(h_0+r_t)$, we also deduce 
$L_Xh_0=D\Pi_0(\dot{r})=2(D\omega_0)_0(\dot{r})h_0+\dot{r}$. Taking the trace with respect to $h_0$,
we obtain $2(D\omega_0)_0(\dot{r})=0$ and $\dot{r}=L_Xh_0=D\Pi_0(L_Xh_0)$. Since any trace free tensor $\dot{r}$ 
can be decomposed as a sum $L_Xh_0+\dot{r}_0$, this achieves the proof that $D\Pi_0={\rm Id}$.
\qed

\subsection{Examples} \label{ssc:fuchsian}
We give two examples where these Hypotheses are satisfied.

\textbf{The case $n=2$}. This is our archetypal motivation. We consider here a 3-manifold $M$ which admits a convex co-compact hyperbolic
metric --- this is the same, in dimension $3$, as an AHE metric. Then $N=\partial M$ is
the disjoint union of a finite set of closed surfaces of genus at least $2$.
The classical Ahlfors-Bers theorem \cite{ahlfors:riemann,Be}, extended by Marden \cite{marden,Ma}, 
gives the map $\Phi$ of Hypothesis \ref{hypo1}, for any choice of slice $\mc{S}_0$ (in fact the map 
$\Phi$ is well defined on Teichm\"uller space in this case).
Moreover, we have seen that $v_2=-\tfrac{1}{4}\scal_{h_0}$,
so given a metric $h_0$ on $N=\partial M$, it has $v_2=-\pi \chi(N)$ if and only 
it has constant curvature. Since there is a unique constant curvature metric with volume $1$ on each
connected component of $\partial M$, Hypothesis \ref{hypo2} is also satisfied.\\

\textbf{Fuchsian-Einstein manifolds}.
We now recall a particularly simple type of AHE manifolds. 
Let $(N,h_0)$ be a closed Einstein manifold with $\ric_{h_0}=-(n-1) h_0$. 
Its conformal class will be denoted $[h_0]$ as before. 

We consider the product $M=\rr\times N$, with the warped product metric:
\begin{equation}\label{warped}
g:= dt^2+\cosh^2(t) h_0.
\end{equation}
We will call {\it Fuchsian} a Riemannian manifold of this type, the reason being that, for $n=2$, we find
precisely the Fuchsian hyperbolic 3-manifolds, that is, quotients of $\hh^3$ by co-compact Fuchsian groups $\Gamma\subset \mathrm{PSL}_2(\rr)\hookrightarrow \mathrm{PSL}_2(\cc)$, or equivalently hyperbolic $3$-manifolds which are topologically
the product of a surface of genus at least $2$ by an interval, and which contain a closed totally geodesic
surface.


It follows directly from \eqref{ricec} that $\Ric_g=-ng$. To prove that $(M,g)$ is actually AHE, set $x=2e^{-|t|}$ away from $t=0$. In this new variable,
$$ g = \frac{dx^2}{x^2} + \left(\frac{1+\frac14x^2}{x}\right)^2 h_0~, $$
so $g$ is Poincar\'e-Einstein. 
The subset corresponding to $t=0$ is a closed totally geodesic hypersurface since the warping function is even. 

Write $g=dt^2+f^2(t)h_0$ with $f(t)=\cosh(t)$. Let $v,w$ be some ($t$-independent) vector fields on $N$ 
and let $V:=f^{-1}v$, $W:=f^{-1}w$ and $T=\partial/\partial t$, then one has by a direct computation
\begin{align}\label{natp}
\nabla_TT=0, && \nabla_{V}T=f^{-1}f'V,&&
\nabla_TV=0, && \nabla_{V}W=f^{-2}\nabla^N_{v}w-f^{-1}f'\cjg v,w\cjd_{h_0} T.
\end{align}	
This implies for $X,Y$ tangent to $N$
\begin{align}
 \label{eq:Rvv}
  R_{X,T}T ={}& -X~,&  R_{X,Y}T=0.
\end{align}
Moreover, for $X,Y,Z,W$ tangent to $N$,  
\begin{equation}\label{eccurvpt} 
\cjg R(X,Y)Z,W\cjd_g =\cjg R^{h_0}(X,Y)Z,W\cjd_g-\frac{(f')^2}{f^2}(\cjg Y,Z\cjd_g\cjg X,W\cjd_g-\cjg X,Z \cjd_g\cjg Y,W, \cjd_g)
\end{equation}
where $R^{h_0}$ is the Riemann tensor on $(N,h_0)$,
showing in particular that if $h_0$ has non-positive sectional curvature, then $g$ also has non-positive sectional curvature.

The conformal boundary of $M$ is the disjoint union of two copies of $(N,[h_0])$, one corresponding to $t=-\infty$ and the other 
to $t=\infty$. We call these two components of the conformal boundary $(N_{\pm},[h_0])$.

We summarize the discussion in the 
\begin{lem}
Let $(N,h_0)$ be a closed Einstein manifold with non-positive sectional curvatures and negative Ricci curvature, 
and let $M=N\times \rr$ be endowed with the warped product metric $g=dt^2+\cosh^2(t)h_0$. Then $g$ satisfies the assumptions of Proposition \ref{rk:rigidity}.
\end{lem}

\subsection{Poincar\'e-Einstein ends as cotangent vectors to conformal structures}

We then use the description in Section \ref{ssc:conformal}: $\mc{T}(N)$ correspond to the quotient
of the space of metrics $\mc{M}(N)$ by the group $C^\infty(N)\rtimes\mc{D}_0(N)$ of diffeomorphism of $N$ isotopic to 
the Identity, or equivalently it is the space of conformal classes up to diffeomorphisms isototic to the Identity. We will
work near a metric $h_0\in \mc{M}(N)$ where $\mc{T}(N)$ can be locally represnted by a slice.
By the discussion of Section \ref{ssc:conformal}, $T^*_{[h]}\mc{T}(N)$ can be identified with the space of 
trace-free and divergence-free (for $h$) symmetric 2-tensors on $\partial M$.

Let $\cE$ be the space of Poincar\'e-Einstein ends (with conformal boundary $N$), i.e.  the set of products
$N\times (0,\eps)_x$ equipped with a Poincar\'e-Einstein metric $g=(dx^2+h_x)/x^2$; 
here $\eps>0$ is not relevant since a Poincar\'e-Einstein metric
is defined only up to $\mc{O}(x^\infty)$. The group $\mc{D}_0(N)$ acts naturally on $\mc{E}$ by 
$\phi.g=(dx^2+(\phi^{-1})^*h_x)/x^2$ where $(\phi^{-1})^*h_x$ is just the pull-back of $h_x$ by $\phi^{-1}$, 
viewed as a metric on $N$.
The group $C^\infty(N)$ also acts on $\mc{E}$ as follows: $\omega_0.g:=(d\hat{x}^2+\hat{h}_{\hat{x}})/\hat{x}^2$ where
$\hat{x}$ is the geodesic boundary defining function associated to the conformal representative 
$e^{2\omega_0}h_0$, in the sense of Lemma \ref{geodesicbdf}. This induces an action of $C^\infty(N)\rtimes\mc{D}_0(N)$ by 
$(\omega_0,\phi).g:=\omega_0.(\phi.g)$. This group action corresponds to the action of the group of those diffeomorphisms 
which map a Poincar\'e-Einstein end to another one: this is the natural gauge group of $\mc{E}$.
 
\textbf{Case $n$ odd}.
We observe that the action of an element $(f,\phi)\in C^\infty(N)\rtimes\mc{D}_0(N)$ 
on a Poincar\'e-Einstein end $g$ transforms the pair $(h_0,h_n)$ in the expansion of $g$
into the pair $(e^{2\omega_0} (\phi^{-1})^*h_0, e^{(2-n)\omega_0}(\phi^{-1})^*h_n)$ in the expansion of $(\omega_0,\phi).g$. 
This is easy to show: the $\mc{D}_0(N)$ action is clear, as for the conformal action, 
it comes from the fact that $h_n$ is the coefficient of the first odd power of $x$ in the expansion of $g$ 
and that the geodesic boundary defining function $\hat{x}$ associated to $e^{2\omega_0}h$ is of the form $\hat{x}=xe^{\omega_x}$
with $\omega_x$ an even function of $x$ up to $\mc{O}(x^{n+2})$ (see for e.g. Lemma 2.1 in \cite{GuDMJ} and its proof). Notice that the action $(f,\phi).g$ corresponds
exactly to the action \eqref{actionsurT*}
of $(\omega_0,\phi)$ on $T^*\mc{M}(N)$ if we view $(h_0,h_n)$ as an element in  $T^*\mc{T}(N)$ 
(here $h_n$ is a divergence-free trace-free tensor). We therefore deduce 
\begin{prop}
If $n$ is odd, over the points where $\mc{T}(N)$ has a Fr\'echet manifold structure, 
the space $\mc{G}\backslash \mc{E}$ of Poincar\'e-Einstein ends, up to the gauge group 
$\mc{G}=C^\infty(N)\rtimes\mc{D}_0(N)$, identifies naturally to the cotangent space $T^*\mc{T}(N)$ of the set of conformal structures.
\end{prop}

\textbf{Case $n$ even}.
In even dimension, the pairs $(h_0,h_n)$ representing a Poincar\'e-Einstein ends are not 
identified directly to an element in $T_{[h_0]}^*\mc{T}(N)$ as for n odd. Indeed, it is easy to verify 
that for a change of conformal representative $\hat{h}_0=e^{2\omega_0}h_0\in [h_0]$, the formally undetermined
term $\hat{h}_n$ in the end is of the form $\hat{h}_n=e^{(2-n)\omega_0}h_n+P(\omega_0,h_0)$ where $P$ is some non-linear differential operator.  Moreover $h_n$ is neither trace free nor is divergence free with respect to $h_0$.
However, Theorem \ref{variationvol} and Corollary \ref{divGn} tell us that if $h_0$ satisfies 
$v_n(h_0)=\int_{N}v_n(h_0){\rm dvol}_{h_0}$ then there is a formally determined tensor $F_n=F_n(h_0)$ such that 
the trace-free part $G_n^\circ=G_n-\tfrac{v_n}{2n}h_0$ of $G_n=-\tfrac{1}{4}(h_n+F_n)$
is divergence-free. By the description \eqref{T^*TN}Ê of  $T^*\mc{T}(N)$ in Section \ref{ssc:conformal}, 
we can thus see $G^\circ_n$ as a cotangent vector to $h_0$. 
We then obtain 
\begin{prop}
Let $n$ be even and assume Hypothesis \ref{hypo1}. 
Near the base point $[h_0]\in \mc{T}(N)$, we can identify the cotangent space $T^*\mc{T}(N)$ of the set of conformal structures to the space $\mc{G}\backslash \mc{E}$ of Poincar\'e-Einstein ends as follows: if 
$h\in \mc{S}_0$ and $r\in C^\infty(N,S^2N)$ with  ${\rm Tr}_h(r)=0$, $\delta_h(r)=0$, we assign 
to the cotangent data $(h,r)\in T^*_{[h]}\mc{T}(N)$ the Poincar\'e-Einstein end 
$(h,-4r-F_n(h)-\tfrac{2v_n(h)}{n}h)$.
\end{prop}

\textbf{Example: $n=2$}.
In this case, $N$ is a closed surface of genus at least $2$, and
$\cE$ is the space of hyperbolic ends on $N\times (0,\infty)$. 
Hyperbolic ends on $N\times (0,\infty)$ are in one-to-one correspondence to complex projective
structures on $N$. Let $\cCP$ the space of complex projective structures on $N$. Given 
$\sigma\in \cCP$, one can consider the underlying complex structure $c$, and the Fuchsian
complex projective structure $\sigma_0$ obtained by applying Riemann uniformization to $c$.
Let $\phi$ be the holomorphic map isotopic to the identity between $(N,\sigma_0)$ to
$(N, \sigma)$, and let $q=\cS(\phi)$ be the Schwarzian derivative of $\phi$.

\begin{lem}
Let $h$ be the hyperbolic metric in the conformal class $c$, then, for all $g\in \cE$, 
$(h,\demi \re(q))$ is the associated cotangent data to $c$.
\end{lem}

\begin{proof}
It is proved in \cite[Lemma 8.3]{KrSc} that $\II^*_0=-\re(q)$, where $\II^*$ is the ``second
fundamental form at infinity'' considered in \cite{KrSc} and $\II^*_0$ is its traceless part.
However comparing the expressions of the hyperbolic metric at infinity in terms of $h_2$ used
here, and in terms of $\II^*$ as in \cite{KrSc}, shows that $h_2=2\II^*$. Finally we have seen
in Section \ref{ssc:n=2} that $h_2^\circ = -4G_2^\circ$. The result follows.
\end{proof}

\subsection{Lagrangian submanifold in $T^*\mc{T}(N)$}

We now come back to the situation where Hypothesis \ref{hypo0}, \ref{hypo2} and \ref{hypo1} apply (we use the same notations as there). 
Using again that $T^*\mc{T}(N)$ near $[h_0]$ is represented by  \eqref{T^*TN},
 we define the modified Dirichlet-to-Neumann map 
\begin{equation}\label{NXi} 
\mc{N}_\Xi: h\in \mc{S}_0 \mapsto G_n^\circ(h)\in T_{[h]}^*\mc{T}(N)\end{equation}
where $G_n^\circ(h)$ is the divergence-free/trace-free tensor $G_n^\circ$ associated to the Poincar\'e-Einstein 
end of the AHE metric $\Xi(h)$.
 
\begin{prop} \label{pr:closed}
The section $\mc{N}_\Xi$ is an exact $1$-form on the slice $\mc{S}_0$.  
\end{prop}
\begin{proof}
By linearizing the identity 
$v_n(h)=\int_Nv_n(h){\rm dvol}_h$ valid for every 
$h\in \mc{S}_0$, we get $\int_Nv_n(h)\Tr_h(\dot{h}){\rm dvol}_h=0$ if 
$\dot{h}\in T_h\mc{S}_0$, and therefore $\cjg G_n,\dot{h}\cjd=\cjg G_n^\circ ,\dot{h}\cjd$.
By Theorem \ref{variationvol} we deduce that $\mc{N}_{\Xi}$ is the differential of the map 
$h\mapsto \vol_R(M,\Xi(h);h)$. 
\end{proof}

\begin{cor} \label{cr:symplectic}
The image of $\mc{N}_\Xi$ is a Lagrangian Fr\'echet submanifold in $T^*\mc{T}(N)$.
\end{cor}
\begin{proof}
The image is a submanifold since it is the image of a smooth section. 
It is isotropic for the symplectic form $\Omega$ of \eqref{sympform} since the section is an exact 
form and $\Omega$ is the exterior derivative of the Liouville $1$-form. Moreover, it is maximal 
isotropic since it is diffeomorphic to the base by the projection. 
\end{proof}

The following corollary is the analog in our higher-dimensional setting of McMullen's Kleinian reciprocity,
see \cite[Theorem 9.1]{mcmullen}.

\begin{cor} \label{cr:kleinian}
Let $h\in \mc{S}_0$, and let $u,v\in T_{h}\mc{S}_0$. Let $u^*,v^*$ be the corresponding first-order variations 
of $G_n^\circ$, so that $u^*, v^*\in T^*_{[h]}\mc{T}(N)$. Then
$$ \langle v, u^* -\tfrac{n}{2} \cjg \mc{N}_\Xi(h),u\cjd h \rangle  
= \langle u,v^*-\tfrac{n}{2} \cjg \mc{N}_\Xi(h),v\cjd h \rangle   $$
where $\langle, \rangle$ is the $L^2$ pairing with respect to $h$. 
Equivalently, the linearization $d\mc{N}_{\Xi}$ of $\mc{N}_\Xi$ is such that
\[ (d\mc{N}_\Xi)_h -\tfrac{n}{2}\cjg \mc{N}_\Xi(h),\cdot\cjd h \]  
is self-adjoint.
\end{cor}

\begin{proof}
This is a direct translation of Corollary \ref{cr:symplectic} using the definition of the cotangent
symplectic structure induced by \eqref{sympform} on $T^*\mc{T}(N)$.
\end{proof}

\textbf{Quasifuchsian reciprocity for Poincar\'e-Einstein manifolds}. 
We now consider a more specific setting, analogous to the situation occuring for the quasifuchsian reciprocity
for 3-dimensional hyperbolic manifolds, see \cite{mcmullen}. 
We consider a manifold $M$ such that $\partial M$ has two connected components,
$N_+$ and $N_-$. We denote by $\mc{M}(N_\pm)$ and $\mc{T}(N_\pm)$ the space of Riemannian metrics and 
the space of conformal structures on $N_\pm$, and we assume that
Hypothesis \ref{hypo0}, \ref{hypo1} apply and Hypothesis \ref{hypo2} applies 
on $N_+,N_-$ separately with $v_n\not=0$, i.e. $\mc{S}_0=\mc{S}_0^-\x \mc{S}_0^+$.

Given $h=(h_-,h_+)\in \mc{S}_0$,  let $\mc{N}_{\Xi}^+(h)\in T^*_{h_+}\mc{T}(N_+)$ 
and $\mc{N}_\Xi^-\in T^*_{h_-}\mc{T}(N_-)$ be $N_\pm$ component of $\mc{N}_\Xi(h)$.
For fixed $h_-$ we have a section $\mc{N}^+_\Xi(h_-,\cdot)$ of $T^*\mc{T}(N_+)$, while for fixed $h_+$ we have a section $\mc{N}_\Xi(\cdot,h_+)$
of $T^*\mc{T}(N_-)$. 

For fixed $h=(h_-, h_+)$, we now consider the linear maps
\begin{align*}
  \phi_{h_+}:  T_{h_-}\mc{S}^-_0 \to {}& T^*_{h_+}\mc{T}(N_+), &v_- \mapsto {}& (d\mc{N}^+_\Xi)_h(v_-,0)-
  \tfrac{n}{2} \cjg \mc{N}_\Xi^-(h),v_-\cjd h_+\\
\phi_{h_-}:  T_{h_+}\mc{S}^+_0 \to {}& T^*_{h_-}\mc{T}(N_-),&
  v_+ \mapsto {}& (d\mc{N}^-_\Xi)_h(0,v_+)-
  \tfrac{n}{2} \cjg \mc{N}_\Xi^+(h),v_+\cjd h_-
\end{align*} 
\begin{prop} \label{pr:reciprocity}
$\phi_{h_-}$ and $\phi_{h_+}$ are adjoint.   
\end{prop}
\begin{proof} This is simply a particular case of Corollary \ref{cr:kleinian}.
\end{proof}
\section{The Dirichlet-to-Neumann map for the Fuchsian-Einstein case and Hessian of the renormalized volume}\label{secDNmap}

In this last section we compute the Hessian of the renormalized volume at a Fuchsian-Einstein metric $(M=\rr\x N,g)$,
when $n$ is odd, and when $n=2,4$. In the latter case we will consider the renormalized volume as a function on 
the slice $\mc{S}\subset \mc{M}(N)$ of metrics satisfying $v_n=\int_{N}v_n$ near $h_0$ Einstein 
with negative Ricci curvature.

\subsection{Hessian of ${\rm Vol}_R$ at the Fuchsian locus when $n=2$}
It is instructive to do first the computation for $n=2$. Let $g=dt^2+\cosh(t)^2 h_0$ be a Fuchsian metric on $M=N\times\rr_t$
for a hyperbolic surface $(N,h_0)$. The conformal boundary consists of $(M,g)$ 
is $\pl M=N_+\sqcup N_-$ (corresponding to $t\to \pm \infty$) where each $N_\pm$ is $N$ equipped 
with the conformal class of $h_0$. The geodesic boundary defining function associated to 
$h_0$ is $x:=2e^{-|t|}$ near $t=\pm \infty$; the metric $g$ takes the form near $\pl M$
\begin{align*}
g=x^{-2}(dx^2+h_0+\tfrac12 x^2 h_0+\tfrac{1}{16} x^4 h_0)\text{ as } x\to 0. 
\end{align*}
We have the following result
\begin{prop}\label{hessianfuchsian}
Let $h_0$ be a hyperbolic metric on a Riemann surface $N$ with genus $\geq 2$. We identify Teichm\"uller space $\mc{T}(N)$ of $N$ with a slice of hyperbolic metrics, with tangent spaces at each point 
the space of divergence-free/trace-free tensors. 
Let $\Phi: h_-\mapsto \Phi(h_-)$ be the Bers map sending a hyperbolic metric $h_-\in \mc{T}(N)$ on $N$ 
to the quasifuchsian 
hyperbolic metric on $N\x \rr_t$ with conformal boundary $h_0$ at $N_+$, $h_-$ at $N_-$.  
Then the map $V_{h_0}: h_-\mapsto {\rm Vol}_R(M,\Phi(h_-);(h_0,h_-))$ has a unique critical point at $h_-=h_0$ on $\mc{T}(N)$ and the Hessian there is 
\[{\rm Hess}_{h_0}(V_{h_0})(k,k)=\frac{1}{8}\int_{N}|k|_{h_0}^2 {\rm dvol}_{h_0}, \quad k\in T_{h_0}\mc{T}(N).\]
\end{prop} 
\begin{proof} The fact that the Fuchsian metric $g$ is a critical point is a consequence of the fact that the trace-free 
part $G_2^\circ$ of $G_2$ is $0$ by \eqref{G2F2}. To see that it is the unique critical point, we claim that 
for a critical point quasifuchsian metric $g=\Phi(h_-)$, the trace-free part $G_2^\circ$ of $G_2$ at both conformal boundaries 
is $0$, which means that the $2$ hyperbolic ends are of the form 
$x^{-2}(dx^2+(1+\tfrac{x^2}{4})h_\pm)^2$ for $h_+=h_0$ and $h_-$ some hyperbolic metric; 
thus the quasifuchsian metric $g$ would have
two embedded totally geodesic surfaces if $h_+\not=h_0$, and this is not possible by topological reasons since by doubling the region 
bounded by theses surfaces, we would get 
a closed $3$ dimensional hyperbolic manifold which is $S^1\x N$ with $2$ embedded totally geodesic surfaces, thus 
toroidal, contradicting Thurston hyperbolisation theorem.

Next we compute the Hessian. We deform $g$ by a $1$-parameter family of quasifuchsian metrics $g^s$ 
by means of a divergence-free/trace-free tensor $k$ as follows:
\begin{align*}
g^s:= dt^2+h_0+e^{-2t} h_2^s +\tfrac14 e^{-4t} (h_2^s)^2,&&h_2^s=\tfrac12 h_0+sk.
\end{align*}
This amounts to changing the conformal class on $N_-$ only. We denote conformal representatives 
in the conformal boundary by pairs $(h^s_+,h^s_-)$ corresponding to the 
components $N_\pm$. 
For small $s$, the expression for $g^s$ 
makes sense for all $t\in\rr$; at $N_+$, $x$ induces the conformal representative $h^s_+=h_0$ 
and at $N_-$ one has $h_-^s= 4(h_2^s)^2$ since the metric near $t=-\infty$ is 
\[
g^s=x^{-2}(dx^2 + h_-^s + x^2 h^s_2+ x^4h_4^s).
\]
Notice that $h_-^0=h_0$ is hyperbolic, but for other values of $s$ it is not. 
The variation formula from Section \ref{ssc:n=2} gives for $s$ near $0$ (with $\dot{h}^s_-=\pl_sh_-^s$) 
\[-4\pl_s\rvol(M,g^s;(h_0,h_-^s))=\int_N \cjg \dot{h}_-^s,h^s_2-\tr_{h_-^s}(h^s_2)h_-^s)\cjd_{h_-^s}{\rm dvol}_{h^s_-}.\]
We have
\begin{align*}
h_-:=h_-^0=h_0,&& \dot{h}_-:=\dot{h}_-^0=4k,&& \pl_s^2h_-^s|_{s=0}=8k^2,&&\tr_{h_-^s}(h^s_2)=1+\mc{O}(s^2), 
\end{align*}
and  we compute (with the dot notation for $\pl_s|_{s=0}$)
\begin{align*}
-4\pl^2_s\rvol(M,g^s;(h_0,h_-^s))|_{s=0}= \cjg \dot{h}_-,\dot{h}_2\cjd -\tfrac12 \cjg  \pl_s^2h^s_-|_{s=0},h_0\cjd=0.
\end{align*}
We are interested in the renormalized volume where the boundary families are uniformized to have scalar curvature $-2$. This means, we must consider $\rvol(M,g^s;(h_0,\hat{h}^s_-))$ where $\hat{h}_-^s:=e^{2\omega_0^s}h_-^s$ is the unique hyperbolic metric in the conformal class of $h_-^s$. Then $\pl_s\hat{h}_-^s$ is the sum of a Lie derivative of $\hat{h}^s$ and 
a divergence free/trace free tensor, in particular
$\int_N\tr_{\hat{h}_-^s}(\pl_s\hat{h}_-^s){\rm dvol}_{\hat{h}_-^s}=0$. From this we can derive the identity 
\[\int_N \dot{\omega}_0^s{\rm dvol}_{h_-^s}=4s\int_N |k|^2_{h_0}{\rm dvol}_{h_0}+\mc{O}(s^2).\] 
Since ${\rm Scal}_{h^s_-}=-2\tr_{h_-^s}(h^s_2)=-2+\mc{O}(s^2)$, it follows that $\omega^s_0=\tfrac12 s^2 \alpha+o(s^2)$ 
for some $\alpha$ with $\int_N\alpha=4\int_N|k|^2_{h_0}{\rm dvol}_{h_0}$.
Proposition \ref{functionaldim2} shows that 
\[\rvol(M, g^s;(h_0,\hat{h}^s_-))=-\frac{1}{4}\int_M (|\nabla\omega_0^s|+{\rm Scal}_{h^s_-}\omega_0^s){\rm dvol}_{h^s_-}.\]
The only term of order $2$ which survives is 
\[\rvol(M, g^s;(h_0,\hat{h}^s_-))=s^2\int_N |k|_{h_0}^2{\rm dvol}_{h_0}+o(s^3).\]
This computes the Hessian of the renormalized volume at $(h_0,h_0)$ in the direction $(0,4k)$. 
\end{proof}

\subsection{Higher dimensions}

\textbf{Second variation of the volume in terms of $G_n$ and $h_n$}
Let us consider a family of AHE metrics $g^s$ (for $s$ near $0$) on $M=\rr_t \x N$ with $N$ compact and $g^0=g$ with 
\[g=dt^2 +\cosh^2(t)h_0\] 
where ${\rm Ric}_{h_0}=-(n-1)h_0$. The conformal infinity of $(M,g)$ is $\pl M=N_+\sqcup N_-$ where each $N_\pm$ is $N$ equipped 
with the conformal class of $h_0$. Notice that $x:=2e^{-|t|}$, defined outside $t=0$, is the geodesic boundary defining function
associated to the conformal representative $h_0$ on $\pl M$. 
When $n$ is even,  we  choose (for $s$ near $0$) the smooth family $h_0^s$ of metrics on $\pl M$
so that $v_n(h_0^s)=\int_{N}v_n(h_0^s){\rm dvol}_{h_0^s}\not=0$ and $[h_0^s]$ is the conformal infinity of $(M,g^s)$; this is possible
by Corollary \ref{sectionvncst} and implies ${\rm Vol}(N,h_0^s)=1$.

\textbf{Case n odd.}  By a result of Albin \cite{Al} (see Theorem \ref{variationvol}), we have  $\pl_s {\rm Vol}_R(M,g^s;h_0^s)=-\tfrac14\cjg h^s_n,\pl_sh_0^s\cjd_{L^2}$
and since $h_n=0$ at $s=0$,
\[\pl_s^2 {\rm Vol}_R(M,g^s)|_{s=0}=-\frac{1}{4}\cjg \dot{h}_n,\dot{h}_0\cjd_{L^2}.\]

\textbf{Case n even.} By Theorem \ref{variationvol}, we have  $\pl_s {\rm Vol}_R(M,g^s;h_0^s)=\cjg G^s_n,\pl_sh_0^s\cjd_{L^2}$ and 
$\tra_{h_0^s}(G^s_n)=\frac12 v_n(h_0^s)$ is constant, it follows that $\pl_s {\rm Vol}_R(M,g^s;h_0^s)=
\cjg (G_n^s)^\circ,\pl_sh_0^s\cjd$ where $(G_n^s)^\circ$
is the trace-free part of $G_n^s$, which vanishes at $s=0$ by Lemma \ref{FnEinstein}. Hence
\[
\pl_s^2 {\rm Vol}_R(M,g^s;h_0^s)|_{s=0}=\cjg \dot{G}^\circ_n,\dot{h}_0\cjd_{L^2}
\]
where $\dot{G}^\circ_n:=\pl_s[(G_n^s)^\circ]|_{s=0}$. Moreover since $\tra_{h^s_0}(G^s_n)=\demi v_n(h_0^s)$, 
\begin{equation}\label{dotgn}
\pl_s^2 {\rm Vol}_R(M,g^s;h_0^s)|_{s=0}=\cjg\dot{G}_n,\dot{h}_0\cjd_{L^2}
-\frac{v_n}{2n}|\dot{h}_0|_{L^2}^2-\frac{1}{2n}\int_{N}\dot{v}_n\tra_{h_0}(\dot{h}_0){\rm dvol}_{h_0}.
\end{equation}

\textbf{Variation of the local term $F_n$ when $n=4$}.
Since $-4G_n=h_n+F_n$ where $F_n$ is local in terms of $h_0$, we have to compute the variation $\dot{F}_n$.
In general even dimension $n$, we do not have a formula for $F_n$, thus we will restrict to $n=4$.

Will will now assume that $\dot{h}_0$ is divergence free, so that by Lemma \ref{implytrace0},
\begin{equation}\label{assumdiv0} 
\tra_{h_0}(\dot{h}_0)=0, \quad \delta_{h_0}(\dot{h}_0)=0.
\end{equation}

Using this, we compute $\cjg \dot{F}_n,\dot{h}_0\cjd$ for $n=4$.
\begin{lem}\label{dotFn}
In dimension $n=4$, assuming \eqref{assumdiv0}, we have
\begin{align*}
\cjg \dot{F}_4,\dot{h}_0\cjd_{L^2}=(\tfrac{1}{8}-v_4)|\dot{h}_0|_{L^2}^2+\tfrac12 \cjg\dot{k},\dot{h}_0\cjd_{L^2},&& {\rm Tr}_{h_0}(\dot{F}_4)=0.\end{align*}
\end{lem}
\begin{proof} We recall that $F_{4}^s=-\demi (h^s_2)^2+\frac{1}{4}\tra_{h^s_0}(h_2^s)h^s_2-v_4^sh_0^s+\demi k^s$.
Using that $h_2=\demi h_0$ and ${\rm Scal}_{h_0}=-12$, we obtain  
\[\cjg \dot{F}_4,\dot{h}_0\cjd =-\cjg\frac{\dot{{\rm Scal}}}{48},\dot{h}_0\cjd+(\tfrac{1}{8}-v_4)|\dot{h}_0|^2+\tfrac12 \cjg \dot{k},\dot{h}_0\cjd.\]
Moreover by \eqref{vscal} and the fact that $\tra_{h_0}(\dot{h}_0)=0$ and $\delta_{h_0}(\dot{h}_0)=0$, we have 
$\dot{{\rm Scal}}=0$. Similarly, using $\dot{v}_4=0$, $\tra_{h_0}(\dot{k})=0$, and that  $\tra_{h_0}(\dot{h}_2)=2\dot{v}_2$
is a multiple of $\dot{{\rm Scal}}=0$, we easily see that the trace of $\dot{F}_4$ is $0$.
\end{proof}

\textbf{Bianchi gauge}.
Let us define $\dot{g}:=\pl_sg^s|_{s=0}$, which solves the linearized Einstein equation. 
Since this equation is not elliptic due to gauge invariance (by diffeomorphism actions)
we have to fix a gauge, as is well known in the study of Einstein equation. We shall use Bianchi gauge: 
using for instance Proposition 4.5 in \cite{Fawa}, there exists a smooth vector field $X$ on $M$ so that 
\begin{equation}\label{defofq} 
q:=\dot{g}+L_Xg \textrm{ solves } \delta_g(q)+\tfrac12 d\tra_g(q)=0
\end{equation}
and $q$ has an asymptotic expansion $q=x^{-2}(q_0+\sum_{j\leq n}q_jx^j+x^{n}\log(x)q_{n,1})+o(x^n)$ as $x\to 0$
for some $x$ independent tensors $q_{j},q_{n,1}$ on $[0,\eps)_x\x N$ and
\begin{align*}
q_0=\dot{h}_0 ,&& q_n=\dot{h}_n+T_n\dot{h}_0
\end{align*}
with $T_n$ a differential operator. We notice (see \cite{Fawa}) that $X$ is the vector field dual to the form 
$\omega$ solving 
\begin{equation}\label{Bianchigauge}
(\Delta_{g}+n)\omega=-2\delta_g(\dot{g})-d\tra_g(\dot{g}).
\end{equation}
where $\Delta_g=\nabla^*\nabla$. In dimension $n=4$, we compute $q_4$:
\begin{lem}\label{dotq4}
Let $n=4$, then assuming \eqref{assumdiv0}, we have for $q$ defined by \eqref{defofq}
\[q=\dot{g}+o(x^4),\]
where the $o(x^4)$ is with respect to the norm induced by $g$.
\end{lem}
\begin{proof}
In this proof, all error terms are measured with respect to the metric $g$.
First, since $\tra_{h_0}(\dot{h}_0)=0$, $\tra_{h_0}(\dot{h}_2)=\dot{{\rm Scal}}=0$ and $\tra_{h_0}(\dot{k})=0$ (since $k=0$ for Einstein manifold and $\tra_{h_0^s}(k^s)=0$ for all $s$), we have modulo $o(x^4)$
\[\tra_g(\dot{g})=x^{4}\tra_{h_0}(\dot{h}_4), \quad d\tra_g(\dot{g})=4x^4\tra_{h_0}(\dot{h}_4)\frac{dx}{x}.\] 
For the divergence, we use formula \eqref{div=0} and $\delta_{h_0}(\dot{h_0})=0$, 
we get modulo $o(x^4)$ (we use $x=2e^{-|t|}$)
\[\delta_g(\dot{g})=x^4\tra_{h_0}(\dot{h}_4)\frac{dx}{x}+x^2\delta_{h_0}(\dot{h}_2).\]
But since $\dot{{\rm Scal}}=0$, $\delta_{h^s_0}(h^s_0)=0$ 
and $\delta_{h_0^s}({\rm Ric}_{h_0^s}-\demi{\rm Scal}_{h_0^s})=0$ we have
\[\delta_{h_0}(\dot{h}_2)=\tfrac12 \dot{\delta}({\rm Ric}_{h_0}-\tfrac12{\rm Scal}_{h_0})=\frac{3}{2}\dot{\delta}(h_0)=-\frac{3}{2}\delta(\dot{h}_0)=0.\]
where $\dot{\delta}=\pl_s\delta_{h_0^s}|_{s=0}$. Therefore modulo $o(x^4)$ 
\[-2\delta_{g}(\dot{g})-d\tra_g(\dot{g})=-6x^4\tra_{h_0}(\dot{h}_4)\frac{dx}{x}.\]
Now by Theorem \ref{variationvol}, $\tra_{h_0}(\dot{h}_4)+\tra_{h_0}(\dot{F}_4)=-2\dot{v}_4=-\demi\cjg k,\dot{h}_0\cjd=0$ thus $\tra_{h_0}(\dot{h}_4)=0$ by Lemma \ref{dotFn}.
We now use Section 4 in \cite{Fawa} and refer the reader to that for details: 
 the construction of \cite{Fawa} (based on an approximate solution using indicial equations and the correction using the Green's function of 
 $\Delta_g+n$ on $1$-forms on $M$) yields that there is a polyhomogeneous form $\omega=o(x^4)$, 
satisfying $(\Delta_g+n)\omega=-2\delta_{g}(\dot{g})-d\tra_g(\dot{g})$. A straightforward computation gives that if $X$ is the dual 
vector field defined by $g(X,\cdot)=\omega$, then $L_Xg=o(x^4)$.
\end{proof}

\textbf{Linearized Einstein operator}
In this section, $n$ can be either even or odd. Now that $q$ is in the kernel of the Bianchi operator $\delta_g+\demi d\tra_g$,
then we see by linearizing the Einstein equation that $q$ solves 
\begin{equation}\label{definitionLg} 
L_gq:=(\nabla^*\nabla -2\rci)q=0
\end{equation}
where $\rci$ is the operator acting on symmetric $2$ tensors defined by 
\[ (\rci q)(Y,Z)=-\sum_{i,j}\cjg R_{Y,E_i}Z,E_j\cjd q(E_i,E_j)\]
if $(E_j)_j$ is an orthonormal basis for $g$ and $R$ the Riemann tensor of $g$.
Notice that if $u$ is a function, then $L_g(ug)=((\Delta_g +2n)u)g$. Since moreover $L_g$ maps trace-free tensors to trace-free tensors,  
we deduce that  $(\Delta_g+2n)\tra_g(q)=0.$
From the work of Mazzeo \cite{MaCPDE}, the solutions of this equation are polyhomogeneous, they are combinations of functions in  $x^{\ndemi \pm s}C^\infty(\bbar{M})$,
where $s=\demi \sqrt{n(n+8)}$, and thus  since $\tra_{g}(q)\in C^\infty(\bbar{M})+x^n\log(x)C^\infty(\bbar{M})$, 
then $\tra_{g}(q)=0$, and thus 
\begin{equation}\label{qisTT}
\delta_{g}(q)=0, \quad \tra_g(q)=0. 
\end{equation}
We want to express the operator $L$ in the decomposition  $\rr_t\x N$ acting on divergence free, trace free tensors.
We will decompose such a tensor $q$ into 
\[q=udt^2+\xi\stackrel{s}\otimes dt+r\] 
where $u$ is a function, $\xi\in \Lambda^1(N)$ is a one form on $N$ and $r\in S^2(N)$ a symmetric tensor on $N$. Here $\stackrel{s}\otimes$ denote
the symmetric tensor product. The following Lemma is proved by Delay \cite{De}\footnote{The $t$-derivative denoted by prime in our setting is with respect to the connection of $g$ and is not exactly the same as Delay, which is why the coefficients are slightly different.}, we give a couple of details of the computations 
for the reader's convenience.
\begin{lem}\label{Lgalmostpr}
Let $g:=dt^2+f^2h_0$ on $M=\rr_t\x N$ for some compact manifold $(N,h_0)$ and $f\in C^2(\rr)$ some positive function. Then, if 
$q=udt^2+\xi\stackrel{s}\otimes dt+r$ with $\tra_g(q)=0$ and $\delta_{g}(q)=0$, we have 
\begin{align*}
L_gq={}& \Big(-u''+f^{-2}\Delta_{h_0}u -(n+4)\frac{f'}{f}u'-2[(n+1)\frac{(f')^2}{f^2}+\frac{f''}{f})]u \Big) dt^2\\
{}& +
\left(-\xi''-(n+2)\frac{f'}{f}\xi'-[(n-1)\frac{(f')^2}{f^2}+2\frac{f''}{f})]\xi+f^{-2}\Delta_{h_0}\xi
-2\frac{f'}{f}d^Nu
\right)\stackrel{s}\otimes dt\\
{}& +2(ff''-(f')^2)uh_0 -4\frac{f'}{f}\delta^*_{h_0}\xi -r''-n\frac{f'}{f}r'+2\frac{(f')^2}{f^2}\tra_{h_0}(r)h_0 + f^{-2}L_{h_0}r
\end{align*}
where $\xi':=\nabla_{\pl_t}\xi$, $\xi''=\nabla_{\pl_t}\nabla_{\pl_t}\xi$ with the same notation for $r'',r'$. Here $L_{h_0}$ is 
the linearized Einstein operator defined like \eqref{definitionLg} but on $N$ with the metric $h_0$.
\end{lem}
\begin{proof}
First, since $\tra_g(q)=0$, we have 
\begin{equation}\label{tracenulle}
u=-f^{-2}\tra_{h_0}(r).
\end{equation}
Let $T:=\pl_t$, let $v$ be some ($t$-independent) vector field on $N$, and set $V:=f^{-1}v$.
From \eqref{natp} we deduce that if $A=fa$ with $a\in \Lambda_1(N)$ independent of $t$, 
\begin{align*}
\nabla dt= f'f h_0,&& \nabla_T A=0,&& \nabla_VA=\nabla^N_VA-f^{-1}f'A(V)dt.
\end{align*}
We also have that for any $q\in S^2(M)$
\[ \nabla^*(dt\otimes q)=-n\frac{f'}{f}q-\nabla_Tq.\]
By direct computation we also obtain the formula for the divergence
\begin{equation}\label{div=0}\begin{split}
\delta_{g}(q)={}&\left(-u'-n\frac{f'}{f}u+f^{-2}\delta_{h_0}(\xi)+\frac{f'}{f^3}\tra_{h_0}(r)\right)dt \\
{}&-(\nabla_T\xi+(n+1)\frac{f'}{f}\xi) +f^{-2}\delta_{h_0}(r).
\end{split}\end{equation}
From \eqref{tracenulle} and \eqref{div=0}, since $q$ is divergence-free we obtain 
\begin{align}\label{trdiv=0} 
u'=-(n+1)\frac{f'}{f}u+f^{-2}\delta_{h_0}(\xi),&& \nabla_T\xi+(n+1)\frac{f'}{f}\xi=f^{-2}\delta_{h_0}(r). 
\end{align}
Let $y_j$ be Riemannian normal coordinates at $p\in N$. Then $e_j:=\pl_{y_j}$ are parallel and orthonormal at $p$: $\nabla^N_{e_i}e_j=0$  and $h_0(e_i,e_j)=\delta_{ij}$.
Set $E_i=f^{-1}e_i$. At the point $(t,p)$ for all $t\in\rr$ we have 
\[\nabla^*\nabla q=-\nabla_T\nabla_T q-n\frac{f'}{f}\nabla_Tq-\sum_{j=1}^n\nabla_{E_i}\nabla_{E_i}q.\]
Using this we compute  for $q_1=u\, dt^2$ 
\[ \nabla^*\nabla (u dt^2)=\Big(-u''+f^{-2}\Delta_Nu -n\frac{f'}{f}u'+2n\frac{(f')^2}{f^2}u\Big) dt^2-2\frac{f'}{f}(d^Nu\stackrel{s}\otimes dt)-2(f')^2uh_0.
\]
For $q_2=\xi\stackrel{s}\otimes dt$, we get
\begin{align*}
\nabla^*\nabla (\xi\stackrel{s}\otimes dt)={}&
\left(-\xi''-n\frac{f'}{f}\xi'+(n+3)\frac{(f')^2}{f^2}\xi+f^{-2}\Delta_{h_0}\xi\right)\stackrel{s}\otimes dt\\
{}&-4\frac{f'}{f^3}\delta_{h_0}(\xi)dt^2-4\frac{f'}{f}\delta^*_{h_0}\xi.
\end{align*}
Finally for the tangential part $r$, we get
\begin{align*}
\nabla^*\nabla r={}& -2f^{-2}\tr_{h_0}(r)\frac{(f')^2}{f^2}dt^2-2\frac{f'}{f^3}\delta_{h_0}r\stackrel{s}\otimes dt
-r''-n\frac{f'}{f}r'+2\frac{(f')^2}{f^2}r+ f^{-2}\Delta_{h_0}r .
\end{align*}
In conclusion, using \eqref{tracenulle} and \eqref{trdiv=0} to substitute for $\tr_{h_0}(r),\delta_{h_0}\xi$ and $\delta_{h_0} r$ we get
\begin{align*}
\nabla^*\nabla q={}& \Big(-u''+f^{-2}\Delta_Nu -(n+4)\frac{f'}{f}u'-2(n+1)\frac{(f')^2}{f^2}u \Big) dt^2\\
{}& +
\left(-\xi''-(n+2)\frac{f'}{f}\xi'-(n-1)\frac{(f')^2}{f^2}\xi+f^{-2}\Delta_{h_0}\xi
-2\frac{f'}{f}d^Nu
\right)\stackrel{s}\otimes dt\\
{}& -2(f')^2uh_0 -4\frac{f'}{f}\delta^*_{h_0}\xi -r''-n\frac{f'}{f}r'+2\frac{(f')^2}{f^2}r + f^{-2}\Delta_{h_0}r
\end{align*}
On the other hand, from \eqref{eccurvpt} and \eqref{tracenulle} we get
\[ (\rci q)=\frac{f''}{f}(u\, dt^2+\xi\stackrel{s}\otimes dt) + f^{-2}(\rci_{h_0} r)-ff''uh_0+\frac{(f')^2}{f^2}(r-\tra_{h_0}(r)h_0).\]
Combining this with the formula for $\nabla^*\nabla$, the Lemma is proved.
\end{proof}

\subsection{Computation of $q$ for $n$ odd or $n=4$}

We start by showing that the $dt^2$ and $\xi\stackrel{s}\otimes dt$ components of $q$ vanish identically. In the following
 Lemma, $n$ can be either odd or even.
\begin{lem}\label{u=q=0}
Assume that $f(t)=\cosh(t)$. Let $q=\dot{g}+L_Xg=u dt^2+\xi\stackrel{s}\otimes dt+r$ be the trace free and 
divergence free tensor in $\ker L_g$ defined in \eqref{defofq}. Then $u=0$ and $\xi=0$.
\end{lem}
\begin{proof}
Let $(\varphi_j)_{j\in \nn}$ be an orthonormal basis of eigenvectors for the Laplacian $\Delta_{h_0}$ acting on functions on $N$, with eigenvalues $\la_j$. 
 From Lemma \ref{Lgalmostpr}, we see after writing $u=\sum_{j\in\nn}u_j =\sum_{j\in \nn}\varphi_j\cjg u,\varphi_j\cjd\varphi_j$ 
 that $u_j$ satisfies the ODE
\[ -u_j''-(n+4)\tanh(t)u_j'-\Big(2(n+1)\tanh(t)^2+2-\frac{\la_j}{\cosh(t)^2}\Big)u_j=0.\]
Setting  $u_j=f^{-\frac{n}{2}-2}v_j$, this equation can be rewritten
\[-v_j''+\Big(\frac{n(n-2)}{4}\tanh(t)^2+\frac{\la_j}{\cosh(t)^2}+\ndemi\Big)v_j=0\]
and since $\lambda_j\geq 0$, this equation has no solution in $L^2(\rr,dt)$ because the corresponding operator is strictly positive. By standard ODE theory (e.g this equation is also a hypergeometric equation after setting $\sinh(t)^2=z$), the solutions are linear combinations of two independent functions $F_1,F_2$ 
such that $F_1(t)\sim_{t\to +\infty} e^{\alpha_+ t}$ and $F_2(t)\sim_{t\to +\infty}e^{\alpha_- t}$ with $\alpha_\pm=\pm\ndemi$ the roots of the polynomial $-\alpha^2+\frac{n^2}{4}$. 
Since $\delta_{h_0}(\dot{h}_0)=0$, we have $|\delta_{h_0}(r)|_{g}=\mc{O}(1)$ and thus by the second equation of 
\eqref{trdiv=0}, we deduce that $|\xi|_{g}=\mc{O}(x^2)$, which implies that $\delta_{h_0}(\xi)=\mc{O}(x)$ and by the first equation 
of \eqref{trdiv=0} we get $u=\mc{O}(x^3)=\mc{O}(e^{-3|t|})$ (here recall that $x=2e^{-|t|}$ for large $|t|$). Therefore $v_j$ is a constant times $F_2$ and thus of order $\mc{O}(e^{-\ndemi t})$ when $|t|\to +\infty$. This shows that $v_j=\mc{O}(e^{-\ndemi |t|})\in L^2(\rr, dt)$, thus $v_j=0$ and hence $u=0$. 

Writing the mixed component of $L_g(q)$ to be $0$, using $u=0$ and decomposing $\xi=\sum_{j}\xi_jf\psi_j$ where $(\psi_j)_j$ is an 
$L^2(N,{\rm dvol}_{h_0})$ orthonormal basis of eigenvectors of $\Delta_{h_0}$ on $1$-forms with eigenvalues $\alpha_j$,  we  get from  Lemma \ref{Lgalmostpr} and $\nabla_{\pl_t}(f\psi_j)=0$
\[-\pl_t^2\xi_j-(n+2)\tanh(t)\pl_t\xi_j-\Big((n-1)\tanh(t)^2+2-\frac{\alpha_j}{\cosh(t)^2}\Big)\xi_j=0.\]
Setting $\xi_j=\zeta_j f^{-\ndemi-1}$, one has 
\[ -\pl_t^2\zeta_j+\Big((\frac{n^2}{4}-\ndemi+1)\tanh(t)^2+\frac{\alpha_j}{\cosh(t)^2}+\ndemi-1\Big)\zeta_j=0.\]
One one hand, again by positivity, this equation has no solutions in $L^2(\rr,dt)$. On the other hand, we have seen above that $|\xi|_{g}=\mc{O}(x^2)=\mc{O}(e^{-2|t|})$ hence $\zeta_j=\mc{O}(e^{|t|})$. Since the indicial roots in the above equation are $\pm 2$, $\zeta_j$ must be of order $\mc{O}(e^{-2|t|})$ which is clearly in $L^2(\rr,dt)$, so actually we deduce $\xi_j=0$.
\end{proof}

We are going now to compute the coefficient of $x^{n-2}$ in the expansion of $r$. Recall that we chose $x=2e^{-|t|}$ for $t\not=0$.
\begin{prop}\label{prophessien}
Let $r=q$ be the TT tensor in $\ker L_g$ 
defined in \eqref{defofq} under the assumption \eqref{assumdiv0}. Let $r_0^\pm, r_{n}^\pm, r_{n,1}^\pm$ be the tensors on $N$ so that  as $t\to \pm \infty$ 
\[r=x^{-2}(r^\pm_0+\sum_{j =1}^{n}r^\pm_{j}x^{j}+r^\pm_{n,1}x^{n}\log(x)+o(x^{n}))\]
Let $L_{h_0}=\Delta_{h_0}-2\rci_{h_0}$ be the linearized Einstein operator on $(N,h_0)$. For every $j$ denote by $r_j^0, r_j^1$ the even, respectively the odd component 
of the pair $r_j=(r_j^+,r_j^-)$ with respect to $t\mapsto -t$.
\begin{enumerate}
\item When $n$ is even, $r_{n,1}^\pm$ is given by a differential (hence, local) operator of order $n$ in terms of $r_0$:
\begin{equation}\label{rn1}
r_{n,1}=\frac{(-1)^{\ndemi+1}2^{1-n}}{\frac{n}{2}!(\frac{n}{2}-1)!}\prod_{j=0}^{\ndemi-1}(L_{h_0}-j(n-1-j))r_0.
\end{equation}
\item If $n$ is odd, then $r_{n,1}^\pm=0$, and for $\varepsilon\in\{0,1\}$, $r_{n}^\varepsilon$ are given by 
\begin{equation}\label{formulaforrn} 
r_n^\varepsilon= \cN_\varepsilon(r_0^\varepsilon)
\end{equation}
where $\cN_\epsilon$ is the pseudodifferential operator of order $n$ 
\[\mc{N}_\varepsilon=2^{-n}\frac{\Gamma(-\frac{n}{2})}{\Gamma(\frac{n}{2})}\mc{F}_\varepsilon(\sqrt{L_{h_0}-(n-1)^2/4})\] 
for $\sqrt{\cdot } \, : \, \rr\to \rr^+\cup i\rr^-$ being the square root function, and $\mc{F}_\varepsilon$ defined by
\[
\mc{F}_\varepsilon(u):= u (\tanh(\tfrac{\pi}{2}u))^{(-1)^{\frac{n-1}{2}-\varepsilon}}
\prod_{\ell=1}^{\frac{n-1}{2}}(u^2+\ell^2).
\]
\item For $n=4$ and $\varepsilon\in\{0,1\}$,
$r_{4}^\varepsilon$ are given by $r_{4}^\varepsilon=\cG_\varepsilon(\sqrt{L_{h_0}-\tfrac{9}{4}})r_0^\varepsilon$ with 
\[\begin{split}
\cG_\varepsilon(u):=-\frac{1}{32}  &\Big[
\Big(c_0-(-1)^\varepsilon\pi\frac{1-{\rm Im}(\sinh(\pi u))}{\cosh(\pi u)}+2{\rm Re}\Psi(\tfrac{5}{2}-iu)\Big)
(u^2+\tfrac{1}{4})(u^2+\tfrac{9}{4})\\
& -2{\rm Im}(u)(2u^2+\tfrac{5}{2})+(u^2+\tfrac{1}{4})^2\Big]
\end{split}\]
where $\Psi(z)=\Gamma'(z)/\Gamma(z)$ is the digamma function, and $c_0:=-\tfrac{5}{2}+2\gamma-2\ln(2)$.
\end{enumerate}
\end{prop}
\begin{proof}
Setting  $r=sf^{-\ndemi+2}$ and $s=\sum_{j\in\nn} s_j\phi_j$ where $\phi_j$ is an $L^2(N,{\rm dvol}_{h_0})$ orthonormal basis of eigenvectors for $L_{h_0}$ with eigenvalues $\gamma_j$, then since $\nabla_{\pl_t}(f^{2}\phi_j)=0$
\begin{align*}
r'=f^{-\ndemi+2}\sum_{j}(\pl_ts_j-\ndemi \frac{f'}{f}s_j)\phi_j,&& r''=f^{-\ndemi+2}\sum_{j}((\pl_t-\ndemi \frac{f'}{f})^2 s_j)\phi_j.\end{align*}
We then have from Lemma \ref{Lgalmostpr} and Lemma \ref{u=q=0} that $s_j$ satisfies
the equation
\begin{equation}\label{eqpoursj}
-\pl_t^2s_j+\Big(z^2-\frac{\nu_j(\nu_j+1)}{\cosh^2(t)}\Big)s_j=0,\\
\end{equation}
with $\nu_j\in (-\tfrac12+i\rr_+)\cup [-\tfrac12,\infty)$, $\nu_j(\nu_j+1)=\tfrac{n(n-2)}{4}-\gamma_j$, and $z=\tfrac{n}{2}$.
Let us consider more generally this equation for $z\in\rr$ near $n/2$.
From \cite[Appendix]{GZ}, it has two independent solutions on $\rr$, one odd and one even in $t$:
\begin{align*}
E_1(t)={}&\sinh(t)\cosh(t)^{1+\nu_j}F_1(\tfrac{\nu_j+z+2}{2},\tfrac{\nu_j-z+2}{2},\tfrac{3}{2};-\sinh(t)^2)\\
E_0(t)={}&\cosh(t)^{1+\nu_j}F_1(\tfrac{\nu_j+z+1}{2},\tfrac{\nu_j-z+1}{2},\tfrac{1}{2};-\sinh(t)^2)
\end{align*}
where $F_1(a,b,c; \tau)$ is the hypergeometric function. The solution $E_1$ corresponds to taking $r_0^-=-r_0^+$ while $E_0$ corresponds to $r_0^+=r_0^-$. Using the identity
\[\begin{split}
F_1(a,b,c;\tau)={} &\frac{\Gamma(c)\Gamma(b-a)}{\Gamma(c-a)\Gamma(b)}(-\tau)^{-a}F_1(a,a+1-c,a+1-b;\tfrac{1}{\tau})\\
&+ \frac{\Gamma(c)\Gamma(a-b)}{\Gamma(c-b)\Gamma(a)}(-\tau)^{-b}F_1(b,b+1-c,b+1-a;\tfrac{1}{\tau})
\end{split}\]
and $F_1(a,b,c;0)=1$, we see that, for $x=2e^{-|t|}$, there exist meromorphic coefficients $a_{2k}(z)$
 such that 
\begin{equation}\label{expansionz}
\begin{split}
\lefteqn{E_1(t)\frac{\Gamma(\tfrac{\nu_j+z+2}{2})\Gamma(\tfrac{-\nu_j+z+1}{2})}{\Gamma(\tfrac{3}{2})\Gamma(z)}}&\\
={}&{\rm sign}(t) x^{-z}\Big(1+\sum_{k=1}^{\ndemi}x^{2k}a_{2k}(z,\nu_j)+x^{2z}\frac{\Gamma(-z)\Gamma(\tfrac{-\nu_j+z+1}{2})\Gamma(\tfrac{\nu_j+z+2}{2})}
{\Gamma(z)\Gamma(\tfrac{-\nu_j-z+1}{2})\Gamma(\tfrac{\nu_j-z+2}{2})}+o(x^{2z})\Big).
\end{split}
\end{equation}
We shall use the same type of arguments as in the work of Graham-Zworski \cite{GrZw};
the coefficients $a_k(z,\nu_j)$ are regular near $z=n/2$ except for $a_{n}(z,\nu_j)$ when $n$ is even, which has a first order pole at $z=n/2$. The coefficient \begin{equation}\label{formuleS_j(z)}
S^1_{j}(z)= \frac{\Gamma(-z)\Gamma(\tfrac{-\nu_j+z+1}{2})\Gamma(\tfrac{\nu_j+z+2}{2})}
{\Gamma(z)\Gamma(\tfrac{-\nu_j-z+1}{2})\Gamma(\tfrac{\nu_j-z+2}{2})}.\end{equation}
of $x^{z}$ in \eqref{expansionz} also has a pole in that case, and its residue is $-{\rm Res}_{\ndemi}a_n(z,\nu_j)$.
Notice that $S_j^1$ is the action of the scattering operator of $L_{h_0}$ at $z\in\cc$ on the odd pair of tensors $(\phi_j,-\phi_j)$.
Using the formula $\Gamma(s)\Gamma(1-s)=\pi/\sin(\pi s)$ and $\Gamma(s)\Gamma(s+\demi)=2^{1-2s}\sqrt{\pi}\Gamma(2s)$ 
we rewrite
\begin{equation}\label{formular}
S^1_{j}(z)=2^{-2z}\frac{\Gamma(-z)}{\Gamma(z)} \frac{\sin(\tfrac{\pi}{2}(\nu_j-z))}
{\sin(\tfrac{\pi}{2}(\nu_j+z))}\frac{\Gamma(z-\nu_j)}{\Gamma(-z-\nu_j)}.\end{equation}
When $n$ is even, the right-hand side of \eqref{expansionz} at $z=n/2$
has the asymptotic expansion as $t\to \pm \infty$
\begin{equation}\label{sjz=n/2}
\begin{split}
\pm x^{-\ndemi}\Big( &1+\sum_{k=1}^{\ndemi-1}x^{2k}a_{2k}(\tfrac{n}{2},\nu_j)+2({\rm Res}_{\ndemi}S_j^1(z))x^{n}\log(x)\\
 & +{\rm FP}_{\ndemi}(a_n(z,\nu_j)+S_j^1(z))x^n+o(x^n)\Big)\end{split}\end{equation}
where ${\rm FP}$ denotes finite part.
From \eqref{formular}, we deduce the formula \eqref{rn1} by taking the residue at $z=n/2$ and the fact that 
$L_{h_0}\phi_j=(\tfrac{n(n-2)}{4}-\nu_j(\nu_j+1))\phi_j$.

If now $n$ is odd, we can take the limit $z\to \ndemi$ in \eqref{expansionz} and each coefficient is smooth at $z=n/2$ ($a_n$ does not exist in this parity) 
writing $\nu_j=-\demi+i\alpha_j$ with $\alpha_j=\sqrt{\gamma_j-\tfrac{(n-1)^2}{4}}$ (the convention is 
$i\alpha_j\in \rr^+$ if $\gamma_j<\tfrac{(n-1)^2}{4}$), we obtain directly that the coefficient of $x^{\ndemi}$ in \eqref{expansionz} is 
\begin{equation}\label{formularneven}
2^{-n}\frac{\Gamma(-\tfrac{n}{2})}{\Gamma(\tfrac{n}{2})}\alpha_j \Big(\frac{\cosh(\tfrac{\pi}{2}\alpha_j)}{\sinh(\tfrac{\pi}{2}\alpha_j)}\Big)^{(-1)^{\frac{n-1}{2}}}
\prod_{\ell=1}^{\frac{n-1}{2}}(\alpha_j^2+\ell^2)
\end{equation}
which implies formula \eqref{formulaforrn} for the odd component $r_n^1$.
Now we can do the same analysis with the even solution $E_{0}(t)$; we do not give details of the calculations which are very similar to the above, notice however that by locality, the formula \eqref{formularneven} for the logarithmic term cannot change at all. We eventually obtain in this case
\begin{equation}\label{sjeven}
S_j^0(z)=\frac{\Gamma(-z)\Gamma(\tfrac{-\nu_j+z}{2})\Gamma(\tfrac{\nu_j+z+1}{2})}
{\Gamma(z)\Gamma(\tfrac{-\nu_j-z}{2})\Gamma(\tfrac{\nu_j-z+1}{2})}
= 2^{-2z}\frac{\Gamma(-z)}{\Gamma(z)}\frac{\Gamma(-\nu_j+z)}{\Gamma(-\nu_j-z)}\frac{\cos(\tfrac{\pi}{2}(\nu_j-z))}{\cos(\tfrac{\pi}{2}(\nu_j+z))}
\end{equation}
which implies
\begin{equation}\label{formula2rneven}
r_n^0=2^{-n}\frac{\Gamma(-\tfrac{n}{2})}{\Gamma(\tfrac{n}{2})}\sum_j \Big(\alpha_j \Big(\frac{\sinh(\tfrac{\pi}{2}\alpha_j)}{\cosh(\tfrac{\pi}{2}\alpha_j)}\Big)^{(-1)^{\frac{n-1}{2}}}
\prod_{\ell=1}^{\frac{n-1}{2}}(\alpha_j^2+\ell^2)\Big) \cjg r_0^0,\phi_j\cjd \phi_j .
\end{equation}

Let us now consider the case $n=4$. Let us first compute $a_{4}(z,\nu_j)$ in  \eqref{expansionz}.
We rewrite equation \eqref{eqpoursj} in terms of $x=2e^{-|t|}$ for $|t|>1$: 
\[ -(x\pl_x)^2s_j+(z^2-x^2\nu_j(\nu_j+1)+\tfrac{ \nu_j(\nu_j+1)}{2}x^4+\mc{O}(x^6))s_j=0.\]
Solving this equation as a series in $x$, the coefficients $a_{2k}(z,\nu_j)$ are uniquely determined and we obtain for $t\to \infty$ the asymptotic expansion for the even, respectively odd solution
\[s_j^\varepsilon(z)= x^{-z}\Big(1+ \tfrac{\nu_j(\nu_j+1)}{4(z-1)}x^2+ \tfrac{1}{8(z-2)}(\tfrac{\nu_j^2(\nu_j+1)^2}{4(z-1)}-\tfrac{\nu_j(\nu_j+1)}{2})x^4\Big)+x^zS^\varepsilon_j(z)+o(x^4)\]
and thus $a_4(z,\nu_j)=\tfrac{1}{8(z-2)}(\tfrac{\nu_j^2(\nu_j+1)^2}{4(z-1)}-\tfrac{\nu_j(\nu_j+1)}{2})$. 
By \eqref{sjz=n/2},
\begin{equation}\label{rnpf}
 r_n^\varepsilon =\sum_{j} \Big( {\rm FP}_{2} S_j^\varepsilon (z)-\tfrac{\nu_j^2(\nu_j+1)^2}{32}\Big) \cjg r_0^\varepsilon , \phi_j\cjd \phi_j.
 \end{equation}
We now compute ${\rm FP}_{2} S_j^1 (z)$. We assume that $L_{h_0}\geq 2$ 
so that we can write $\nu_j=-\demi+i\alpha_j$ with $\alpha_j=\sqrt{\gamma_j-\tfrac{9}{4}}\geq 0$ if $\gamma_j\geq \tfrac{9}{4}$ and 
$i\alpha_j\in [0,\demi]$ if $\gamma_j\leq 9/4$.
We use formula \eqref{formular} for $S^1_j(z)$, then for $\nu_j\in\rr$ we see that $S^1_j(z)\in\rr$, but we also notice that \eqref{formuleS_j(z)} implies that
\[S^1_{j}(z)= \frac{\Gamma(-z)|\Gamma(\tfrac{z+\frac{3}{2}+i\alpha_j}{2})|^2}
{\Gamma(z)|\Gamma(\tfrac{-z+\frac{3}{2}+i\alpha_j}{2})|^2} \in \rr \textrm{ if }z\in\rr, \alpha_j\in\rr^+.\]
Let $\gamma=-\Gamma'(1)$ be the Euler constant, then for $z$ close to $2$
\begin{equation}\label{gammataylor}
2^{-2z}\frac{\Gamma(-z)}{\Gamma(z)}=-\tfrac{1}{32}[ (z-2)^{-1}+2\gamma-\tfrac52-2\ln(2)]+\mc{O}((z-2)). 
\end{equation}
We write
\[
\frac{\Gamma(z-\nu_j)}{\Gamma(-z-\nu_j)}=\frac{\Gamma(z-\nu_j)}{\Gamma(-z+4-\nu_j)}(z+\nu_j)(z+\nu_j-1)(z+\nu_j-2)(z+\nu_j-3)
\]
and consider its Taylor expansion at $z=2$:
\begin{equation}\label{taylor}
\begin{split}
\frac{\Gamma(z-\nu_j)}{\Gamma(-z-\nu_j)}={}&(\alpha_j^2+\tfrac{1}{4})(\alpha^2_j+\tfrac{9}{4})+ 
(z-2)\Big[2\Psi(2-\nu_j)(\alpha_j^2+\tfrac{1}{4})(\alpha^2_j+\tfrac{9}{4})\\
&  -2i\alpha_j (2\alpha_j^2+\tfrac{5}{2}))\Big]+\mc{O}((z-2)^2).
\end{split}
\end{equation}
where 
$\Psi(z)=\pl_z\Gamma(z)/\Gamma(z)$ is the digamma function. Finally we expand 
\begin{equation}\label{expandsin}
\frac{\sin(\tfrac{\pi}{2}(\nu_j-z))}{\sin(\tfrac{\pi}{2}(\nu_j+z))}=1- (z-2)\pi \frac{\cos(\tfrac{\pi}{2}\nu_j)}{\sin(\tfrac{\pi}{2}\nu_j)}+\mc{O}((z-2)^2).
\end{equation}
Using \eqref{rnpf} and combining  \eqref{gammataylor}, \eqref{taylor}, \eqref{expandsin} we get when $\nu_j=-\demi+\la_j$ with $\la_j:=i\alpha_j\in \rr^+$ 
\[\begin{split}
\cjg r_4^\pm,\phi_j\cjd= -\frac{1}{32} &\Big[\Big(c_0-\pi\frac{\cos(\tfrac{\pi}{2}(\la_j-\tfrac{1}{2}))}{\sin(\tfrac{\pi}{2}(\la_j-\tfrac{1}{2}))}+2\Psi(\tfrac{5}{2}-\la_j)\Big)
(-\la_j^2+\tfrac{1}{4})(-\la_j^2+\tfrac{9}{4})\\
& +2\la_j(2\la_j^2-\tfrac{5}{2})+(-\la_j^2+\tfrac{1}{4})^2\Big]
\end{split}\]
where $c_0:=2\gamma-\tfrac{5}{2}-2\ln(2)$. When $\nu_j=-\demi +i\alpha_j$ with $\alpha_j\in\rr$ we get (using that $S_j^1(z)$ is real)
\[\begin{split}
\cjg r_4^\pm,\phi_j\cjd= -\frac{1}{32} \Big[
\Big(c_0+\frac{\pi}{\cosh(\pi\alpha_j)}+2{\rm Re}(\Psi(\tfrac{5}{2}-i\alpha_j)\Big)
(\alpha_j^2+\tfrac{1}{4})(\alpha_j^2+\tfrac{9}{4})+(\alpha_j^2+\tfrac{1}{4})^2\Big]
\end{split}\]
This gives the desired result when $r_0^+=-r_0^-$ by using \eqref{rnpf}.
When $r_0^+=r_0^-$, we consider the expansion of the even solution $E_{0}(t)$, this is  a similar computation to what we did for $E_1$, 
but using formula
\eqref{sjeven} instead of \eqref{formuleS_j(z)}), and 
\[\frac{\cos(\tfrac{\pi}{2}(\nu_j-z))}{\cos(\tfrac{\pi}{2}(\nu_j+z))}=1+ (z-2)\pi \frac{\sin(\tfrac{\pi}{2}\nu_j)}{\cos(\tfrac{\pi}{2}\nu_j)}+\mc{O}((z-2)^2).
\]
instead of \eqref{expandsin}. We find for $\alpha_j\geq 0$
\[\begin{split}
\cjg r_4^0,\phi_j\cjd=  -\frac{1}{32}  \Big[
\Big(c_0-\frac{\pi}{\cosh(\pi\alpha_j)}+2{\rm Re}\Psi(\tfrac{5}{2}-i\alpha_j)\Big)
(\alpha_j^2+\tfrac{1}{4})(\alpha_j^2+\tfrac{9}{4})+(\alpha_j^2+\tfrac{1}{4})^2\Big]
\end{split}\]
and for $\la_j=i\alpha_j\in\rr^+$ 
\[\begin{split}
\cjg r_4^0,\phi_j\cjd= -\frac{1}{32}& \Big[\Big(c_0-\pi\tan(\tfrac{\pi}{2}(\tfrac{1}{2}-\la_j))+2\Psi(\tfrac{5}{2}-\la_j)\Big)
(-\la_j^2+\tfrac{1}{4})(-\la_j^2+\tfrac{9}{4})\\
&+2\la_j(2\la_j^2-\tfrac{5}{2})+(-\la_j^2+\tfrac{1}{4})^2\Big].
\end{split}\]
This finishes the proof.
\end{proof}

As a first corollary, we recover a formula proved recently by Matsumoto \cite{Mat} for the Hessian of the functional 
$h_0\mapsto \int_{N}v_n(h_0){\rm dvol}_{h_0}$ defined on the space $\mc{C}(N)$ of conformal structures.  
\begin{cor}\label{hessienL}
Let $n$ be even, let $h_0$ satisfies ${\rm Ric}_{h_0}=-(n-1)h_0$ on $N$, and let $L_{h_0}=\nabla^*\nabla-2\rci_{h_0}$ be the linearized Einstein operator at $h_0$. 
Then the obstruction tensor $k$ linearized at $h_0$ and acting on 
divergence-free/trace free tensors $\dot{h}_0$ is given by 
\[ Dk_{h_0}.\dot{h}_0=\frac{(-1)^{\ndemi+1}2^{1-n}}{\frac{n}{2}!(\frac{n}{2}-1)!}\prod_{j=0}^{\ndemi-1}(L_{h_0}-j(n-1-j))\dot{h}_0.\] 
\end{cor}
\begin{proof} If $g^s$ is a deformation of Einstein metrics as before and $\dot{g}=\pl_sg^s|_{s=0}$, then  
the first log term in the expansion of $\dot{g}$ is $\dot{k}x^{n-1}\log(x)$ where $\dot{k}$ is the variation of the obstruction 
tensor $k^s$ of $g^s$. We modify $\dot{g}$ by $L_Xg$ as in \eqref{defofq}, and we apply \cite[Prop. 4.5]{Fawa}\footnote{The proof in \cite{Fawa} is technically for $n$ odd, but the same arguments apply, once we have noticed 
that $\delta_g(\dot{g})+\tfrac{1}{2}d\tra(\dot{g})$ has no log coefficient before $x^{n+1}\log(x)$ when measured with respect to $g$, this is easy to check.} to deduce that
$L_Xg$ has no log term before $x^{n}\log(x)$, thus $q=L_Xg+\dot{g}$ has first log term given by $\dot{k}x^{n-1}\log(x)$ in its expansion.
Now it remains to use formula \eqref{rn1} and this gives $\dot{k}$ in terms of $\dot{h}_0$. 
\end{proof}
Our second corollary is
\begin{cor}
Let $g^s$ be a family of AHE metrics such that $g^0=g$ is the Fuchsian-Einstein metric.
\begin{enumerate} 
\item Let $n$ be odd and $L_{h_0}=\Delta_{h_0}-2\rci_{h_0}$ be the linearized Einstein operator. 
Then there exists $C>0$ such that for all $\dot{h}_0$ satisfying 
$\delta_{h_0}(\dot{h}_0)=0$, ${\rm Tr}_{h_0}(\dot{h}_0)=0$ and $\indic_{[0,\frac{(n-1)^2}{4}]}(L_{h_0})\dot{h}_0=0$
\[ (-1)^{\frac{n+1}{2}}\pl_s^2 {\rm Vol}_R(M,g^s)|_{s=0}\geq C|\dot{h}_0|_{H^{\ndemi}(N)}^2.\]
\item Let $n=4$ and let $h_0^s$ be a smooth family of conformal representatives of the conformal infinity satisfying
 $v_n(h_0^s)=\int_{N}v_n(h_0^s){\rm dvol}_{h_0^s}$. Assuming that $L_{h_0}-2 > 0$ on the subspace of trace-free/divergence free tensors, there exists $C>0$ such that for all $\dot{h}_0$
 satisfying  $\delta_{h_0}(\dot{h}_0)=0$ and ${\rm Tr}_{h_0}(\dot{h}_0)=0$
\[ \pl_s^2 {\rm Vol}_R(M,g^s;h_0^s)|_{s=0}\geq C|\dot{h}_0|_{H^2(N)}^2.\]
\end{enumerate}
\end{cor}
\begin{rem*}
By using Weitzenb\"ock type formula, one obtains that the assumption $L_{h_0}-2\geq 0$ is satisfied for hyperbolic metrics $h_0$, see for example the proof of \cite[Th 12.67]{besse}.
\end{rem*}
\begin{proof} The case $n$ odd is a direct consequence of Proposition \ref{prophessien}. For the case $n=4$, we 
use \eqref{dotgn}, Lemma \ref{implytrace0}, Lemma \ref{dotFn} and Lemma \ref{u=q=0} to deduce that 
\[4\pl_s^2 {\rm Vol}_R(M,g^s;h_0^s)|_{s=0}=-\cjg r_4,\dot{h}_0\cjd_{L^2}+(\tfrac{1}{2}v_4-\tfrac{1}{8})|\dot{h}_0|^2_{L^2}-\tfrac12 \cjg r_{4,1},\dot{h}_0\cjd_{L^2}\]
where $r_4$ and $r_{4,1}$ are given in Proposition \ref{prophessien}. By Lemma \ref{vncst}, $v_4=3/8$ and so 
\[\begin{split}
2^7\pl_s^2 {\rm Vol}_R(M,g^s;h_0^s)|_{s=0}={}&\cjg 
\mc{H}_1(\sqrt{L_{h_0}-\tfrac{9}{4}})\dot{h}^1_0,\dot{h}_0^1\cjd_{L^2} +\cjg 
\mc{H}_0(\sqrt{L_{h_0}-\tfrac{9}{4}})\dot{h}^0_0,\dot{h}_0^0\cjd_{L^2}
\end{split}\]
where $\mc{H}_\varepsilon$, $\varepsilon\in\{0,1\}$ are the functions defined by
\begin{equation}\label{HoHe}
\begin{split}
\mc{H}_\varepsilon(u):={}&
\Big(c_0+-\frac{(-1)^{\varepsilon}\pi}{\cosh(\pi u)}+2{\rm Re}(\Psi(\tfrac{5}{2}-iu)\Big)
(u^2+\tfrac{1}{4})(u^2+\tfrac{9}{4})+(u^2+\tfrac{1}{4})^2+2
\end{split}\end{equation}
for $u\geq 0$, and to
\begin{equation}\label{HoHeiu}
\begin{split}
\mc{H}_\varepsilon(-i u):={}& \Big(c_0+1-(-1)^\varepsilon\pi (\tan(\tfrac{\pi}{2}(\tfrac{1}{2}-u)))^{(-1)^\varepsilon}+2\Psi(\tfrac{5}{2}-u)\Big)
(-u^2+\tfrac{1}{4})(-u^2+\tfrac{9}{4})\\
& +2u(2u^2-\tfrac{5}{2})+(-u^2+\tfrac{1}{4})^2+{2}
\end{split}\end{equation}
for $u\geq 0$. Let us show that the functions in \eqref{HoHe} are positive for $u\geq 0$ (numerically this follows from Figure \ref{fig1} but we give a 
formal proof). 
\begin{figure}[h]
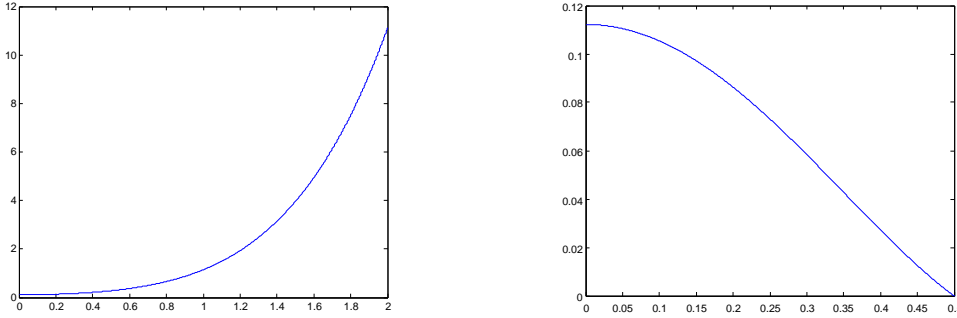

\includegraphics[scale=0.4, trim = 30mm 80mm 20mm 80mm, clip]{courbe.pdf} \qquad
\includegraphics[scale=0.4, trim = 30mm 80mm 20mm 80mm, clip]{courbe_imaginaire.pdf} 
\caption{\label{fig1} Left: the graph of the function $\tilde{\mc{H}}(u)$ bounding $\cH_0(u)$ from below, where $\tilde{\mc{H}}(u):=\Big(c_0+1-\frac{\pi}{\cosh(\pi u)}+2\Psi(\tfrac{5}{2})\Big) (u^2+\tfrac{1}{4})(u^2+\tfrac{9}{4})+(u^2+\tfrac{1}{4})^2+{2}$.
Right: the graph of the function $\mc{H}_0(-iu)$ for $u\in[0,1/2]$.}
\end{figure}
Since $\mc{H}_1(u)\geq \mc{H}_0(u)$ for $u\geq 0$, it suffices to show that $\mc{H}_0(u)>0$. We write 
\begin{align*}
\cH_0(u)=a(u)x(x+2) + x^2+{2},&& x:=u^2+\tfrac14,&& a(u):=c_0+1-\frac{\pi}{\cosh(\pi u)}+2{\rm Re}(\Psi(\tfrac{5}{2}-iu)).
\end{align*}
The real part of the digamma function $\Re(\Psi(\tfrac52+it))$ is increasing as a function of $|t|$:
\begin{equation}\label{Psifct}
\Re(\Psi(\sigma+it))=-\gamma+\sum_{n=0}^\infty\left( \frac{1}{n+1}-\frac{n+\sigma}{(n+\sigma)^2+t^2}\right),
\end{equation} thus 
\begin{equation}\label{Psiincrease}
{\rm Re}(\Psi(\tfrac{5}{2}-i\tfrac{u}{2}))\geq \Psi(\tfrac{5}{2})=-\gamma-2\ln 2 +\tfrac83,
\end{equation} 
and so $a(u)$ is increasing and satisfies
\begin{equation}\label{ineqa}
a(u)\geq a(0)= \tfrac{23}{6}-6 \ln 2 -\pi.
\end{equation}
For $a\in\rr$ let $P_a(x):=ax(x+2) + x^2+{2}$. Since $a(u)$ is increasing, we have $\cH_0(u)=P_{a(u)}(x)\geq P_{a(u_0)}(x)$ for all
$u\geq u_0$.  We use a bootstrap argument:
by \eqref{ineqa}, $a(0)>-3.468=:a_0$ and so $P_{a_0}(x)>0$ for $x<x_1$ with some explicitly computable
$x_1$
This means that 
$\cH_0(u)>0$ for $u<(x_1-0.25)^{1/2}=:u_1$
Using \eqref{Psiincrease}, we have a lower bound $a(u_1)>a_1$. The binomial $P_{a_1}(x)$ is positive for 
$x<x_2$, etc. This tedious algorithm stops after a finite number of steps. 
Hence $\mc{H}_0(u)>0$ for all $u>0$.

We now show that the function $\cH_0(-iu)$ from \eqref{HoHeiu} is positive for $0\leq u<1/2$.
We introduce the notation
\begin{align*}
u:=\tfrac12-v,&& a(v):= c_0+1-\pi (\tan(\tfrac{\pi}{2}v))+2\Psi(v+2)
\end{align*}
and we compute $\cH_0(-iu)= v(2-v)[a(v)(v^2-1)+v^2-4v-1 ]$. Therefore for $0<v\leq 1/2$,
$\cH_0(-iu)>0$ if and only if $-(a(v)+1)< \frac{4v}{1-v^2}$. At $v=0$ this is verified since $c_0>-4$. It is thus enough to show that for $0<v\leq 1/2$ we have
\begin{align*}
-a'(v)<2\left(\frac{1}{(1-v)^2} + \frac{1}{(1+v)^2} \right).
\end{align*}
Using \eqref{Psifct}, we have
\[\begin{split}
-a'(v)={}&\frac{\pi^2}{2\cos^2(\pi v/2)} -2 \Psi'(2+v)
= 2\sum_{k\in \zz} \frac{1}{(v+2k+1)^2}-2 \sum_{k\geq 0}\frac{1}{(v+k+2)^2}\\
={}& 2\left(\frac{1}{(1-v)^2} + \frac{1}{(1+v)^2} \right)+2\sum_{k=1}^\infty\frac{(2v-1)(4k+1)}{(v-2k-1)^2(v+2k)^2}
\end{split}\]
and this finishes the proof since $2v-1<0$.
\end{proof}

\begin{rem}
For $n$ odd, we notice that from Proposition \ref{prophessien}, the Fuchsian-Einstein metric is a saddle point for  regularized volume if $L_{h_0}-\tfrac{(n-1)^2}{4}$ has negative eigenvalues. 
\end{rem}


\providecommand{\bysame}{\leavevmode\hbox to3em{\hrulefill}\thinspace}
\providecommand{\MR}{\relax\ifhmode\unskip\space\fi MR }
\providecommand{\MRhref}[2]{%
  \href{http://www.ams.org/mathscinet-getitem?mr=#1}{#2}
}
\providecommand{\href}[2]{#2}

\end{document}